\documentclass{amsart}

\usepackage{amsfonts,amssymb,verbatim,amsmath,amsthm,latexsym,textcomp,amscd}
\usepackage{latexsym,amsfonts,amssymb,epsfig,verbatim}
\usepackage{amsmath,amsthm,amssymb,latexsym,graphics,textcomp}
\usepackage{graphicx}
\usepackage{color}
\usepackage{url}

\begin{document}

\newtheorem{theorem}{Theorem}[section]
\newtheorem{prop}[theorem]{Proposition}
\newtheorem{lemma}[theorem]{Lemma}
\newtheorem{cor}[theorem]{Corollary}
\newtheorem{defn}[theorem]{Definition}
\newtheorem{conj}[theorem]{Conjecture}
\newtheorem{claim}[theorem]{Claim}
\newtheorem{example}[theorem]{Example}
\newtheorem{rem}[theorem]{Remark}
\newtheorem{rmk}[theorem]{Remark}
\newtheorem{obs}[theorem]{Observation}

\newcommand{\map}{\rightarrow}
\newcommand{\C}{\mathcal C}
\newcommand\AAA{{\mathcal A}}
\newcommand\BB{{\mathcal B}}
\newcommand\DD{{\mathcal D}}
\newcommand\EE{{\mathcal E}}
\newcommand\FF{{\mathcal F}}
\newcommand\GG{{\mathcal G}}
\newcommand\HH{{\mathcal H}}
\newcommand\II{{\mathcal I}}
\newcommand\JJ{{\mathcal J}}
\newcommand\KK{{\mathcal K}}
\newcommand\LL{{\mathcal L}}
\newcommand\MM{{\mathcal M}}
\newcommand\NN{{\mathcal N}}
\newcommand\OO{{\mathcal O}}
\newcommand\PP{{\mathcal P}}
\newcommand\QQ{{\mathcal Q}}
\newcommand\RR{{\mathcal R}}
\newcommand\SSS{{\mathcal S}}
\newcommand\TT{{\mathcal T}}
\newcommand\UU{{\mathcal U}}
\newcommand\VV{{\mathcal V}}
\newcommand\WW{{\mathcal W}}
\newcommand\XX{{\mathcal X}}
\newcommand\YY{{\mathcal Y}}
\newcommand\ZZ{{\mathcal Z}}
\newcommand\hhat{\widehat}

\title{ A Combination Theorem for Metric Bundles}

\author{Mahan Mj}

\author{Pranab Sardar}
\address{RKM Vivekananda University, Belur Math, WB-711 202, India}
\thanks{Research of first author partially supported by a CEFIPRA Indo-French Research grant.
The second author is partly supported by a CSIR  Research Fellowship. This paper is part
of PS's PhD thesis  written under the supervision of MM} 
\date{\today}

\begin{abstract}
We introduce the notion of metric (graph) bundles
 which provide a coarse-geometric generalization of
the notion of trees of metric spaces a la Bestvina-Feighn in the special case that the inclusions of the edge spaces
into the vertex spaces are uniform coarsely surjective quasi-isometries. We prove the existence of
quasi-isometric sections in this generality.
Then we prove a combination theorem for  metric (graph) bundles  that establishes
sufficient conditions, particularly flaring, under which the metric
bundles  are hyperbolic. We use this  to give  examples of surface bundles over hyperbolic disks, whose
universal cover is Gromov-hyperbolic.
We also show that in typical situations, flaring is also a necessary condition.
\end{abstract}

\maketitle
\tableofcontents
\section{Introduction}
In this paper we introduce the notions of  metric bundles and metric graph bundles
which provide  a  purely  coarse-geometric generalization of the notion of trees of metric spaces
a la Bestvina-Feighn \cite{BF} (see Section \ref{thrhms}) in the special case that the inclusions of the edge spaces
into the vertex spaces are uniform coarsely surjective quasi-isometries. We generalize the base space from
a  tree to an arbitrary hyperbolic metric space. In \cite{farb-mosher}, Farb and Mosher introduced the notion of metric fibrations
which was used by Hamenstadt to give a combination theorem in \cite{hamenst-word}. Metric fibrations can be thought of as metric bundles (in our terminology)
equipped with a {\it foliation by totally geodesic sections} of the base space. We first prove the following Proposition which ensures the existence 
of  {\it q(uasi)-i(sometric)} sections in the general context of metric bundles, generalizing and giving a different proof of
 a result due to Mosher \cite{mosher-hypextns} in the context of exact sequences of groups (see Example $\ref{eg-mbdl}$).

\smallskip

\noindent {\bf Proposition \ref{existence-qi-section}  (Existence of qi sections):} {\it
Let $\delta,N\geq 0$ and 
suppose $p : X \rightarrow B$ is an $(f,K)$-metric graph bundle with the following properties:
\begin{enumerate}
\item Each of the fibers $F_b$, $b \in \mathcal{ V}(B)$ is a $\delta$-hyperbolic geodesic metric space with respect to the
 path metric $d_b$ induced from $X$.
\item The barycenter maps $\phi_b : \partial^3 F_b \rightarrow F_b$ are uniformly coarsely surjective, i.e. 
 $F_b$ is contained in the $N$-neighborhood of the image of $\phi_b$
for all $b\in \mathcal{ V} (B)$.
\end{enumerate}
Then there is a $K_0=K_0(f,\delta,N)$-qi section through each point of $\mathcal V(X)$.}

\smallskip

 Proposition \ref{existence-qi-section} 
 provides a context for developing a `coarse theory of bundles' and proving the following combination theorem,
which is the main theorem of this paper. \\

\smallskip

\noindent {\bf Theorem \ref{combthm}:} {\it 
Suppose $p:X\rightarrow B$ is a metric bundle (resp. metric graph bundle) such that\\
$(1)$ $B$ is a $\delta$-hyperbolic metric space.\\
$(2)$ Each of the fibers $F_b$, $b\in B$ (resp. $b\in \mathcal{ V}(B)$) is a $\delta^{'}$-hyperbolic geodesic
metric space with respect to the path metric induced from $X$. \\
$(3)$ The barycenter maps $\partial^3F_b\rightarrow F_b$, $b\in B$ (resp. $b\in \mathcal{ V}(B)$) are uniformly coarsely  surjective.\\
$(4)$ A flaring condition is satisfied.\\
Then $X$ is a hyperbolic metric space.}

\medskip

This
is a first step towards proving a combination Theorem for more general complexes of spaces
(cf. Problem 90 of \cite{kap-prob}).

Theorem \ref{combthm} generalizes Hamenstadt's combination theorem  (Corollary 3.8 of \cite{hamenst-word}) in two ways:\\
a) It removes
 the hypothesis of properness of the base space $B$ -- a hypothesis that is crucial in \cite{hamenst-word} to ensure compactness of the boundary of the base space
and hence allow the arguments in \cite{hamenst-word} to work. 
This generalization is relevant for two reasons. First, underlying trees in trees of spaces are frequently non-proper. Secondly,  curve complexes of surfaces are mostly
 non-proper metric spaces and occur as natural base spaces for metric bundles. See \cite{mahan-uct} by Leininger-Mj-Schleimer for a closely related
example.\\
b) It removes the hypothesis on existence of totally geodesic sections in \cite{hamenst-word} altogether. Proposition \ref{existence-qi-section} 
ensures the existence 
of  qi sections under mild technical assumptions. 


A word about the proof of Theorem \ref{combthm} 
ahead of time.  Proposition \ref{existence-qi-section} ensures the existence 
of  qi sections through points of $X$. We use the notion of flaring from  Bestvina-Feighn \cite{BF} and a criterion for hyperbolicity introduced by Hamenstadt in  \cite{hamenst-teich} to
 construct certain path families  and use them
to prove hyperbolicity. Another crucial ingredient is a `ladder-construction' due to the first author
 \cite{mitra-trees}, which may be regarded as an analog of the
hallways of  \cite{BF}.

Recall \cite{farb-relhyp} that for a pair
$(X, \mathcal{H})$ of a metric space $(X,d_X)$ and a family of path-connected subsets $\mathcal{H}$ of $X$, the electric space
$\mathcal{E} (X, \mathcal{H})$ is the pseudo-metric space 
  $X \sqcup_{H \in  \mathcal{H}} H \times [0,1]$ with $H \times \{0\}$ identified with $H \subset X$
and $H \times \{1\}$ equipped with the zero metric. Each $\{ h \} \times [0,1]$
is isometric to the unit interval. There is a natural inclusion map $E: X \rightarrow \mathcal{E} (X, \mathcal{H})$
which is referred to as the electrocution map. The image $E(X) $ inherits a metric called the electric metric $d_e$.

As an application of Theorem \ref{combthm} we obtain a rather plentiful supply of examples from the following
Proposition, where the base space
need not be a tree (as in all previously known examples). Let $S$ be a closed surface of genus greater than one
and  $Teich(S)$ be the Teichmuller space
of $S$.   The Teichmuller metric on  $Teich(S)$ is  denoted as  $d_T$ and $d_e$ denotes the electric metric on  $Teich(S)$ obtained by electrocuting
the $\alpha$-thin parts of $Teich(S)$ for every 
essential simple closed curve $\alpha$ on $S$. For $j: K \rightarrow (Teich(S), d_T)$ a map, let
$U(S,K)$ denote the pullback (under $j$) of the universal curve over $Teich(S)$
equipped with the natural path metric. Also, the universal cover of the  universal curve over $Teich(S)$
is a hyperbolic plane bundle over $Teich(S)$. Let  $\widetilde{U(S,K)}$ denote the pullback to $K$ of this hyperbolic plane bundle.

\smallskip

\noindent {\bf Proposition \ref{eg}:} {\it 
Let $(K,d_K)$ be a hyperbolic metric space satisfying the following:\\
There exists $C \geq 0$ such that
for any two points $u, v\in K$, there exists a bi-infinite $C$-quasigeodesic $\gamma \subset K$ with $d_K(u, \gamma ) \leq C$
and $d_K(v, \gamma ) \leq C$. \\
Let $j: K \rightarrow (Teich(S), d_T)$ be a quasi-isometric embedding
such that $E\circ j: K \rightarrow (Teich(S), d_e)$ is also a quasi-isometric embedding.  Then
$\widetilde{U(S,K)}$ is a hyperbolic metric space. }

\smallskip

It is an open question (cf. \cite{kl} \cite{farb-mosher}) to find purely pseudo Anosov surface groups $Q$ ($=\pi_1(\Sigma)$,
say) in $MCG(S)$.
This is equivalent to constructing surface bundles  over surfaces with total space $W$, fiber $S$, and base $\Sigma$, such that
$\pi_1(W)$ does not contain a copy of $\mathbb{Z} \oplus \mathbb{Z}$. One way of ensuring this is to find an example
where the total space has  (Gromov) hyperbolic fundamental group $\pi_1(W)$.  A quasi-isometric model for the universal
cover $\widetilde{W}$ is a metric graph bundle where the fibers are Cayley graphs of $\pi_1(S)$ and the base $K$ a 
Cayley graph of $\pi_1(\Sigma)$. 
Using a construction of Leininger and Schleimer \cite{ls-disk} in conjunction with Proposition \ref{eg} we construct examples 
of hyperbolic metric graph bundles where fibers are Cayley graphs of $\pi_1(S)$ and $K$ is a 
hyperbolic disk.  However the disks $K$  are not invariant under a surface group; so we only obtain surface bundles $W$ over
$K$ with fiber $S$ such that the universal cover $\widetilde{W}$ is hyperbolic.

We also obtain the following  characterization of convex cocompact subgroups of mapping class groups of surfaces
$S^h$ with punctures. Recall that the pure mapping class group is the (finite index) subgroup of the mapping class
group that keeps the punctures fixed.

\smallskip

\noindent {\bf Proposition \ref{coco}:} {\it
Let $K=\pi_1(S^h)$  be the fundamental group of a surface with finitely many punctures
and let $K_1, \cdots, K_n$ be its peripheral subgroups.  Let $Q$ be a convex cocompact subgroup of the  
 pure mapping class group of $S^h$.
Let
\[
1\rightarrow K\rightarrow G\stackrel{p}{\rightarrow}Q\rightarrow 1
\]
and
\[
1\rightarrow K_i\rightarrow N_G(K_i) \stackrel{p}{\rightarrow}Q_i\rightarrow 1
\]
be the induced short exact sequences of   groups.
Then
$G$ is  strongly hyperbolic  relative to the collection $\{ N_G(K_i)\}, i=1, \cdots, n$. 

Conversely, if $G$ is  (strongly) hyperbolic  relative to the collection $\{ N_G(K_i)\}, i=1, \cdots, n$,
 then $Q$ is convex-cocompact. }

\smallskip

Theorem \ref{combthm} also provides the following
 combination theorem whenever we have  an exact sequence with hyperbolic quotient and kernel. This
 gives a converse to a result of Mosher \cite{mosher-hypextns}.

\smallskip

\noindent {\bf Theorem \ref{combthmgps}:} {\it
Suppose that the short exact sequence of finitely generated groups 

\begin{center}
$1\rightarrow K\rightarrow G\rightarrow Q\rightarrow 1$
\end{center}
satisfies a flaring condition such that $K,~ Q$ are word hyperbolic and $K$ is non-elementary. Then $G$ is hyperbolic.
}

\smallskip

The next Proposition links the flaring condition to hyperbolicity of the base.

\smallskip

\noindent {\bf Proposition \ref{converseflare}: }{\it
Consider the short exact sequence of finitely generated groups 
\begin{center}
$1\rightarrow K\rightarrow G\rightarrow Q\rightarrow 1$
\end{center}
such that $K$ is non-elementary word hyperbolic but $Q$ is not hyperbolic.
Then the short exact sequence cannot satisfy a flaring condition.}

\smallskip

We also prove an analog of Proposition \ref{converseflare} for relatively hyperbolic groups and use it to generalize a 
 result  of Mosher \cite{mosher-hypextns}
as follows.

\smallskip

\noindent {\bf Proposition \ref{mosher-genlzn}: }{\it Suppose we have a short exact sequence of finitely generated groups 
\[
1\rightarrow (K,K_1)\rightarrow (G,N_G(K_1))\stackrel{p}{\rightarrow}(Q,Q_1)\rightarrow 1
\]
with $K$  (strongly) hyperbolic  relative to the cusp subgroup $K_1$ such that $G$ preserves cusps and
  $Q_1= {N_G(K_1)}/{K_1}$. Suppose further that 
$G$ is  (strongly) hyperbolic  relative to $N_G(K_1)$. Then $Q (=Q_1)$ is hyperbolic. }

\smallskip

Finally we show the necessity of flaring.

\smallskip

\noindent {\bf Proposition \ref{necflaring}: }{\it Let $P: X \rightarrow B$ be a metric (graph) bundle such that \\
1) $X$ is hyperbolic. \\
2) There exists $\delta_0$ such that each fiber $F_z = p^{-1} (z) \subset X$ equipped with the inherited path metric is
$\delta_0$-hyperbolic.\\
Then the metric bundle satisfies a flaring condition.

\smallskip

In particular,  any exact sequence of finitely generated groups $1 \rightarrow N \rightarrow G \rightarrow Q \rightarrow 1$
with $N$ and $G$ hyperbolic, satisfies a flaring condition.}

\medskip

\noindent {\bf Outline of the main steps:}\\
There are four main steps in the proof of  Combination Theorem  \ref{combthm}. 
Precise definitions of terms  are given in the next subsection.\\
1) First we construct a metric graph bundle (see Definition $\ref{defn-mgbdl}$) out of a  given metric bundle.
The bundles have quasi-isometric base space and total space. Next we set out to prove 
that this metric graph bundle is hyperbolic under the given conditions on the metric bundle.\\
2) Proposition \ref{existence-qi-section}  proves the existence of qi sections and is the coarse geometric analog
of the statement that any fiber bundle with contractible base admits a section. The main ingredient of the
proof is the definition of a `discrete flow' of one fiber to another fiber. 
This is the content of Section 2.1. The main idea is elaborated upon in the first paragraph
of Section 2.1.\\
3) Any two such qi sections bound a `ladder' between them (cf. Definition
 \ref{defn-ladder} below, \cite{mitra-ct}, \cite{mitra-trees}). The next step is to
prove the hyperbolicity of these ladders. In Section \ref{small-ladder}
we prove hyperbolicity of small-girth ladders (Proposition \ref{main-lem}). In Section \ref{big-ladder} we break up
a big ladder into small-girth ladders and use a consequence (Proposition \ref{hyp-tree})
 of a combination theorem due to Mj-Reeves \cite{mahan-reeves} to conclude that
the whole ladder is hyperbolic. \\
4) In Section 4, we assemble the pieces to prove Theorem  \ref{combthm}.

For the reader interested in getting to the main ideas of the proof of Theorem  \ref{combthm} without getting into technical details,
we have  sketched  Step   (2) above in the first paragraph
of Section 2.1,  and Step   (3) above in the the first paragraph
of Section 3 and the paragraph following the statement of Proposition \ref{main-lem} in Section \ref{small-ladder}.

\smallskip

%
%

\smallskip

\noindent {\bf Acknowledgments:} We would like to thank Panos Papasoglu for explaining the proof
of the last statement of Theorem \ref{bigon} to us. We would also like
to thank Chris Leininger for sharing his examples in \cite{ls-disk} with us.
 This paper  owes an intellectual debt to Hamenstadt's paper  \cite{hamenst-word}, which inspired
us to find a combination Theorem in the generality described here. Finally we would like to thank the referee for a meticulous reading
of the manuscript and for several helpful
remarks and comments. In particular, the notion of a metric graph bundle arose out of the referee's comments on an earlier draft.

\subsection{Metric Bundles} \label{mbdls}

\subsubsection{Some Basic Concepts}
We  recall some basic notions from large scale geometry. 

Let $X$, $Y$ be metric spaces and let $k\geq 1, \, \epsilon \geq 0$. 

\begin{enumerate}
\item A map $\phi:X\rightarrow Y$ is said to be metrically proper if for all  $N\geq 0$ there exists 
$M\geq 0$ such that $x,y\in X$, and $d(\phi(x),\phi(y))\leq N$ implies $d(x,y)\leq M$.

Suppose $\{(X_{\alpha}, d_{X_{\alpha}}) \}$ and $\{(Y_{\alpha}, d_{Y_{\alpha}})\}$ are  families of metric spaces.
For   any function $f:\mathbb R^{+} \rightarrow \mathbb R^{+}$,
a  family of maps $\phi_{\alpha}:X_{\alpha}\rightarrow Y_{\alpha}$ is said to be {\bf uniformly metrically proper as measured
by} $f$ if for all $\alpha$ and $x,y \in X_{\alpha}$, $d_{Y_{\alpha}}(\phi_{\alpha}(x),\phi_{\alpha} (y))\leq N$
implies  $d_{X_{\alpha}}(x,y)\leq f(N)$. If such an $f$ exists we shall say that the 
collection of maps $\phi_{\alpha}$ is  uniformly metrically proper or, more simply, uniformly proper.

\item Suppose $A$ is a set. A map $\phi:A \rightarrow Y$ is said to be $\epsilon-${\bf coarsely surjective}
if $Y$ is contained in the $\epsilon$-neighborhood $\phi(A)$.

Suppose $\{A_{\alpha} \}$ and $\{ Y_{\alpha} \}$ are respectively a family of sets and a family of metric spaces.
A family of maps $\phi_{\alpha}:A_{\alpha}\rightarrow Y_{\alpha}$ is said to be {\bf uniformly coarsely surjective} 
if there is a constant $D\geq 0$, such that for all $\alpha$,  $ Y_{\alpha}$ is contained in the $D$-neighborhood of 
$\phi_{\alpha}(A_{\alpha})$.

\item A map $\phi: X\rightarrow Y$ is said to be  $\epsilon$-{\bf coarsely Lipschitz} if $\forall x_1,x_2\in X$ we have
$d(\phi(x_1),\phi(x_2))\leq \epsilon.d(x_1,x_2) + \epsilon $. A map $\phi$ is coarsely Lipschitz if it is 
$\epsilon$- coarsely Lipschitz for some $\epsilon \geq 1$.

\item $(i)$ Recall \cite{gromov-hypgps} \cite{GhH} that a map $\phi:X\rightarrow Y$ is said to be
a $(k,\epsilon)$-{\bf quasi-isometric embedding} if   $\forall x_1,x_2\in X$ one has
$$ d(x_1,x_2)/k  -\epsilon\leq d(\phi(x_1),\phi(x_2))\leq k.d(x_1,x_2) + \epsilon .$$
A map $\phi:X\rightarrow Y$ will simply be referred to as
a quasi-isometric embedding if it is a $( k, \epsilon)$-quasi-isometric embedding for some $k\geq 1$ and $\epsilon\geq 0$.
A $(k,k)$-quasi-isometric embedding will  be referred to as a $k$-{\bf quasi-isometric embedding}. \\
$(ii)$ A map $\phi:X\rightarrow Y$ is said to be a $(k,\epsilon)$-{\bf quasi-isometry}
(resp. $k$-{\bf quasi-isometry}) if it is a $(k,\epsilon)$-quasi-isometric embedding (resp. 
$k$-quasi-isometric embedding) and if $\phi$ is $D-$coarsely surjective for some $D\geq 0$.\\
$(iii)$ A $(k,\epsilon)$-{\bf quasi-geodesic} (resp. a $k$-{\bf quasi-geodesic}) in a metric space $X$ is 
a $(k,\epsilon)$-quasi-isometric embedding (resp. a $k$-quasi-isometric embedding) $\gamma:I\rightarrow X$, where
$I\subseteq \mathbb R$ is an interval.

\item A map $\psi: Y\rightarrow X$ is said to be an $\epsilon$-{\bf coarse inverse} of a map $\phi: X\rightarrow Y$ if 
for all $x\in X$ and  $y\in Y$ one has $d_X(\psi\circ \phi(x),x)\leq \epsilon$  and
$d_X(\phi\circ \psi(y),y)\leq \epsilon$.

\end{enumerate}

The following lemma is straightforward. We include a proof for the
sake of completeness.

\begin{lemma}\label{elem-lemma1}
For every $K_1,K_2\geq 1$ and $D\geq 0$ there are $K_{\ref{elem-lemma1}}=K_{\ref{elem-lemma1}}(K_1,K_2,D)$,
and $K^{'}_{\ref{elem-lemma1}}=K^{'}_{\ref{elem-lemma1}}(K_1,D)$ such that the following hold.

\begin{enumerate}
\item A $K_1$-coarsely Lipschitz map with a $K_2$-coarsely Lipschitz, $D$-coarse inverse is a 
$K_{\ref{elem-lemma1}}$-quasi-isometry.
\item Any $D$-coarsely surjective, $K_1$-quasi-isometry has a $K^{'}_{\ref{elem-lemma1}}$-quasi-isometric
coarse inverse.
\end{enumerate}

\end{lemma}

\begin{proof}
$1.$ Let $f:X\rightarrow Y$ be a $K_1$-coarsely Lipschitz map with a $K_2$-coarsely Lipschitz, $D$-coarse inverse
$g:Y\rightarrow X$. Let $x,y, x^{'}, y^{'}\in X$ be such that $g(f(x))=x^{'}$, $g(f(y))=y^{'}$. Since $g$ is a
$D$-coarse inverse of $f$, we have $d(x,x^{'})\leq D$, $d(y,y^{'})\leq D$.
Now, $-d(x,x^{'})-d(y,y^{'})+d(x,y)\leq d(x^{'}, y^{'})\leq K_2 d(f(x),f(y)) +K_2$. Hence,
$-2D+d(x,y)\leq K_2 d(f(x),f(y)) +K_2$. Choosing 
$K_{\ref{elem-lemma1}}= max\{ K_1, 2D+K_2 \}$ completes the proof.

$2.$ Suppose $f:X\rightarrow Y$ is a 
$D$-coarsely surjective, $K_1$-quasi-isometry. We define a map $g:Y\rightarrow X$ as follows:
For all $v\in Y$, choose $x\in X$ such that $d(v, f(x))\leq D$. Define $g(v)=x$.
Let $v_1,v_2\in Y$ and let $g(v_i)=x_i$, $i=1,2$. Then $d(v_i, f(x_i))\leq D$, $i=1,2$.
It follows that $|d(f(x_1),f(x_2))-d(v_1,v_2)|\leq 2D$.
Again, since $f$ is a $K_1$-quasi-isometry, we have 
$-K_1 + \frac{1}{K_1} d(x_1,x_2)\leq d(f(x_1),f(x_2))\leq K_1 + K_1d(x_1,x_2)$.
We deduce from the previous two inequalities that 
$-(K_1+2D) + \frac{1}{K_1} d(x_1,x_2)\leq d(v_1,v_2)\leq (K_1+2D) + K_1d(x_1,x_2)$.
Hence finally, we have 
$$-\frac{(K_1+2D)}{K_1} + \frac{1}{K_1}d(v_1,v_2) \leq d(x_1,x_2) \leq K_1d(v_1,v_2) + (K_1+2D)K_1.$$
Thus $g$ is a $K^{'}_{\ref{elem-lemma1}}$-quasi-isometric embedding where $K^{'}_{\ref{elem-lemma1}}=K_1(K_1+2D)$.

It follows from the definition of $g$ that for all $v\in Y$, one has $d(f(g(v)),v)\leq D$. Let $x\in X$ and
$g(f(x))=x_1$. Hence $d(f(x),f(x_1))\leq D$. Since $f$ is a $K_1$-quasi-isometric embedding, it follows that
$d(g(f(x)),x)=d(x,x_1)\leq K_1(K_1+D)$. Thus $g$ is $K_1(K_1+D)$-coarsely surjective whence a  
$K^{'}_{\ref{elem-lemma1}}$-quasi-isometry. Also  $g$ is a $K_1(K_1+D)$-coarse inverse of $f$. 
\end{proof}

\subsubsection{Metric Bundles and Metric Graph Bundles}
In this subsection we define the primary objects of  study and obtain some basic properties.
\begin{defn}\label{defn-mbdle}
Suppose $(X,d)$ and $(B, d_B)$ are geodesic metric spaces; let $c\geq 1$ and let 
$f:{\mathbb R}^+ \rightarrow {\mathbb R}^+$ be a function.
We say that $X$ is an $(f,c)-$ {\bf metric bundle} over $B$ if there is a surjective $1$-Lipschitz
map $p:X\rightarrow B$ such that the following conditions hold:\\
1) For each point $z\in B$, $F_z:=p^{-1}(z)$ is a geodesic metric space
with respect to the path metric $d_z$ induced from $X$. The inclusion maps
$i: (F_z,d_z) \rightarrow X$ are uniformly metrically proper as measured by $f$. \\
$2)$  Suppose $z_1,z_2\in B$, $d_B(z_1,z_2)\leq 1$ and let $\gamma$ be
a geodesic in $B$ joining them. \\
$2(i)$ Then for any point $x\in F_z$, $z\in \gamma$, there is a path in $p^{-1}(\gamma)$
of length at most $c$ joining $x$ to both $F_{z_1}$ and $F_{z_2}$.
\end{defn}

\begin{rem}
Since the metric on each fiber $F_z,z\in B$ is the  path metric induced from $X$ we always  have
$f(t)\geq t$ for all $t\in \mathbb R^{+}$.
\end{rem}


\noindent {\bf Convention:} We shall use {\bf subscripts} for constants
to indicate the Lemma/ Proposition/Theorem/Corollary where they first appear. 

\begin{prop} \label{def} Let $X$ be an $(f,c)-$  metric bundle over $B$.
Then there exists $K_{\ref{def}}=K_{\ref{def}}(f,c) \geq 1$, such that the following holds. 

Suppose $z_1,z_2\in B$ with $d_B(z_1,z_2)\leq 1$ and let $\gamma$ be
a geodesic in $B$ joining them.
Let $\phi: F_{z_1}\rightarrow F_{z_2}$, be any map such that
$\forall x_1\in F_{z_1}$ there is a path of length at most $c$ in $p^{-1}(\gamma)$
joining $x_1$ to $\phi(x_1)$. Then $\phi$ is a $K_{\ref{def}}$-quasi-isometry.\end{prop}

\begin{proof} Let  $u, v \in F_{z_1}$ such that $d_{z_1} (u,v) \leq 1$. Then $d(\phi(u), \phi (v)) \leq 2c+1$ by the triangle
inequality and hence $d_{z_2}(\phi(u), \phi (v)) \leq f(2c+1)$
by condition 2(i) of the definition of metric
bundles.  It follows that the map $\phi$ is an $f(2c+1)$-coarsely
Lipschitz map.
A similar map $\overline{\phi} : F_{z_2}\rightarrow F_{z_1}$
may be defined, appealing again to condition $2(i)$ of the definition of metric bundles, interchanging the roles of $z_1, z_2$
such that $\overline{\phi}$ is also an $f(2c+1)$-coarsely Lipschitz map. 

Also, $\overline{\phi}$ is a coarse inverse of $\phi$:\\
 $d(\overline{\phi} \circ \phi (u),u) \leq d(\overline{\phi} \circ \phi (u), \phi (u)) 
+  d(\phi (u),u) \leq 2c$ and hence
$d_{z_1}(\overline{\phi} \circ \phi (u),u) \leq f(2c)$; similarly $d_{z_2}(\phi\circ \overline{\phi}(v),v) \leq f(2c)$ 
for all $u\in F_{z_1}$, $v\in F_{z_2}$.

Hence by Lemma \ref{elem-lemma1} (1), $\phi$ is a $K_{\ref{def}}$-quasi-isometry where $K_{\ref{def}}=K_{\ref{elem-lemma1}}(f(2c+1), f(2c+1),f(2c))$.
Note further that $\phi$ is $f(2c)$-coarsely surjective. 
\end{proof}

We will find it convenient to refer to an $(f,c)-$ metric bundle 
as an {\bf $(f,c,K)-$ metric bundle} (with  $K =  K_{\ref{def}}(f,c)$),
or simply a {\bf metric bundle} when the parameters are not important, and refer to the conclusion of the above proposition as {\em Condition 2(ii)} of Definition \ref{defn-mbdle}
of metric bundles.

\smallskip

For the rest of the paper by a {\bf graph} we will always mean a connected metric graph all of whose edges are of length 
$1$.
For a graph $X$,   $\mathcal{V} (X)$ will denote its vertex  set. By a {\bf path} in a graph we will always  mean an
edge path starting and ending at two vertices.

\begin{defn}\label{defn-mgbdl}
Suppose $X$ and $B$ are graphs. Let $f:\mathbb N \rightarrow \mathbb N$ be
a function.

We say that $X$ is an $f-$ {\bf metric graph bundle}
over $B$ if there exists a surjective simplicial map $p:X\rightarrow B$  such that:\\
$1.$ 
For each $b\in \mathcal{ V}(B)$, $F_b:=p^{-1}(b)$ is a connected subgraph of $X$ and the inclusion maps
$i:\mathcal V(F_b)\rightarrow X$ are 
uniformly metrically proper  (as measured by $f$) for the  path metric $d_b$ induced on $F_b$, i.e.
for all $b\in \mathcal V(B)$ and $x,y\in \mathcal V(F_b)$,  $d(i(x),i(y))\leq N$ implies that $d_b(x,y)\leq f(N)$. \\
$2.$ Suppose $b_1,b_2\in \mathcal{ V}(B)$ are adjacent vertices.\\
$2(i).$ Then each vertex $x_1$ of $F_{b_1}$ is connected by an edge with a vertex in $F_{b_2}$.
\end{defn}

\begin{rem}
Since the map $p$ is simplicial it follows that it is $1$-Lipschitz.
\end{rem}

Now, we have the following analog of Proposition \ref{def}.

\begin{prop}\label{def2}
Suppose $X$ is an $f$-metric graph bundle over $B$. Then there exists $K_{\ref{def2}}=K_{\ref{def2}}(f)\geq 1$ 
such that the following holds.

Suppose $b_1,b_2\in \mathcal{ V}(B)$ are adjacent vertices. Let $\phi:F_{b_1}\rightarrow F_{b_2}$ be any map 
such that each $x_1\in \mathcal V(F_{b_1})$ is connected to $\phi(x_1)\in \mathcal V(F_{b_2})$
by an edge, and any interior point on an edge of $F_{b_1}$ is sent to the image of one of the vertices on which the edge
is incident. Then  any such $\phi$ is a  $K_{\ref{def2}}$-quasi-isometry.  
\end{prop}

\begin{proof} First note that $ d_{b_1} (u,v) \leq 1$ implies that $ d_X (\phi (u), \phi (v)) \leq  4$
by the triangle inequality. Hence $ d_{b_2} (\phi (u), \phi (v)) \leq  f(4)$ since $X$ is an $f-$  metric graph bundle.
Thus $\phi$ is an $f(4)$-coarsely Lipschitz map. 

Let $\overline{\phi}:F_{b_2}\rightarrow F_{b_1}$ be an analogous map defined by interchanging the roles of $b_1$ and
$b_2$.
As in the proof of Proposition \ref{def} we see that $\overline{\phi}$ is an $f(3)$-coarsely surjective, 
$f(4)$-coarsely Lipschitz, $f(3)$-coarse inverse of $\phi$. 
Thus $\phi$ is a $K_{\ref{def2}}=K_{\ref{elem-lemma1}}(f(4),f(4),f(3))$-quasi-isometry (by Lemma \ref{elem-lemma1} (1)).

Note also that $\phi$ is an $f(3)$-coarsely surjective map.
\end{proof}

We will find it convenient to refer to an $f-$metric graph bundle
as an {\bf $(f,K)$-metric graph bundle} (with $K=K_{\ref{def2}}(f)$), or simply as a {\bf metric graph bundle} when $f, K$ are understood,
 and refer to the conclusion of the above proposition as {\em Condition 2(ii)} of Definition
\ref{defn-mgbdl}.

For both metric bundles and metric graph bundles
the spaces $(F_z,d_z)$, $z\in B$ or  $z\in \mathcal{V} (B)$, will be referred to as  {\bf horizontal spaces} or {\bf fibers}
and the distance between two points in $F_z$ will be referred to as their {\bf horizontal distance}.
(Here we have the mental picture that the bundle projection maps go from left to right, and identify  fibers to points.)
A geodesic in $F_z$ will be called a {\bf horizontal geodesic}. The spaces $X$ and $B$ will be referred to as the {\em total space} 
and the {\em base space} respectively. By a statement of the form `$X$ is a metric bundle (resp. metric graph bundle)' we will mean
that it is the total space of a metric bundle (resp. metric graph bundle).

A principal motivational example is the following.
\begin{example}\label{eg-mbdl}
{\rm Suppose we have an exact sequence of finitely generated groups 
$$1\rightarrow N\stackrel{i}{\rightarrow} G\stackrel{\pi}{\rightarrow} Q\rightarrow 1.$$
This naturally gives rise to a metric graph bundle as follows. Choose a finite symmetric generating set $S$ of $G$ such that $S$ 
contains a symmetric generating set $S_1$ of $N$. 
Let $X=\Gamma(G,S)$ be the Cayley graph of $G$ with respect to the generating set $S$.
Let $T=(\pi(S) \setminus \{1\})$ and $B:=\Gamma(Q,T)$ be the Cayley graph of the group $Q$ with respect to the generating set $T$.

Then  the map $\pi$ naturally induces a simplicial map  $\pi:X\rightarrow B$ between  Cayley graphs.
In fact, $\pi$ maps an edge connecting two vertices of $X$ to a vertex of $B$ iff the vertices
are both contained in the same coset of $N$ in $G$ and $\pi$ maps any edge connecting two distinct cosets of $N$
isometrically onto an edge of $B$. Define $f:\mathbb N\rightarrow \mathbb N$ as follows: $f(n)=$ number of vertices
of $\Gamma(N, N\cap S)$ contained in the $n$-ball of $X$ about the identity element $1_G$ of $G\subset X$.
Note that $\Gamma(N, N\cap S)$ is the inverse image of the identity  element of $Q\subset B$
under $\pi$. Since the inverse images of the vertices of $B$ under $\pi$ are  translates
of the Cayley graph $\Gamma(N, N\cap S)$ under  left multiplication by elements of $G$,  condition $1$ of Definition $\ref{defn-mgbdl}$ is satisfied. 

Condition $2(i)$ may be verified as follows: Let 
$\pi (g_iN)= v_i\in Q$, $i=1,2$. Suppose $v_1,v_2$ are adjacent vertices of $B$.
Then there exist $n_1, n_2\in N$ such that $g_1n_1$ and $g_2n_2$ are connected
by an edge in $X$. Thus $s=(g_1n_1)^{-1}g_2n_2\in S$. Hence for any element $n\in N$, $g_1n$ is connected 
to $g_1n.s=(g_1.n.n^{-1}_1g^{-1}_1)g_2n_2$ and $g_1n.s$ is contained in $g_2N$ since $N$ is a normal subgroup of
$G$. 
Thus we have a metric graph bundle structure on $X$ over $B$.}
\end{example}

Another simple example to keep in mind is the following.
\begin{example}\label{eg-h2}
{\rm Let $X={\mathbb{H}}^2$ and  $B={\mathbb{R}}$. Identify $B$ with a bi-infinite geodesic $\gamma \subset X$ with endpoints $a, b$ on the ideal
boundary. Through $x \in \gamma$, let $F_x$ be the unique horocycle based at $a$. Define $p : X \rightarrow B$ by $p (F_x) = x$. 
This gives rise to a metric bundle structure on $X$ over $B$. Note that
each $F_x$, equipped with the induced path-metric, is   abstractly isometric to  ${\mathbb{R}}$.}
\end{example}

A more interesting set of examples is furnished by Proposition \ref{eg} towards the end of the paper.


\begin{defn} Let $p: X \rightarrow B$ be a metric bundle (resp. metric graph bundle)
and $k\geq 1$.  Then $X_1\subseteq X$ is said to be a
{\bf $k-$q(uasi)-i(sometric) section of $B$}, 
if there is a $k$-quasi-isometric embedding $s:B\rightarrow X$ (resp. $s:\mathcal{V} (B)\rightarrow \mathcal{V} (X)$)
such that $p\circ s=Id$ (resp. $p\circ s = Id$ on $\mathcal V(B)$) and $X_1=Im(s)$. \\
If $X_1$ is a $k$-qi section and $x\in X_1$, then we say that $X_1$ is 
a {\em $k$-qi section through $x$}. Also, $X_1\subset X$ is said to be a qi section if it is a
$k$-qi section for some $k\geq 1$.
\end{defn}

\begin{defn}
{\rm
Let $\gamma:I\rightarrow B$ be a geodesic, where $I\subseteq \mathbb R$
is an interval. By a {\bf $k-$qi lift} of $\gamma$ in $X$, we mean a  
$k$-quasi isometric embedding $\tilde{\gamma}:I\rightarrow X $ such that
$p \circ \tilde{\gamma}=\gamma$ (with the pro viso that for a metric graph bundle, $I$ is of the form $[0,n]$
for some $n \in \mathbb{N}$, and the equality $p \circ \tilde{\gamma}=\gamma$ holds only at the integer points).\\
 Suppose $X_1\subseteq X$ is a $k-$qi-section  and
$\gamma:I\rightarrow B$ is a geodesic. By the $k-$qi lift of $\gamma$ in
$X_1$ we mean a $k-$qi lift of $\gamma$ whose image is contained
in $X_1$. }
\end{defn}

\begin{defn}\label{defn-flare}
Suppose $p:X\rightarrow B$ is a metric bundle or a metric graph bundle.
We say that it satisfies a {\bf flaring condition} if for all   $k \geq 1$, there exist
  $\lambda_k>1$ and  $n_k,M_k\in \mathbb N$ such that
the following holds:\\
Let $\gamma:[-n_k,n_k]\rightarrow B$ be a geodesic and let
$\tilde{\gamma_1}$ and $\tilde{\gamma_2}$ be two
$k$-qi lifts of $\gamma$ in $X$.
If $d_{\gamma(0)}(\tilde{\gamma_1}(0),\tilde{\gamma_2}(0))\geq M_k$,
then we have
\[\mbox{
{\small $\lambda_k.d_{\gamma(0)}(\tilde{\gamma_1}(0),\tilde{\gamma_2}(0))\leq \mbox{max}\{d_{\gamma(n_k)}(\tilde{\gamma_1}(n_k),\tilde{\gamma_2}(n_k)),d_{\gamma(-n_k)}(\tilde{\gamma_1}(-n_k),\tilde{\gamma_2}(-n_k))\}$}}.
\]
\end{defn}

\begin{lemma}\label{condition3}
Given a function $f:\mathbb N \rightarrow \mathbb N$
there is a function $g:\mathbb R^{+}\rightarrow \mathbb R^{+}$ such that the following  holds:

Suppose $X$ is an $(f,K)$-metric graph bundle over $B$ and $b_1,b_2\in \mathcal{V} (B)$ with
$d(b_1,b_2)= 1$. Let $C \geq 0$ and let $\phi: F_{b_1} \rightarrow F_{b_2}$ be any map such that
$\forall x_1\in F_{b_1}$,  $\phi(x_1)\in \mathcal V(F_{b_2})$ and $d(x_1,\phi(x_1))\leq C$. 
Then $\phi$ is a $f(2[C]+1)$-Lipschitz map when
restricted to $\mathcal V(F_{b_1})$; also $\phi$ is a $g(C)$-quasi-isometry (here $[C]$ is the integer part of $C$).
\end{lemma}

\begin{proof} Suppose $z_1,z_2\in \mathcal V(F_{b_1})$ are adjacent vertices. 
Then $d(\phi(z_1),\phi(z_2))\leq d(z_1,\phi(z_1))+d(z_2,\phi(z_2)) + d(z_1,z_2)\leq 2[C]+1$.
since $d(z_j,\phi(z_j))$, $j=1,2$ are integers by the definition of $\phi$. 
Thus $d_{b_2}(\phi(z_1),\phi(z_2))\leq f(2[C]+1)$.
The first conclusion follows.

Let $\phi_0:F_{b_1}\rightarrow F_{b_2}$ be a map such that each $x\in \mathcal V(F_{b_1})$ is connected to 
$\phi_0(x)\in \mathcal V(F_{b_2})$ by an edge, and any interior point on an edge of $F_{b_1}$ is sent to the image of 
one of the vertices on which the edge is incident. We note that $d(x,\phi_0(x))\leq 2$ for all $x\in F_{b_1}$.
Also, condition $2(ii)$ says that $\phi_0$ is a $K$-quasi-isometry.
Now, $d(\phi_0(x),\phi(x))\leq d(\phi_0(x),x)+d(x,\phi(x))$ and so $d(\phi_0(x),\phi(x))\leq [C]+2$, for all $x\in F_{b_1}$. Hence
$d_{b_2}(\phi_0(x),\phi(x))\leq f([C]+2)$, for all $x\in F_{b_1}$. 
We know that any map which is at a distance at most $R$ from a $K$-quasi-isometry is a $(K+2R)$-quasi-isometry.
Choosing $g(C)$ to be $K+2f([C]+2)$ concludes the proof. \end{proof}

{\bf Bounded flaring condition for metric graph bundles}
\begin{cor}\label{bdd-flaring}
For all $k\in \mathbb R$, $k\geq 1$ there is a function $\mu_k:\mathbb N\rightarrow [1, \infty )$
such that the following holds:

Suppose $X$ is an $(f,K)$-metric graph bundle with base space $B$. 
Let $\gamma\subset B$ be a geodesic joining $b_1,b_2\in \mathcal{V} (B)$,
and let $\tilde{\gamma_1}$, $\tilde{\gamma_2}$ be two $k$-qi 
lifts of $\gamma$ in $X$ which join the pairs of points $(x_1,x_2)$ and $(y_1,y_2)$ respectively, so that
$p(x_i)=p(y_i)=b_i$, $i=1,2$. For all $N\in \mathbb N$, if $d_B(b_1,b_2)\leq N$ then 
\[
d_{b_1}(x_1,y_1)\leq \mu_k(N) \mbox{max}\{d_{b_2}(x_2,y_2),1\}.
\]
\end{cor}

\begin{proof} Let $b_1=v_0,v_1,\cdots, v_n=b_1$ be the sequence of consecutive vertices on the geodesic $\gamma$.
We must have $n\leq N$.
Define for all $i=0,1,\cdots n-1$, $\phi_i:F_{v_i}\rightarrow F_{v_{i+1}}$ by appealing to condition $2(i)$
of the definition of metric graph bundles such that $\phi_i(\tilde{\gamma_j}(i)) = \tilde{\gamma_j}(i+1)$, $j=1,2$.
By the first conclusion of Lemma \ref{condition3} each  $\phi_i$ is $f(2[2k]+1)$-Lipschitz
when restricted to $\mathcal V(F_{v_i})$.

Choosing $\mu_k(N)= f(2[2k]+1)^N$  concludes the proof. \end{proof}

  Lemma \ref{condition3} has an obvious analog for any $(f,c)$-metric bundle.
The same  applies to  Corollary \ref{bdd-flaring} as well.
Since the proofs are very similar we omit them.

\begin{lemma}
Given a function $f:\mathbb R^{+} \rightarrow \mathbb R^{+}$ and $c\geq 0$
there is a function $g:\mathbb R^{+}\rightarrow \mathbb R^{+}$ such that the following  holds:

Suppose $X$ is an $(f,c,K)$-metric bundle over $B$ and $b_1,b_2\in B$ with
$d_B(b_1,b_2)\leq 1$. Let $C \geq 0$ and let $\phi: F_{b_1} \rightarrow F_{b_2}$ be any map such that
$\forall x_1\in F_{b_1}$,  $d(x_1,\phi(x_1))\leq C$. Then $\phi$ is a $g(C)$-quasi-isometry.
\end{lemma}

{\bf Bounded flaring condition for metric bundles}
\begin{cor}\label{bdd-flaring-mbdl}	
For all $k\in \mathbb R^{+}, k\geq 1$ there is a function $\mu_k:{\mathbb R}^+\rightarrow [1, \infty )$
such that the following holds:

Suppose $X$ is an $(f,c,K)$-metric bundle with base space $B$. 
Let $\gamma\subset B$ be a geodesic joining $b_1,b_2\in B$,
and let $\tilde{\gamma_1}$, $\tilde{\gamma_2}$ be two $k$-qi 
lifts of $\gamma$ in $X$ which join the pairs of points $(x_1,x_2)$ and $(y_1,y_2)$ respectively, so that
$p(x_i)=p(y_i)=b_i$, $i=1,2$. For all $N\in \mathbb R^{+}$, if $d_B(b_1,b_2)\leq N$ then 
\[
d_{b_1}(x_1,y_1)\leq \mu_k(N) \mbox{max}\{d_{b_2}(x_2,y_2),1\}.
\]
\end{cor}

In the rest of the paper, we will summarize the conclusion 
of Corollaries \ref{bdd-flaring} 
and \ref{bdd-flaring-mbdl} by saying that a metric bundle or a metric graph bundle satisfies a {\bf bounded flaring condition}.

We end this subsection by showing that a metric bundle naturally gives rise to a metric graph bundle, 
such that the respective base and total spaces are  quasi-isometric. But first, 
we recall the general fact that geodesic metric spaces
are quasi-isometric to connected graphs (see \cite{gromov-ai} p.7 or \cite{bridson-haefliger} p. 152).

\begin{lemma}\label{coarse} 
$1.$ Let $Y$ be a geodesic metric space and let $V\subset Y$ be a subset such that for some $D>0$ and all $y\in Y$ there 
exists $z\in V$ such that $d(y,z)\leq D$. Let $E\geq 2D+1$.
Let $Z$ be a graph such that \\
a) the vertex set $\mathcal{V} (Z) = V$\\
b) the edge set $\mathcal{E} (Z)$ is given by $\{ y,z \} \in \mathcal{E} (Z)$   iff $y\neq z$ and $d(y,z)\leq E$.\\
Define $\psi_Z:Z\rightarrow Y$ as follows: $\psi_Z (u) = u$ for
 $u\in V$. For  an edge $e$ of $Z$ choose some $u \in V$ such that $e$ is incident on $u$ and map the interior of $e$ to $u$
under $\psi_Z$.\\
Then for all $u,v\in V$ we have $-1+d_Z(u,v)\leq d_Y(u,v)\leq E.d_Z(u,v)$.
In particular, $\psi_Z$ is a max$\{5,4E\}$-quasi-isometry. \\
$2.$ Suppose $Z_1$ is a connected subgraph of a graph $Z$ such that the vertex sets of $Z_1, Z$ are the same and the
following holds: Let $E_1 > 1$ and suppose any edge of $Z$ which is not in $Z_1$
connects two vertices of $Z_1$ which are at a distance of at most $E_1$ in $Z_1$. 
Then for all $u,v\in Z_1$ we have $d_Z(u,v)\leq d_{Z_1}(u,v)\leq E_1d_Z(u,v)$. In particular the inclusion
$Z_1\hookrightarrow Z$ is an $E_1$-quasi-isometry. 
\end{lemma}

Now, suppose $p:X^{'}\rightarrow B^{'}$ is an $(f,c,K)$-metric  bundle. Let $d$ denote the metric
on $X^{'}$ and let $d_{B^{'}}$ be the metric on $ B^{'}$. Let $V\subset B^{'}$ be a maximal  subset such that
 $u,v\in V, u\neq v$ implies $d_{B^{'}}(u,v)\geq 1$. Then for all $b\in B^{'}$ there exists $u\in V$
such that $d_{B^{'}}(b,u)\leq 1$. Using the recipe of 
Lemma \ref{coarse} (1) construct \\
a) a graph $B$ with vertex set $V$ such that 
$u\neq v\in V$ are connected by an edge iff $d_{B^{'}}(u,v)\leq 3$, \\
b) and a quasi-isometry $\psi_B:B\rightarrow B^{'}$.


Next, for all $u\in V$ let $X^{'}_u$ be a maximal subset of the horizontal space $F_u$
such that for $x,y\in X^{'}_u$, $d_u(x,y)\geq 1$. 
\begin{lemma}\label{coarse2}
$1.$ For all $x\in X^{'}$ there exists $u\in V$ and a path of length at most $c+1$ connecting $x$ to a point of $X^{'}_u$.\\
$2.$ If $u,v\in V$ are connected by an edge in $B$ then each point of $X^{'}_u$ is connected to a point of $X^{'}_v$
by a path in $X^{'}$ of length at most $3c+1$. 
\end{lemma}

\begin{proof}
Both  statements follow from  condition $2(i)$ of the definition of metric bundles. 
\end{proof}

Now construct a graph $X^{''}$ with vertex set
$\mathcal{V} (X^{''})  = \cup_{u\in V}X^{'}_u$ and edge set $\mathcal{E} (X^{''}) =\{ \{x, y \} : x\neq y \in \mathcal{V} (X^{''}),
d(x,y)\leq 6c+3 \}$.

Let $X\subset X^{''}$ be the subgraph of  $X^{''}$  such that $\mathcal V (X)=\mathcal V(X^{''})$ and
any edge $(x, y) \in \mathcal{E} (X^{''})$ also belongs to $ \mathcal{E} (X)$ iff\\
a) either $x,y \in X^{'}_u$ for some $u\in V$\\
b) or $x \in X^{'}_u$ and $y \in X^{'}_v$ with $d_B(u,v)=1$. 

Let $\psi_X: X\rightarrow X^{'}$ be a map as in Lemma \ref{coarse} (1)
defined by setting $\psi_X (x)=x$ for $x \in \cup_{u\in V}X^{'}_u$. 
Then $p\circ \psi_X=\psi_B\circ \pi$ on $\cup_{u\in V}X^{'}_u$.
Let $\psi_X$ again denote an extension of this map over edges of $X$ by sending the interior of any edge to a vertex on which 
it is incident consistently ensuring that $p\circ \psi_X=\psi_B\circ \pi$.

For all $u\in V$ let us denote by $H_u$ the graph
with vertex set $X^{'}_u$ and  
$\mathcal E(H_u):=\{ e\in \mathcal E(X): e\, \, \mbox{connects two elements of}\, X^{'}_u\}$.

\begin{lemma} \label{coarse3}
There is a constant $C$ such that the maps
$H_u\rightarrow F_u$ obtained by restricting $\psi_X$ are $C$-quasi-isometries.
\end{lemma}
\begin{proof}
First of all, $H_u$ is a connected graph by Lemma \ref{coarse} (1).
Next, for all $u\in V$, let $\bar{H}_u$ be the graph with vertex set $X^{'}_u$  and
edge set 
$\mathcal E(\bar{H}_u):=\{e\in \mathcal E(H_u), e \,\, \mbox{connects}\, x,y\in X^{'}_u: 
d_u(x,y)\leq f(6c+3)\}$. Then $H_u$ is a subgraph of $\bar{H}_u$.

Let us consider an extension of the map $H_u\rightarrow F_u$ to a map $\bar{H}_u\rightarrow F_u$
satisfying the properties of Lemma \ref{coarse} (1). Such a map is, therefore, a quasi-isometry. By 
Lemma \ref{coarse} (2) the inclusion map $H_u\hookrightarrow \bar{H}_u$ is also a quasi-isometry.
Since the map $H_u\rightarrow F_u$ is the composition of  quasi-isometries 
$H_u\hookrightarrow \bar{H}_u$ and  $\bar{H}_u\rightarrow F_u$, the lemma follows.
\end{proof}

\begin{lemma} \label{coarse4}
$\psi_X:X\rightarrow X^{'}$ is a quasi-isometry.
\end{lemma}

\begin{proof}
Let $\psi_{X^{''}}:X^{''}\rightarrow X^{'}$ be an extension of the map $\psi_X:X\rightarrow X^{'}$
with the property of Lemma \ref{coarse} (1). By Lemma \ref{coarse} (1) the map 
$\psi_{X^{''}}:X^{''}\rightarrow X^{'}$ is a $2(6c+3)$-quasi-isometry.

Next we show that the inclusion $X\hookrightarrow X^{''}$
is a quasi-isometry. For this suppose $x,y\in \mathcal V(X)$ are connected by an edge in $X^{''}$.
Suppose $x\in X^{'}_u, y\in X^{'}_v$, $u,v\in V$. Then $d_{B^{'}}(u,v)\leq d(x,y)\leq 6c+3$.
Thus $u,v$ can be joined by a path of length at most $6c+4$, by Lemma \ref{coarse} (1).
Thus $x$ can be joined to a point $z\in X^{'}_v$ by an edge path in $X$ of length at most
$6c+4$. It follows that $d(x,z)\leq (3c+1)(6c+4)$. Thus $d(y,z)\leq 1+(3c+1)(6c+4)=D$, say.
Hence $d_v(y,z)\leq f(D)$. Using the previous lemma we have $d_{H_v}(y,z)\leq C(C+f(D))$.
Since $H_v$ is a subgraph of $X$, we have $d_X(y,z)\leq C(C+f(D))$. Thus 
$d_X(x,y)\leq d_X(x,z)+d_X(y,z)\leq (6c+4)+C(C+f(D))$.  Lemma \ref{coarse} (2) now shows that the inclusion $X\hookrightarrow X^{''}$
is a quasi-isometry..

Since $\psi_X:X\rightarrow X^{'}$ is the composition of the quasi-isometries 
$\psi_{X^{''}}:X^{''}\rightarrow X^{'}$ and  $X\hookrightarrow X^{''}$, the lemma follows.
\end{proof}

Define $\pi:X\rightarrow B$
by sending edges connecting any two vertices of $X^{'}_u$ (for some $u\in V$) to $u$. 
Any other edge in $X$ must join vertices $x\in X^{'}_u$ and $y \in X^{'}_v$ for some $X^{'}_u, X^{'}_v$ with $d_B(u,v)=1$.
 On any such edge $[x,y]$, $\pi$ is defined to be 
an isometry onto the edge $[u,v]$. Now we have the following.

\begin{lemma}\label{mbdl-mgbdl} 
The map $\pi :X \rightarrow B$ gives a metric graph bundle.
\end{lemma}  

\begin{proof} 

By definition $\pi$ is a surjective, simplicial map. We check the conditions of the definition of metric graph
bundles. 

Condition $2(i)$ follows from Lemma \ref{coarse2} (2) and the definition of the graph $X$.

Let us check condition 1 now.
Note that for all $u\in \mathcal V(B)$, $\pi^{-1}(u)$ is the graph $H_u$. 
By Lemma \ref{coarse3},  $\pi^{-1}(u)$ is a connected subgraph of $X$, $C$-quasi-isometric to $F_u$.
Let $x,y\in \mathcal V(\pi^{-1}(u))$. 
Suppose $d_X (x,y) \leq N$, $N\in \mathbb N$. Then $d(x,y)  \leq N(6c+3)$. 
Since $p:X^{'}\rightarrow B^{'}$ is an $(f,c,K)$-metric  bundle
it follows that $d_u(x,y) \leq f(N(6c+3))$. Hence $d_{H_u}(x,y) \leq C.f(N(6c+3))+C$.
Defining $g(N) = [C.f(N(6c+3))+C]$, we see that
condition 1 of the definition of a metric graph bundle is satisfied. 
 \end{proof}

\smallskip

\noindent {\bf Note:} In the rest of the paper we shall assume that the maps $\psi_X,\psi_B$
are $K_1$-quasi-isometries. We shall refer to $\pi :X \rightarrow B$ above as an {\it approximating metric graph bundle}
of the metric bundle $p:X^{'}\rightarrow B^{'}$.


\subsection{Hyperbolic metric spaces}
We assume that the reader is familiar with the basic definitions and facts about  hyperbolic metric spaces \cite{gromov-hypgps}, \cite{GhH},
\cite{Shortetal}. In this subsection we collect together some of these to fix notions and for later use.

If $X$ is
a geodesic metric space and $x,y\in X$ then $[x,y]$ will denote a
geodesic segment joining $x$ to $y$. For $x,y,z\in X$ we shall denote by
$\bigtriangleup xyz$ a geodesic triangle with vertices $x,y,z$.
For $D\geq 0$ and $A\subset X$,  $N_D(A):=\{x\in X: \, d(x,a)\leq D \, \mbox{ for some} \, a\in A\}$ will be called
the $D$-neighborhood of $A$ in $X$.

\begin{defn}
Suppose $\Delta x_1x_2x_3\subset X$ is a geodesic triangle,
and let  $\delta\geq 0$, $K\geq 0$.
\begin{enumerate}
\item For all $i\neq j\neq k\neq i$, let $c_k \in [x_i,x_j]$ be such that
$d(x_i,c_j)=d(x_i,c_k)$. The points $c_i$ will be called the {\bf internal points} of $\Delta x_1x_2x_3$.
Note that, for all $i\neq j\neq k\neq i$, $d(x_i,c_j)=\frac{1}{2}\{d(x_i,x_j)+d(x_i,x_k)-d(x_j,x_k)\}$.
\item The diameter of the set $\{c_1,c_2,c_3\}$ will be referred to as the {\bf insize} of the
triangle $\Delta x_1x_2x_3$.
\item We say that the triangle $\Delta x_1x_2x_3$ is $\delta$-{\bf slim} if any side of the triangle is contained in the
$\delta$-neighborhood of the union of the other two sides.
\item We say that the triangle $\Delta x_1x_2x_3$ is $\delta$-{\bf thin} if for all $i\neq j\neq k\neq i$ and
$p\in [x_i,c_j]\subset [x_i,x_k]$, $q\in [x_i,c_k]\subset [x_i,x_j]$ with $d(p,x_i)=d(q,x_i)$ one has $d(p,q)\leq \delta$.
\item A point $x\in X$ is said to be a $K$-{\bf center} of $\bigtriangleup x_1x_2x_3$ if $x$ is contained in the $K$-neighborhood
of each of the sides of $\bigtriangleup x_1x_2x_3$.
\end{enumerate}
\end{defn}

\begin{defn} {\bf Gromov inner product:} Let $X$ be any metric space and let $x,y,z\in X$.
Then the Gromov inner product of $y,z$ with respect to $x$, denoted $(y.z)_x$, is defined to
be the number $\frac{1}{2}\{d(x,y)+d(x,z)-d(y,z)\}$.
\end{defn}

\begin{defn} Let $\delta\geq 0$ and $X$ be a geodesic metric space. We say that $X$ is a $\delta$-hyperbolic
metric space if all  geodesic triangles in $X$ are $\delta$-slim.
\end{defn}

\begin{lemma}\label{hyp-defn} (See Proposition $2.1$,\cite{Shortetal})
Suppose $X$ is a $\delta$-hyperbolic metric space. Then the following hold:
\begin{enumerate}
\item All the triangles in $X$ have insize at most $4\delta$.
\item All the triangles in $X$ are $6\delta$-thin.
\end{enumerate}
\end{lemma}

\begin{lemma}\label{stab-qg} \cite{GhH}
{\bf Stability of quasigeodesics:} For all $\delta\geq 0$ and $k\geq 1$ there
is a constant $D_{\ref{stab-qg}}=D_{\ref{stab-qg}}(\delta, k)$ such that
the following holds:

Suppose $Y$ is a $\delta$-hyperbolic metric space. Then
the Hausdorff distance between a geodesic and a $k$-quasi-geodesic 
joining the same pair of end points is less than or equal to $D_{\ref{stab-qg}}$.
\end{lemma}

\begin{defn} {\bf Local quasi-geodesics:}
Let $X$ be a metric space and $K\geq 1,\epsilon\geq 0,L>0$ be constants.
A map $f:I\rightarrow X$, where $I\subset \mathbb R$ is an interval, is said to
be a $(K,\epsilon,L)-$ local quasi-geodesic if for all $s,t\in I$ with $|s-t|\leq L$,
one has $-\epsilon+(1/K)|s-t|\leq d(f(s),f(t))\leq \epsilon +K |s-t|$.
\end{defn}

For the following important lemma we refer to Theorem $1.4$, Chapter $3$,\cite{CDP};
or Theorem $21$, Chapter $5$, \cite{GhH}.

\begin{lemma} {\bf Local quasi-geodesic vs global quasi-geodesic:} \label{local-global-qg}
For all $\delta\geq 0 , ~\epsilon\geq 0$ and $K\geq 1$ there are constants
$L=L_{\ref{local-global-qg}}(\delta,K,\epsilon)$, $\lambda=\lambda_{\ref{local-global-qg}}(\delta,K,\epsilon)$
such that the following holds:
 
Suppose $X$ is a $\delta$-hyperbolic metric space. Then
any $(K,\epsilon, L)$-local quasi-geodesic in $X$ is a $\lambda$-quasi-geodesic.
\end{lemma}

\begin{lemma}\label{barycen} For all $\delta\geq 0$, $\epsilon \geq 0$ and $k\geq 1$,
there is a constant $D_{\ref{barycen}}=D_{\ref{barycen}}(\delta,k,\epsilon)$ such that
the following hold:

$(1)$ Suppose $Y$ is a $\delta$-hyperbolic metric space. Then every geodesic triangle in $Y$
has a $4\delta$-center.

$(2)$ Suppose both $Y$ and $Y^{'}$ are $\delta$-hyperbolic metric spaces and $\phi :Y\map Y^{'}$ is a
$(k,\epsilon)$-quasi-isometric embedding. If $y$ is a $4\delta$-center of $\bigtriangleup y_1y_2y_3\subseteq Y$ and
$y^{'}\in Y^{'}$ is a $4\delta$-center of $\bigtriangleup \phi(y_1)\phi(y_2)\phi(y_3)\subseteq Y^{'}$ then
  $d(y^{'},\phi(y))\leq D_{\ref{barycen}}$, where $d$ is the metric on $Y^{'}$.
\end{lemma}

\begin{proof} By conclusion (1) of Lemma $\ref{hyp-defn}$ the internal points of $\bigtriangleup y_1y_2y_3$ are $4\delta$-centers
of $\bigtriangleup y_1y_2y_3$. This proves part $(1)$ of the lemma.

For $(2)$, first we make the following  observation:
Let $\{c^{'}_i \}$ be the internal points of $\bigtriangleup \phi(y_1)\phi(y_2)\phi(y_3)$. Suppose
$z\in Y^{'}$ is contained in a $D$-neighborhood of each of the sides of $\bigtriangleup \phi(y_1)\phi(y_2)\phi(y_3)$,
for some $D\geq 0$.
Let $p_i\in [\phi(y_j),\phi(y_k)]$, $i\neq j\neq k\neq i$, be such that $d(p_i,z)\leq D$, $1\leq i,j,k\leq 3$;
then $d(p_i,p_j)\leq 2D$. 

Claim: $d(c^{'}_i,p_i)\leq 3D$, $i=1,2,3$.

Since the proofs are quite similar, we do the computation for $i=3$ for concreteness.
Set $A_i=\phi(y_i), i=1,2,3$. Then\\

\noindent
$2d(p_3,c^{'}_3)  \\ =2|d(A_1,c^{'}_3)-d(A_1,p_3)|\\ =2|(A_2,A_3)_{A_1}-d(A_1,p_3)|\\ =|d(A_1,A_3)+d(A_1,A_2)-d(A_2,A_3)-2d(A_1,p_3)| $\\
$=|\{d(A_1,p_2)+d(A_3,p_2)\}+\{d(A_1,p_3)+d(A_2,p_3)\} $

\hspace{20mm}   $~~~~~~ \, -\{d(A_3,p_1)+d(A_2,p_1)\}-2d(A_1,p_3)|$

\noindent 
$=|\{d(A_1,p_2)-d(A_1,p_3)\}+\{d(A_3,p_2)-d(A_3,p_1)\}+\{d(A_2,p_3)-d(A_2,p_1)\}|$ \\
$\leq |d(A_1,p_2)-d(A_1,p_3)|+|d(A_3,p_2)-d(A_3,p_1)|+|d(A_2,p_3)-d(A_2,p_1)|$ \\
$\leq d(p_2,p_3)+d(p_2,p_1)+d(p_3,p_1)$\\
$\leq 6D$ \\

This proves the claim and thus $d(z,c^{'}_i)\leq 4D$ for all $i$, $1\leq i\leq 3$.

Since $\phi$ is a $(k,\epsilon)$-quasi-isometric embedding, it follows that $\phi(y)$ is contained in the $(4k\delta +\epsilon)-$ neighborhood
of the image under $\phi$ of each of the sides $[y_i,y_j]$, $i\neq j$. Also, the image of $[y_i,y_j]$, for all $i\neq j$, is
a $(k,\epsilon)$-quasi-geodesic, and hence a $(k+\epsilon)$-quasi-geodesic, joining $\phi(y_i)$, $\phi(y_j)$.  
By Lemma \ref{stab-qg}, $\phi(y)$ is contained in a $\{(4k\delta+\epsilon) + D_{\ref{stab-qg}}(\delta, k+\epsilon)\}-$
neighborhood of each of the sides of $\bigtriangleup \phi(y_1)\phi(y_2)\phi(y_3)$. 
Taking $D_{\ref{barycen}}(\delta,k,\epsilon):=4.\{(4k\delta+\epsilon) + D_{\ref{stab-qg}}(\delta, k+\epsilon)\}$, we are through.  \end{proof}

\begin{defn}
Let $X$ be a geodesic metric space and let $A\subseteq X$. For $K\geq 0$,
we say that $A$ is $K$-quasiconvex in $X$ if any geodesic with end points
in $A$ is contained in the $K$-neighborhood of $A$. A subset $A\subset X$
is said to be quasi-convex if it is $K$-quasi-convex for some $K$.
\end{defn}

\begin{lemma}\label{subqc-elem}
 Let $X$ be a geodesic metric space. 
\begin{enumerate}
\item
Let $p,q,r\in X$. Suppose $q$ is a nearest point
projection of $p$ on a geodesic $[q,r]$ joining $q,r$. Then the arc length parametrization
of the union $[p,q]\cup [q,r]$ is a $(3,0)$-quasi-geodesic in $X$.\\
\item Suppose $U\subset X$ is a $K$-quasi-convex set and $p\not \in U$. Suppose $q\in U$ is a 
nearest point projection of $p$ on $U$. Let $r\in U$. Then the arc length parametrization of
the union $[p,q]\cup [q,r]$ is $(3+2K)$-quasi-geodesic in $X$. 
\end{enumerate}

\end{lemma}

\begin{proof}
$1.$ Suppose $p_1\in [p,q]$, $r_1\in [q,r]$. Then $q$ is a nearest point projection of $p_1$
on $[q,r_1]$. Thus $d(p_1,q)\leq d(p_1,r_1)$. Using the triangle inequality,
$d(q,r_1)\leq d(p_1,r_1)+d(p_1,q)\leq 2d(p_1,r_1)$. Hence  $d(p_1,q)+d(q,r_1)\leq 3d(p_1,r_1)$.

$2.$ Let $p_1\in [p,q]$, $r_1\in [q,r]$. There exists $s\in U$ such that
$d(r_1,s)\leq K$, since $U\subset X$ is $K$-quasi-convex. Now, as before,  $q$ is a
nearest point projection of $p_1$ on $U$. Hence
$d(p_1,q)\leq d(p_1,s)\leq d(p_1,r_1)+K$ and so
$d(q,r_1)\leq d(p_1,q)+d(p_1,r_1)\leq 2d(p_1,r_1)+K$. Thus
$d(p_1,q)+d(q,r_1)\leq 3.d(p_1,r_1)+K$.
\end{proof}

\begin{lemma}\label{subqc}
For each $\delta\geq 0$ and $K\geq 0$ there is a constant $D_{\ref{subqc}}$=$D_{\ref{subqc}}(\delta,K)$
such that the following holds: \\
Suppose $X$ is a $\delta$-hyperbolic metric space and $V\subseteq U$
are $K$-quasi-convex subsets of $X$. Let $x\in X$ and let 
$x_1,x_2$ be nearest point projections of $x$ on $U$ and $V$ respectively.
If $x_3$ is a nearest point projection of $x_1$ on $V$,
then $d(x_2,x_3)\leq D_{\ref{subqc}}$.
\end{lemma}

 \begin{proof}
By Lemma \ref{subqc-elem}(2), $[x,x_1]\cup [x_1,x_2]$ is a $(3+2k)-$quasi-geodesic.
Hence by Lemma $\ref{stab-qg}$, there is a point $x_4\in [x,x_2]$ with 
$d(x_1,x_4)\leq D_{\ref{stab-qg}}(\delta,3+2K)=D$, say. Similarly,
$[x_1,x_3]\cup [x_3,x_2]$ is a $(3+2K)$-quasi-geodesic and thus
there is a point $x^{'}_3\in [x_1,x_2]$ such that $d(x_3,x^{'}_3)\leq D$.
Using the $\delta$-slimness of $\bigtriangleup x_1x_2x_4$, there exists
$x^{''}_3\in [x_2,x_4]$ such that $d(x^{'}_3,x^{''}_3)\leq D+\delta$.
Hence $d(x_3,x^{''}_3)\leq 2D +\delta$. 
Since $x_2$ is a nearest point projection of $x^{''}_3$ on $V$, we have $d(x_2,x^{''}_3)\leq 2D+\delta$.
Thus $d(x_2,x_3)\leq d(x_2,x^{''}_3)+d(x^{''}_3,x_3)\leq 4D+2\delta$.
Setting $D_{\ref{subqc}}=4D+2\delta$ completes the proof of the lemma. \end{proof}

\begin{defn}
Suppose $Y$ is a metric space and $U,V\subset Y$. We say that  $U,V$ are $\epsilon$-{\em separated}
if ${\rm inf}\{d(y_1,y_2):y_1\in U,y_2\in V\}\geq \epsilon$. A collection of subsets $\{U_{\alpha}\}$ of $Y$ is said
to be  uniformly separated if there exists an $\epsilon>0$ such that any pair of distinct elements of the
collection $\{U_{\alpha}\}$ is $\epsilon$-separated.
\end{defn}

\begin{defn} Suppose $Y$ is a $\delta$-hyperbolic metric space
and $U_1,U_2$ are two quasi-convex subsets. Let  $D>0$. 
We say that $U_1,U_2$ are mutually $D$-cobounded, or simply  $D$-cobounded, if any nearest point
projection of $U_1$ to $U_2$ has diameter at most $D$ and vice versa.
\end{defn}

\begin{lemma}\label{cobdd-cor}
Given $\delta\geq 0$ and $K\geq 0$ there are constants $R=R_{\ref{cobdd-cor}}(\delta, K)$ and $D=D_{\ref{cobdd-cor}}(\delta, K)$ such 
that the following  holds:\\
Suppose $X$ is a $\delta$-hyperbolic metric space and $U,V\subset X$ are two $K$-quasiconvex and $R$-separated subsets.
Then $U,V$ are $D$-cobounded.
\end{lemma}

\begin{proof}
Let $V_1 (\subset V)$ be  the set of all  nearest point projections from  points of $U$ to $V$. We want to show that $V_1$ is a set
of uniformly bounded diameter for large enough $R$.

Suppose that $x_1,x_2\in U$ and let $y_1,y_2\in V$ be respectively
their nearest point projections. Then by  Lemma $\ref{subqc-elem}$(2),
$[x_1,y_1]\cup [y_1,y_2]$ and $[x_2,y_2]\cup [y_2,y_1]$ are $K_1-$ quasi-geodesics for $K_1=(3+2K)$. 
If $d(y_1,y_2)\geq D_1:=L_{\ref{local-global-qg}}(\delta,K_1,K_1)$ then
the curve $[x_1,y_1]\cup [y_1,y_2]\cup [y_2,x_2]$ is a $\lambda=\lambda_{\ref{local-global-qg}}(\delta,K_1,K_1)$-quasi-geodesic
by Lemma \ref{local-global-qg}.
Hence every point of this curve is within distance $D_{\ref{stab-qg}}(\delta,\lambda)+K$ from a point in $U$.
Choosing $R= D_{\ref{stab-qg}}(\delta,\lambda)+K+1$ proves the Lemma. 
\end{proof}

\begin{lemma}\label{cobdd-lemma}
Given $\delta\geq 0$ and $K\geq 0$ there are constants $R=R_{\ref{cobdd-lemma}}(\delta, K)$ and $D=D_{\ref{cobdd-lemma}}(\delta, K)$ such 
that the following  holds:\\
Suppose $X$ is a $\delta$-hyperbolic metric space and $U,V\subset X$ are two $K$-quasiconvex and $R$-separated subsets.
Then there are points $x_0\in U$, $y_0\in V$ such that $[x_0,y_0]\subset N_D([x,y])$, for all $x\in U$ and $y\in V$.
\end{lemma}

\begin{proof}
First  consider the set $V_1 (\subset V)$ of all  nearest point projections from  points of $U$ onto $V$. By Lemma
\ref{cobdd-cor}, there exists $R (=R_{\ref{cobdd-cor}})$ such that the diameter of $V_1$ is less than $D=D_{\ref{cobdd-cor}}$
whenever $U,V$ are  $R$-separated.

Choose any point $y_0\in V_1$ and let $x_0$ be a nearest point projection of $y_0$ onto $U$. Let
$x\in U$, $y\in V$ be any pair of points and let $y_1$ be a nearest point projection of $x$ onto $V$. Since $y_1\in V_1$,
it follows that $d(y_0,y_1)\leq D_1$.  By Lemma $\ref{subqc-elem}$ (2)
$[y_0,x_0]\cup [x_0,x]$ is a  $(3+2K)$-quasi-geodesic. Since $X$ is a $\delta$-hyperbolic metric space, 
 the Hausdorff distance between this quasi-geodesic and the geodesic
$[x_0,x]$ is at most $D_{\ref{stab-qg}}(\delta, 3+2K)$ by Lemma \ref{stab-qg}. Similarly the Hausdorff distance 
 between $[x,y_1]\cup [y_1,y]$ and 
$[x,y]$ is  at most $D_{\ref{stab-qg}}(\delta, 3+2K)$. Lastly, since $d(y_0,y_1)\leq D_1$, it follows that 
the Hausdorff distance  between  $[y_0,x]$ and $[y_1,x]$ is at most $\delta+ D_1$. The Lemma follows by choosing 
$D= 2D_{\ref{stab-qg}}(\delta, 3+2K)+D_1+\delta$. \end{proof}

The following is a direct consequence of the
proofs of Lemmas \ref{cobdd-cor} and \ref{cobdd-lemma} (cf. Lemma 3.3 of \cite{mitra-trees}).

\begin{cor}\label{cobdd-cor2}  Given $\delta\geq 0$ and $D, K\geq 0$ there exists $C=C_{\ref{cobdd-cor2}}(\delta, D, K)$ such 
that the following holds.\\
Suppose $X$ is a $\delta$-hyperbolic metric space and $U,V\subset X$ are two $K$-quasiconvex and $D$-cobounded subsets.
Choose $a \in U, b\in V$ such that
$d(a,b) = d(U, V)$, and $[c,a] \subset U$, $[b,d] \subset V$ are $K$-quasigeodesics, then $[c,a] \cup [a,b] \cup [b,d]$
is a $C$-quasigeodesic.
\end{cor}

The following Lemma  \cite{mitra-trees} says that quasi-isometries and nearest point
projections `almost commute'. We include a proof for completeness.

\begin{lemma} (Lemma 3.5 of \cite{mitra-trees}) \label{qi-comm-proj}
For all $\delta\geq 0$ and $k\geq 1$ there is a constant $D_{\ref{qi-comm-proj}}=D_{\ref{qi-comm-proj}}(\delta,k)$
such that the following  holds:\\
Suppose $\phi:X\rightarrow Y$ is a $k$-quasi isometric embedding of
$\delta$-hyperbolic metric spaces. Let $x,y,z\in X$ and let $\gamma$
be a geodesic in $X$ joining $x,y$. Let $u$ be a nearest point projection
of $z$ onto $\gamma$ and suppose $v$ is a nearest point projection of $\phi(z)$
onto a geodesic joining $\phi(x)$ and $\phi(y)$, 
then $d(v,\phi(u))\leq D_{\ref{qi-comm-proj}}$.
\end{lemma} 

\begin{proof}
Let $\{c_i\}$ and $\{c^{'}_i\}$ be respectively the internal points of $\bigtriangleup xyz\subset X$
and $\bigtriangleup \phi(x)\phi(y)\phi(z)$. By Lemma $\ref{subqc-elem}$ (1) the unions
$[x,u]\cup [u,z]$ and $[y,u]\cup [u,z]$ are both $(3,0)$-quasi-geodesics in $X$. It follows that they are $3$-quasi-geodesics.
Hence $u$ is contained in the $D_{\ref{stab-qg}}(\delta, 3)$-neighborhood of both $[x,z]$ and $[y,z]$.
Similarly, $v$ is contained in the $D_{\ref{stab-qg}}(\delta, 3)$-neighborhood of both $[\phi(x),\phi(z)]$ and 
$[\phi(y),\phi(z)]$. Therefore, using the proof of the claim in the proof of the Lemma \ref{barycen}(2), we have
$d_X(u,c_i)\leq 4.D_{\ref{stab-qg}}(\delta,3)$ and $d_Y(v,c^{'}_i)\leq 4.D_{\ref{stab-qg}}(\delta,3)$, for
$i=1,2,3$. 

Now for each $i$, $1\leq i\leq 3$ we have the following:\\
Since $\phi$ is a $k$-quasi-isometric embedding,  we have 
$d_Y(\phi(c_i), c^{'}_i)\leq D_{\ref{barycen}}(\delta, k,k)$ by Lemma $\ref{barycen}$(2). Thus, 
$d_Y(\phi(c_i), v)\leq d_Y(\phi(c_i),c^{'}_i)+ d_Y(c^{'}_i, v)\leq D_{\ref{stab-qg}}(\delta, 3)+D_{\ref{barycen}}(\delta, k,k)$. Again, using the fact that $\phi$
is a $k$-quasi-isometric embedding we have $d(\phi(c_i), \phi(u))\leq k.d_X(c_i,u)+k \leq k.D_{\ref{stab-qg}}(\delta, 3)+k$.
Thus 
$d_Y(\phi(u),v)\leq d_Y(\phi(u),\phi(c_i)) + d_Y(\phi(c_i), v)\leq 
k+ (k+1).D_{\ref{stab-qg}}(\delta, 3)+ D_{\ref{barycen}}(\delta, k,k)$.
Choosing $D_{\ref{qi-comm-proj}}= k+(k+1).D_{\ref{stab-qg}}(\delta, 3)+D_{\ref{barycen}}(\delta, k,k)$ completes the proof.
\end{proof}

To prove our main theorem, the following characterization of hyperbolicity
turns out to be very useful.

\begin{lemma} (Proposition 3.5 of \cite{hamenst-teich})\label{ori-hyp-lemma}
Suppose X is a geodesic metric space and
there is a collection of rectifiable curves 
$\{ c(x,y):x,y\in X\}$, one for each pair
of distinct points $x,y\in X$, and constants $D_1,D_2\geq 1$ such that for all $x,y,z\in X$ the
following hold: 
\begin{enumerate}
\item  If $d(x,y)\leq D_1$ then the length of the curve
$c(x,y)$ is less than or equal to $D_2$.
\item  If $x^{'},y^{'}\in c(x,y)$ then the Hausdorff
distance between $c(x^{'},y^{'})$ and the segment of $c(x,y)$
between $x^{'}$ and $y^{'}$ is bounded by $D_2$.
\item  The triangle formed by the curves joining any three points in $X$ is $D_2$-slim: $c(x,y) \subseteq N_{D_2}(c(x,z)\bigcup c(y,z))$.
\end{enumerate}

Then $X$ is $\delta_{\ref{ori-hyp-lemma}}=\delta_{\ref{ori-hyp-lemma}}(D_1,D_2)$-hyperbolic
and each of the curves $c(x,y)$ is a 
$K_{\ref{ori-hyp-lemma}}=K_{\ref{ori-hyp-lemma}}(D_1,D_2)$-quasi-geodesic in $X$.
\end{lemma} 

This lemma has the following straightforward corollary, which is a discrete
version of Lemma \ref{ori-hyp-lemma} (see Lemma \ref{coarse} for instance). A {\it discrete path} $c(x,y)$ will refer to
 a finite sequence of points. The {\it length} of a discrete path is the
sum of the distances between all  pairs of successive points in the discrete path. In this context, a triangle will refer to
the union of three discrete paths of the form $c(x,y)$, $c(y,z)$, $c(z,x)$.

\begin{cor}\label{hyp-lemma} Given $D, C_1,C_2 >0$ and $\Phi:\mathbb R^{+}\rightarrow \mathbb R^{+}$, there exist
$\delta_{\ref{hyp-lemma}} =\delta_{\ref{hyp-lemma}}(D,C_1,C_2,\Phi) \geq 0$ and $K_{\ref{hyp-lemma}}=K_{\ref{hyp-lemma}}(D,C_1,C_2,\Phi) \geq 1$
such that the following hold:\\
 Let $X$ be a geodesic metric space and let $X_1\subset X$ be a discrete set such that $X=N_D(X_1)$. \\
Further suppose that for all $x\neq y\in X_1$, there is a discrete path  $c(x,y)$ in $X_1$ connecting $x,y$
such that:
\begin{enumerate}
\item Distance between successive points of $c(x,y)$ is at most $C_1$.
\item If $d(x,y)\leq N $ then the number of points on the discrete path $c(x,y)$ 
      is at most $f(N)$.
\item If $x_1\neq y_1$ are two points of the discrete path $c(x,y)$, then the Hausdorff 
distance between the discrete path $c(x_1,y_1)$ and the discrete
subpath of $c(x,y)$ connecting $x_1,y_1$ is at most $C_2$.
\item For any three points $x,y,z\in X_1$, the triangle formed by the paths $c(x,y),c(y,z)$ and $c(x,z)$ is $C_2$-slim.
\end{enumerate}
 Then $X$ is $\delta_{\ref{hyp-lemma}}$-hyperbolic and the discrete paths are 
$K_{\ref{hyp-lemma}}$-quasi-geodesics in $X$.
\end{cor} 

\begin{proof} Let $\phi:X\rightarrow X_1$ be a map such that for all $x\in X$,  $d(x,\phi (x))\leq D$.

Given $x,y\in X$ define a curve $\beta(x,y)$
joining $x,y$ as follows: Let $\phi(x)=v_1,v_2,\ldots, v_n=\phi(y)$ be the set of successive points on $c(\phi(x),\phi(y))$. 
Join $x$ to $\phi(x)$, 
 $v_i$ to $v_{i+1}$, for $1\leq i \leq n-1$, and  $\phi(y)$ to $y$ by geodesics in $X$ to obtain $\beta(x,y)$. 

We check that the  curves
$\{\beta(x,y)\}$ satisfy the conditions of Lemma $\ref{ori-hyp-lemma}$: 
\begin{enumerate}
\item That the paths $\beta(x,y)$ are rectifiable follows from conditions 1 and 2.
\item We verify that condition 1 of Lemma \ref{ori-hyp-lemma} is satisfied with $D_1=1$.
Let $x,y\in X$ such that $d(x,y)\leq 1$. Then $d(\phi(x),\phi(y))\leq d(\phi(x),x)+d(x,y)+d(y,\phi(y))\leq 2D+1$.
Hence there are at most $\Phi(2D+1)$ points on the discrete path $c(\phi(x),\phi(y))$. 
Let $\phi(x)=v_1,v_2,\ldots, v_n=\phi(y)$ be the set of successive points on $c(\phi(x),\phi(y))$.
Then $n\leq \Phi(2D+1)$ and hence the length of the path 
$\beta(x,y)=d(x,\phi(x))+d(y,\phi(y))+\sum_{i=1}^{n}d(v_i,v_{i+1})\leq 2D+\Phi(2D+1)C_1$. Thus we may choose
$D_2\geq 2D+\Phi(2D+1)C_1$.
\item Conditions 2 and 3 of Lemma \ref{ori-hyp-lemma} follow from  conditions $3,4$.
In fact,  choosing $D_2\geq C_2+2C_1$ is enough for this. 
\end{enumerate}
Hence, choosing 
$D_2=max \{2D+\Phi(2D+1)C_1, C_2+2C_1 \}$ completes the proof. 
 \end{proof}

\subsection{Trees of hyperbolic and relatively hyperbolic metric spaces}\label{thrhms}
We refer  to \cite{farb-relhyp} for a detailed account of
 relative hyperbolicity.  We also refer to \cite{mahan-reeves}  for  the definitions and 
results of this subsection.

Suppose $(X,d)$ is a path metric space and let
$\HH = \{ H_\alpha\}$ be a collection of path-connected, uniformly separated subsets of $X$.
Then Farb \cite{farb-relhyp} defines the {\bf electric space} (or coned-off space) $\EE {(X, \HH )} $
corresponding to the
pair $(X,\HH )$ as a metric space which consists of $X$ and a
collection of vertices  $v_\alpha$ (one for each $H_\alpha \in \HH$)
such that each point of $H_\alpha$ is joined to (equivalently, coned off at)
$v_\alpha$ by an edge of length $\frac{1}{2}$. The sets $H_\alpha$ shall be
referred to as  {\it horosphere-like sets} and the vertices $v_{\alpha}$ as cone-points. Geodesics (resp. $P$-quasigeodesics)
in  $\EE {(X, \HH )} $ will be called {\it electric geodesics (resp. electric $P$-quasigeodesics)}.

When the collection $\HH = \{ H_\alpha\}$ is {\it not necessarily separated}, a slightly modified description is given in 
\cite{mahan-split} and \cite{mahan-reeves},  
where we attach a metric product $H_\alpha \times [0,1]$ to $X$, identifying $H_\alpha \times \{ 0 \}$ 
with $H_\alpha \subset X$ for each $H_\alpha \in \HH$,  and equip each
$H_\alpha \times \{ 1 \}$ with the zero metric. We shall call this construction {\it electrocution}. 

Let $i: X \rightarrow \EE {(X, \HH )}$ denote the natural inclusion of spaces. Then
for a path $\gamma \subset X$, the path $i ( \gamma )$ lies in $\EE {(X, \HH )}$. Replacing maximal subsegments
$[a,b]$ of $i (\gamma )$  lying in a particular $H_\alpha$ by a path that goes from
$a$ to $v_\alpha$ and then from $v_\alpha$ to $b$, and repeating this for every $H_\alpha$ that 
$i ( \gamma )$ meets we obtain a new path $\hat{\gamma}$. If $\hat{\gamma}$ is an electric geodesic (resp. electric  $P$-quasigeodesic), 
$\gamma$ is called a {\em relative geodesic} (resp.
{\em relative $P$-quasigeodesic}). A geodesic, or quasigeodesic, or more generally a path  $\gamma$ is said to be
without backtracking if for any  horosphere-like set $H_\alpha$, $\gamma$
does not return to $H_\alpha$ after leaving it.

\begin{defn} \cite{farb-relhyp} \label{brp} {\rm
Relative  $P$-quasigeodesics in  
$(X,\HH )$ are said to satisfy {\bf bounded region penetration} if, for any two  relative
    $P$-quasigeodesics without backtracking
$\beta$, $\gamma$, 
   joining $x, y \in X$,
 there exists $B = B(P) \geq 0$ such that \\
{\bf Similar Intersection Patterns 1:}  if
  precisely one of $\{ \beta , \gamma \}$ meets 
 a  horosphere-like set $H_\alpha$, 
then the length (measured in the intrinsic path-metric
  on  $H_\alpha$) from the first (entry) point
  to the last 
  (exit) point (of the relevant path) is at most $B$. \\
 {\bf Similar Intersection Patterns 2:}  if
 both $\{ \beta , \gamma \}$ meet some  $H_\alpha $
 then the length (measured in the intrinsic path-metric
  on  $H_\alpha$) from the entry point of
 $\beta$ to that of $\gamma$ is at most $B$; similarly for exit points.} \end{defn}

$X$ is {\it strongly hyperbolic relative to the collection $\HH$} if \\
a) $\EE {(X, \HH )}$ is hyperbolic, and\\
b) Relative quasigeodesics satisfy the bounded region penetration property.

The next  notion is based on Bestvina-Feighn's seminal work  \cite{BF}. The notions we use here are the adaptations used in 
\cite{mahan-reeves}.
\begin{defn} \label{tree} A geodesic
 metric space $(X,d)$ equipped with a map
$P: X \rightarrow T$ to a  simplicial tree $T$ is said to be a  tree of geodesic
 metric spaces satisfying
the q(uasi) i(sometrically) embedded condition if  there exist $ \epsilon \geq 0$ and $K \geq 1$ satisfying the 
following: \\
1) For all vertices $v\in{T}$, 
$X_v = P^{-1}(v) \subset X$ with the induced path metric $d_{X_v}$ is 
 a geodesic metric space. Further, the
inclusions ${i_v}:{X_v}\rightarrow{X}$ 
are uniformly proper. \\
2) Let $e$ be an edge of $T$ with initial and final vertices $v_1$ and
$v_2$ respectively.
Let $X_e$ be the pre-image under $P$ of the mid-point of  $e$.  
 There exist continuous maps ${f_e}:{X_e}{\times}[0,1]\rightarrow{X}$, such that
$f_e{|}_{{X_e}{\times}(0,1)}$ is an isometry onto the pre-image of the
interior of $e$ equipped with the path metric. Further, $f_e$ is fiber-preserving,
i.e. projection to the second co-ordinate in ${X_e}{\times}[0,1]$ corresponds via $f_e$
to projection to the tree $P: X \rightarrow T$.\\
3) ${f_e}|_{{X_e}{\times}\{{0}\}}$ and 
${f_e}|_{{X_e}{\times}\{{1}\}}$ are $(K,{\epsilon})$-quasi-isometric
embeddings into $X_{v_1}$ and $X_{v_2}$ respectively.
${f_e}|_{{X_e}{\times}\{{0}\}}$ and 
${f_e}|_{{X_e}{\times}\{{1}\}}$ will occasionally be referred to as
$f_{e,v_1}$ and $f_{e,v_2}$ respectively. 
\end{defn}

$X_v, X_e$ are referred to as vertex and edge spaces respectively.
A tree of spaces as in Definition \ref{tree} above is said to be a
{\it tree of hyperbolic metric spaces}, if there exists $\delta \geq 0$
such that $X_v, X_e$ are all $\delta$-hyperbolic for all vertices $v$ and edges $e$ of $T$.

\begin{defn} \label{rht} A tree $P: X \rightarrow T$ of geodesic
 metric spaces is said to be a  tree of relatively hyperbolic metric spaces
if in addition to the conditions of Definition $\ref{tree}$, we have the following:\\
4) Each vertex (or edge) space $X_v$ (or $X_e$) is strongly hyperbolic relative to a collection $\HH_v$ (or $\HH_e$)\\
5)  the maps $f_{e,v_i}$ above ($i = 1, 2$) are {\bf
  strictly type-preserving}, i.e. $f_{e,v_i}^{-1}(H_{v_i,\alpha})$, $i =
  1, 2$ (for
  any $H_{v_i,\alpha}\in \HH_{v_i}$)
 is
  either empty or some $H_{e,\beta}\in \HH_{e}$. Also, for all 
$H_{e,\beta}\in \HH_{e}$,  and any end-point $v$ of $e$, there exists
$H_{v,\alpha}$, such that $f_{e,v} ( H_{e,\beta}) \subset H_{v,\alpha}$.\\
The sets $H_{v,\alpha}$ and
$H_{e,\alpha}$ will be referred to as
{\bf horosphere-like vertex sets} and {\bf horosphere-like  edge sets} respectively.\\
6) There exists $\delta > 0$ such that each $\EE (X_v, \HH_v )$ is $\delta$-hyperbolic.\\
{\rm Given the tree of spaces with vertex spaces $X_v$ and edge spaces $X_e$, there exists a naturally associated  tree of spaces
with vertex spaces 
$\EE (X_v, {\HH}_v)$ and edge spaces 
$\EE (X_e, {\HH}_e)$, obtained by simply coning off the respective horosphere like sets.
Condition (5) above  ensures that we have natural inclusion maps of edge spaces
$\EE (X_e, {\HH}_e)\times \{ i \}$ ($i=0,1$) into adjacent vertex spaces $\EE (X_v, {\HH}_v)$. These  maps are referred to as {\it induced maps}. 
The 
resulting tree of coned-off spaces will be called the {\bf induced
tree of coned-off spaces} and will be denoted as $\hhat{X}$. }\\
7) The induced maps  of the coned-off edge spaces into the
  coned-off vertex spaces $\hhat{f_{e,v_i}} : \EE ({X_e}, \HH_e ) \rightarrow
  \EE ({X_{v_i}}, \HH_{v_i})$ ($i = 1, 2$) are uniform quasi-isometries. This is called the
 {\bf qi-preserving electrocution condition}.
 \end{defn}

$d_v$ and $d_e$ will denote path metrics on $X_v$ and $X_e$ respectively.
$i_v$, $i_e$ will denote inclusion of $X_v$, $X_e$ respectively into
$X$.

Note that
the first clause of Condition (5) above ensures that for any vertex $v_i$ and edge $e$ incident on $v_i$, and for any horosphere like set
 $H_{v_i,\alpha}$ in $X_{v_i}$, at most one horosphere like set $H_{e,\beta}$ of $X_e$ is mapped by $f_{e,v_i}$ into $H_{v_i,\alpha}$.
Also, the second clause of Condition (5) above ensures that for any such horosphere like set $H_{e,\beta}$ of $X_e$, $f_{e,v_i}$  maps  $H_{e,\beta}$
into some horosphere like set
 $H_{v_i,\alpha}$ in $X_{v_i}$.

\begin{defn} 
The {\bf cone locus} of the
induced tree (T) of coned-off spaces, $\hhat{X}$, is
the graph whose vertex set $\VV$ consists
of horosphere like vertex sets and  edge set $\EE$
consists of  horosphere like edge sets such that
an edge $H_{e,\beta}\in \HH_{e} \subset \EE$
is incident on a vertex $H_{v,\alpha}\in \HH_{v} \subset \VV$ iff $f_{e,v} ( H_{e,\beta}) \subset H_{v,\alpha}$.\\
 A connected component of the cone-locus 
will be called a {\bf maximal cone-subtree}. The collection
of maximal cone-subtrees will be denoted by $\TT$ and elements
of $\TT$ will be denoted as $T_\alpha$. \\
 For each maximal cone-subtree $T_\alpha$, we define
the associated {\bf maximal cone-subtree of horosphere-like spaces} $C_\alpha$
to be the tree of metric spaces whose vertex and edge spaces
are the horosphere like vertex and edge sets $H_{v,\alpha}$, $H_{e,\alpha}$,
$v\in \VV(T_\alpha)$, $e\in \EE(T_{\alpha})$, along with the restrictions
of the maps $f_e$ to $H_{e,\alpha}\times \{0,1\}$.
The collection of $C_\alpha$'s will be denoted as $\C$. 
\end{defn}

The next definition is based on \cite{BF} again.
\begin{defn} 
A disk $f : [-m,m]{\times}{I} \rightarrow 
X$ is a {\bf hallway} of length $2m$ if it satisfies:\\
1) $f^{-1} ({\cup}{X_e} : e \in {\rm Edge} (T)) = \{-m,  \cdots , m \}{\times}
I$, where $ {\rm Edge} (T)$ denotes the collection of mid-points of the edge-set of $T$.\\
2) $f$ maps $i{\times}I$ to a geodesic in  $X_e$ for some edge
space.\\
3) $f$ is transverse, relative to condition (1), to $\cup_e X_e$, i.e. for all $i \in \{-m,  \cdots , m \}$,
$f|_{{B(i, \frac{1}{4})} \times \{t\}}$ is an isometric embedding for all $t \in I$. Here $B(i, \frac{1}{4})$ denotes
the $ \frac{1}{4}$ neighborhood of $i$ in $[-m,m]$.
\end{defn}

Condition (3) above is the adaptation in our context of Condition (2) of \cite{BF} p.87 and simply says that a hallway transversely cuts across
the collection of edge spaces.

\begin{defn} \label{hway}
 A hallway $f : [-m,m]{\times}{I} \rightarrow X$ is {\bf $\rho$-thin} if 
$d({f(i,t)},{f({i+1},t)}) \leq \rho$ for all $i\in \{-m,  \cdots , m \}$ and $ t \in I$.\\
 A hallway is {\bf $\lambda$-hyperbolic} if 
\[\lambda l(f(\{ 0 \} \times I)) \leq \mbox{ max} \{ l(f(\{ -m \} \times I)),
l(f(\{ m \} \times I))\}\]
where $l(\sigma )$ denotes the length of the path $\sigma$.\\
The {\bf girth} of the hallway is the quantity $l(f(\{ 0 \} \times I))$.\\
 A hallway is {\bf essential} if the edge path in $T$ 
resulting from projecting $f( [-m,m]{\times}{I})$ onto $T$ does not backtrack,
and is therefore a geodesic segment in the tree $T$.\\
 {\bf Hallways flare condition:}
The tree of spaces, $X$, is said to satisfy the  hallways flare
condition if there are numbers $\lambda > 1$ and $m \geq 1$ such that
for all $\rho$ there is a constant $H=H(\rho )$ such that  any
$\rho$-thin essential hallway of length $2m$ and girth at least $H$ is
$\lambda$-hyperbolic.
\end{defn}

\begin{defn} \label{cbh}
 An essential  hallway of length $2m$ is {\bf cone-bounded} if\\
a) $f(i \times {\partial I}) = f(i \times \{0, 1\})$ lies in the cone-locus for
$i = \{ -m, \cdots , m\}$.\\
b) $f(i \times \{0\})$ and $f(i \times \{1\})$ lie in different components of the cone-locus.\\
 The tree of spaces, $X$, is said to satisfy the 
{\bf cone-bounded hallways strictly flaring condition} 
 if for all $\rho > 0$, there exists $ \lambda > 1$ and $m \geq 1$ such that
any cone-bounded $\rho-$thin essential hallway of length $2m$  is $\lambda$-hyperbolic.
\end{defn}

Note that the last condition requires all cone-bounded $\rho-$thin essential hallways  to flare 
(not just those of girth at least  $H$ as in Definition \ref{hway}).
The following theorem is one of the main results of \cite{mahan-reeves}.
\begin{theorem}\label{weakcombin} \cite{mahan-reeves}
Let $X$ be a tree ($T$) of strongly relatively hyperbolic spaces
satisfying 
\begin{enumerate}
\item the qi-embedded condition.
\item  the strictly type-preserving
condition.
\item the qi-preserving electrocution condition.
\item the induced tree of coned-off spaces
satisfies the  hallways flare condition.
\item the 
cone-bounded hallways strictly flaring condition.
\end{enumerate}
Then $X$  is  (strongly) hyperbolic relative to the family $\C$ of maximal
  cone-subtrees of horosphere-like spaces.
\end{theorem}

\noindent {\bf Note:} In \cite{mahan-reeves} the definition of cone-bounded hallways does not include Condition (b) of Definition \ref{cbh}.
However the proof there (cf. Proposition 4.4 of \cite{mahan-reeves}) is  enough to give Theorem \ref{weakcombin} under the
(weaker) condition that only those cone-bounded hallways (in the terminology of \cite{mahan-reeves})
that additionally satisfy restriction (b) strictly flare.

\begin{defn}{\bf Partial Electrocution:}
Let $(X, \HH , \GG , \LL )$ be an ordered quadruple, where 
\begin{enumerate}
\item $X$ is a geodesic metric space, 
\item  $\HH = \{ H_\alpha \}$ is a collection of uniformly separated subsets of $X$, 
\item $\LL = \{ L_\alpha \}$ is a collection of $\delta-$ hyperbolic metric spaces for some $\delta \geq 0$, 
\item $\GG = \{ g_\alpha : H_\alpha \rightarrow L_\alpha \} $ are maps.
\end{enumerate}

Further suppose that there exist $ K\geq 1$ 
such that the
following hold:

\begin{enumerate}
\item $X$ is strongly hyperbolic relative to the collection $\HH$ of subsets
$H_\alpha$. 
\item Each $g_\alpha$  is $K-$ coarsely Lipschitz, i.e.
$d_{L_\alpha} (g_\alpha (x), g_\alpha (y)) \leq Kd_{H_\alpha}(x,y)
+ K $ for all $x, y \in H_\alpha$. 
\end{enumerate}
 The {\bf partially electrocuted space} or
{\em partially coned off space} corresponding to $(X, \HH , \GG , \LL)$ 
is the quotient metric space $(\hat{X},d_{pel})$ obtained from $X$ by attaching the metric
mapping cylinders for the maps
 $g_{\alpha} : H_\alpha \rightarrow L_\alpha$, where $d_{pel}$ denotes the resulting partially electrocuted metric. (The metric
mapping cylinder for a map $g:A \rightarrow B$ is the quotient metric space obtained as a quotient space
of the disjoint union  of the metric product  $A \times [0,1]$
and $B$, by identifying $(a,1) \in A \times \{ 1\} $ with $g(a) \in B$.)
\end{defn}

\begin{lemma}\label{pel}\cite{mahan-reeves} (see also Lemmas 1.20. 1.21 of \cite{mahan-pal})
$(\hat{X},d_{pel})$ is a hyperbolic metric space and the sets $L_\alpha$
are uniformly quasiconvex in $\hat{X}$.
\end{lemma}

We end this subsection with a proposition,  a special case of which is due to Hamenstadt
\cite{hamenst-word}, where the tree is taken to be $\mathbb R$ with vertex set $\mathbb Z$. We  give  a different proof
below as  our proof applies in a more general context. 

\begin{prop} \label{hyp-tree} Given  $K \geq 1$ and  $\delta, D > 0$, there exist $\delta^{'}, k^{'} \geq 0$ such that
the following holds.\\
Suppose $Y$ is a tree of $\delta$-hyperbolic metric spaces satisfying the
$K$-qi embedded condition such that the images of the edge spaces
in the vertex spaces are   mutually $D$-cobounded.
Then $Y$ is a $\delta^{'}$-hyperbolic metric space and
all the vertex spaces and edge spaces are $k^{'}$-quasiconvex
in $Y$.
\end{prop}

\begin{proof} First of all we note that by \cite{bowditch-relhyp} (Section 7, esp. Proposition 7.4,
Lemma 7.5, Proposition 7.12; see also Lemma 3.4 of \cite{brahma-ibdd})
the vertex spaces are strongly hyperbolic relative to  the images of edge spaces. Hence $Y$ can be thought of as a tree
of relatively hyperbolic metric spaces whose horosphere like edge sets and
vertex sets are respectively the whole of the edge spaces and the 
images of the edge spaces in the vertex spaces respectively. 
Hence  conditions (1)-(2) of Theorem $\ref{weakcombin}$
are satisfied in this case. 

Edge spaces after electrocution become points. Vertex spaces after electrocution become hyperbolic metric spaces. Inclusion of points into 
spaces being trivially qi-embeddings, condition (3) of Theorem $\ref{weakcombin}$
is satisfied.

Next any essential hallway of length greater than two in $\hat Y$, the induced tree of coned off spaces,
 must have   girth at most one. This is because
 all edge spaces have diameter one after electrocution.
Hence Condition (4) of Theorem $\ref{weakcombin}$
is trivially satisfied by choosing the threshold value  $H$ of the girth to be  2.

Finally, since cone-bounded hallways must have length two or more, it follows that in the present situation
cone-bounded hallways do not exist due to Condition (b) in Definition \ref{cbh}
 and the fact that entire edge spaces are part of the cone locus.
 Hence Condition (5) of Theorem $\ref{weakcombin}$
is vacuously satisfied.

Finally, the family $\mathcal C$ of maximal cone-subtrees of horosphere-like spaces
are precisely the edge spaces.

Hence $Y$ is strongly  hyperbolic
relative to the edge spaces. 

The edge spaces are uniformly hyperbolic with respect to the induced
length metric from $Y$. Hence, by Lemma $\ref{pel}$, we see that when the
maps $g_{\alpha}$ are taken to be identity maps of the edge spaces, the
partially electrocuted space is hyperbolic. This space is clearly
quasi-isometric to $Y$. Hence the result. \end{proof}

As an application of this proposition we have the following  corollary which can be thought of as a `discrete' or `graph' 
 version of Proposition
\ref{hyp-tree}.

\begin{cor}\label{hamenstadt}  Given $\delta,D,D_1,K \geq 1$, there exists $D_{\ref{hamenstadt}}$ such that the following holds.\\
Suppose $X$ is a connected graph and $X_i$, $0 \leq i\leq n$, are connected subgraphs with $X = \cup_i X_i$
such that the following conditions hold. \\
$(1)$ All the spaces $X_i$ are $\delta$-hyperbolic with respect to the path metric  induced from $X$.\\
$(2)$ $X_i\cap X_j\neq \emptyset$ iff $|i-j|\leq 1$.\\
$(3)$ For all $i$, $X_i\cap X_{i+1}$ contains a connected subgraph $Y_i$
and is contained in the $D$-neighborhood of $Y_i$ in (the path-metric on) $X_i$ as well as $X_{i+1}$.\\
$(4)$ The inclusions $Y_i\hookrightarrow X_i$, $Y_i\hookrightarrow X_{i+1}$ are $K$-quasi-isometric embeddings. Also the inclusions
$Y_i \hookrightarrow X$, $1\leq i\leq n-1$ are uniformly metrically proper as measured by $g$, for some map 
$g: \mathbb{R}^{+} \rightarrow \mathbb{R}^{+}$.\\
$(5)$ $Y_i$ and $Y_{i+1}$ are $D_1$-cobounded in $X_{i+1}$.

Then the space $X$ is $D_{\ref{hamenstadt}}(= D_{\ref{hamenstadt}}(\delta,D,D_1,K))$-hyperbolic.
\end{cor}

\begin{proof} First construct a new graph $X^{'}$ with the same vertex set as $X$ and edge-set
$\mathcal{E}(X^{'}) = \mathcal{E}(X) \bigcup
\{ \{u,v\} : u\neq v \in \mathcal{V}(X_i) \, \mbox{for  some} \, i; d_X(u,v)\leq D \}$. Note that $X$ is a subgraph of $X^{'}$.
By Lemma \ref{coarse} (2), $X$ is quasi-isometric to $X^{'}$.

Let us denote by $X^{'}_i$ the subgraph of $X^{'}$ with the same vertex set as $X_i$ (i.e. $\mathcal{V}(X^{'}_i) = \mathcal{V}(X_i)$)
and with edge-set 
$\mathcal{E}(X^{'}_i) = \{ \{u,v\} : u\neq v \in  \mathcal{V}(X_i); d_X(u,v)\leq D \}$. Then $X^{'} = \cup_i X^{'}_i$.
Let $Y^{'}_i:=X^{'}_i\cap X^{'}_{i+1}$. Note that $Y^{'}_i$ is a connected graph by Condition (3).

We show now that  $Y_i$ is quasi-isometric to  $Y^{'}_i$. First, since the inclusion
$Y_i \hookrightarrow X$ is uniformly proper, it follows that the inclusion $Y_i \hookrightarrow X^{'}$ is also uniformly proper (since 
$X$ is quasi-isometric to $X^{'}$). Hence the inclusion $Y_i \hookrightarrow Y^{'}_i $ is also uniformly proper. But the vertex set of 
$Y^{'}_i $ is contained in a $D$-neighborhood of $Y_i$ in $X$. Hence every vertex of $Y^{'}_i $ is connected by an edge to a vertex of $Y_i$
by construction of $X^{'}$. It follows that the inclusion $Y_i \hookrightarrow Y^{'}_i $ is a 
uniform (independent of $i$) quasi-isometry.

Next we claim that the inclusion of  $X_i$ into  $X^{'}_i$ is a uniform
 (independent of $i$) quasi-isometry. Note that $X_i$ and  $X^{'}_i$ have the same vertex set. Also the inclusion
$X_i \hookrightarrow X^{'}_i $ is 1-Lipschitz. Hence it suffices to
show that when two vertices in $X_i$ are at a distance of at most $D$ in $X$
then they are not too far away in the (path) metric on $X_i$. Let $\gamma$ be a geodesic in $X$ joining two points
of $X_i$ that are  at a distance of at most $D$ from each other. If $\gamma$ contains a (maximal) subsegment $\gamma_0 = [a_0,b_0]$
lying outside $X_i$ then $a_0, b_0$ must be distinct 
vertices of $Y^{'}_i$ or $Y^{'}_{i+1}$. Without loss of generality, suppose 
$a_0, b_0\in \mathcal{V}(Y^{'}_i)$. Hence there exist vertices $a, b\in \mathcal{V}(Y_i)$ such that $d_X(a, a_0) \leq D$ and 
$d_X(b, b_0) \leq D$. It follows that $d_{Y_i}(a,b) \leq g(3D)$ and hence $d_{X_i}(a,b) \leq g(3D)$. The claim follows.

Hence there exist $\delta^{'}, K^{'}, D^{'}_1$ such that $X^{'}_i$ is $\delta^{'}$ hyperbolic;
the inclusion maps $Y^{'}_i\hookrightarrow X^{'}_i$, $Y^{'}_i\hookrightarrow X^{'}_{i+1}$
are $K^{'}$-qi embeddings; and  $Y^{'}_i$ and $Y^{'}_{i+1}$ are $D^{'}_1$-cobounded in $X^{'}_{i+1}$ for all $i$.

Now we construct a tree of metric spaces $X_T$ quasi-isometric to $X^{'}$ (and hence to $X$)
where the underlying tree $T$ is the interval $[0,n]$ with vertices the integer points $\{ 0, \cdots , n \}$
and edge set $\{ [i, i+1]: i = 0, \cdots , n-1 \}$. 
For each $Y^{'}_i$ construct $Y^{'}_i \times [0,1]$. $X_T$ is constructed as an identification space from $\cup_i (Y^{'}_i \times [0,1])
\bigcup \cup_i X^{'}_i$ as follows.
For all $i=0 \cdots n-1$ and $x\in \mathcal V(Y^{'}_i)$, identify ${x \times \{ 0 \}}$ with $x \in X^{'}_i$ and
${x \times \{ 1 \}}$ with $x \in X^{'}_{i+1}$. Extend the identification 
linearly over edges of $Y^{'}_i$  for all $i$ to obtain the required tree of metric spaces $X_T$.
We observe that $X^{'}$  (and hence $X $) is quasi-isometric to the tree of metric spaces $X_T$, which in turn
 satisfies all the conditions of Proposition \ref{hyp-tree}. The Corollary follows. \end{proof}

A remark is in order here. Note that in the hypothesis we have not required that each $X_i$ contains all the edges of $X$ between any two of its
vertices. But it is always true that $X_{i-1} \cup X_i \cup X_{i+1}$  contains all the edges of $X$ between any two vertices of $X_i$
since $X = \cup_i X_i$. However, once
we pass to $X^{'}$ this is  no longer an issue because in the construction of $X^{'}_i$ from $X_i$ these edges get introduced in any case (as $D\geq 1$).
Hence each $X^{'}_i$ contains all the edges of $X^{'}$ between any two of its
vertices.


\section{QI Sections}

\subsection{Existence of qi sections}
The main result (Proposition $\ref{existence-qi-section}$) of this subsection is that qi sections
 exist for a large class of examples of metric graph bundles $p:X\rightarrow B$. 
This is {\em the} crucial ingredient in the proof of our main theorem \ref{combthm}. The basic idea of the
proof of Proposition $\ref{existence-qi-section}$ runs as follows: \\
We assume that the horizontal spaces
$F_b$, $b\in \mathcal{ V}(B)$ in our metric graph bundle are uniformly hyperbolic and the barycenter maps $\phi_b: \partial^3 F_b \rightarrow F_b$,
sending a triple of distinct points on the boundary $\partial F_b$ to the barycenter of an ideal 
triangle with the three points as vertices, are uniformly  coarsely surjective.
For simplicity, suppose we have $x\in F_v$, $v\in B$ and there is a triple $\xi=(\xi_1,\xi_2,\xi_3)$ such that $ \phi_v(\xi)=x$. 
Fix one such triple.  `Flow' this triple to the boundaries of all other horizontal spaces $F_w$ by maps induced by
quasi-isometries $f_{vw}:F_v\rightarrow F_w$. 
These maps are coarsely unique and are naturally associated to any given metric graph bundle. 
Let $\partial(f_{vw})$ denote the boundary value of $f_{vw}$. 
Consider the barycenters of the ideal  triangles formed by  the flowed triples 
$(\partial(f_{vw})\xi_1,\partial(f_{vw})\xi_2, \partial(f_{vw})\xi_3)$. The collection of all these barycenters (as $w$ ranges over 
$\mathcal V(B)$)
is then a section through $x$.
The proof that this is indeed a qi section hinges on the fact that for any three points $u,v,w\in \mathcal V(B)$, the quasi-isometries $f_{uv}$ and
$f_{wv}\circ f_{uw}$ are at a bounded distance {\it (depending on $u, v, w$)} from each other, and hence the
induced  boundary maps satisfy the equality
$\partial(f_{uv})= \partial (f_{wv}) \circ \partial (f_{uw})$. 

As an application of the proof of this result we recover an important lemma due to Mosher (see Theorem \ref{qi-mosher}).
It should be noted that though a basic 
ingredient for both Mosher's proof and ours  is the notion of a `barycenter',  we
do not have a group action on the boundaries of  fiber spaces in our context. Mosher extracts his qi-section
from an action of the whole group on the boundary of the normal subgroup. 

\begin{defn}{\bf Sequential Boundary}(See Chapter $4$,\cite{Shortetal})
Let $X$ be a $\delta$-hyperbolic metric space.
A sequence of points $\{x_n\}$ in $X$ is said
to converge to infinity, written $x_n\rightarrow \infty$, if for some (and hence 
any) point $p\in X$, $lim_{m,n\rightarrow \infty}(x_m.x_n)_p=\infty$.

Define an equivalence relation on the set of 
all sequences in $X$ converging to infinity, by setting
$\{x_n\}\sim \{y_n\}$ iff $lim_{n\rightarrow \infty}(x_n.y_n)_p=\infty$.
The set of all equivalence classes $\{[\{x_n\}]:x_n\rightarrow \infty\}$ will be denoted by 
$\partial X$ and will be referred to as the sequential boundary of $X$ or simply the {\em boundary of} $X$ . 
\end{defn}

Suppose $\{x_n\}$ is a sequence of points in $X$ and $x_n\rightarrow \infty$.
We shall write $x_n\rightarrow \xi\in \partial X$ to mean that $\xi=[\{x_n\}]$. The boundary
$\partial X$ comes equipped with a natural `visual' topology \cite{GhH}.

Suppose $f:X\rightarrow Y$ is a $(k,\epsilon)$-quasi-isometric embedding of hyperbolic metric spaces and $\xi=[\{x_n\}]\in \partial X$.
Then $f(x_n)\rightarrow \infty$. Setting $\partial(f)(\xi):=[\{f(x_n)\}]$ gives a well
defined map $\partial(f):\partial X\rightarrow \partial Y$. The next lemma collects together standard properties of such maps.

\begin{lemma}\label{bdry-elementary}
1) If $I_X:X\rightarrow X$ is the identity map then  $\partial (I_X)$
is the identity map on the sequential boundary of $X$.\\
2) If $f:X\rightarrow Y$ and $g:Y\rightarrow Z$ are two $(k,\epsilon)$-quasi-isometric
embeddings then $\partial(g\circ f)=\partial(g)\circ \partial(f)$.\\
3) If $f,g:X\rightarrow Y$ are two $(k,\epsilon)$-quasi-isometric embeddings
such that one has $sup_{x\in X}d(f(x),g(x))<\infty$ then $\partial(f)=\partial(g)$.\\
4) If $f:X\rightarrow Y$ is a quasi-isometry then $\partial(f):\partial X\rightarrow \partial Y$ is a homeomorphism.
\end{lemma}

The next lemma is a  consequence of the stability of quasi-geodesics (Lemma $\ref{stab-qg}$)
in hyperbolic metric spaces.
\begin{lemma}\label{prev}
Let $X$ be a $\delta$-hyperbolic metric space and let $\gamma:[0,\infty)\rightarrow X$
be a $(K,\epsilon)$-quasi geodesic ray. Let $\{t_n\}$ be any sequence of non-negative
real numbers tending to $\infty$; then $\gamma(t_n)\rightarrow \infty$ and the point 
of $\partial X$ represented by $\{\gamma(t_n)\}$ is independent of the sequence $\{t_n\}$.
\end{lemma}
The point of $\partial X$ determined by a quasi-geodesic ray $\gamma$
will be denoted by $\gamma(\infty)$. The next lemma constructs quasigeodesic rays joining points in $X$ to points in $\partial X$
as well as bi-infinite quasigeodesics joining pairs of points in $\partial X$. While this is standard for proper $X$ \cite{GhH},  ready
references for arbitrary (non-proper) $X$ are a bit difficult to come by, and we include a proof for completeness.

\begin{lemma}\label{qi-geod-line}
For any $\delta \geq 0$ there is a constant $K=K_{\ref{qi-geod-line}}(\delta)$ such that
the following holds:

Suppose $X$ is a $\delta$-hyperbolic metric space. 
\begin{enumerate}
\item Given any point $\xi\in \partial X$ and $p\in X$ there is a
$K$-quasi-geodesic ray $\gamma:[0,\infty)\rightarrow X$ of $X$ with
$\gamma(0)=p$ and $\gamma(\infty)=\xi$.

\item Given two points $\xi_1\neq \xi_2\in \partial X$ there is a $K$-quasi-geodesic
line $\alpha:\mathbb R \rightarrow X$ with $\alpha(-\infty)=\xi_1$
and $\alpha(\infty)=\xi_2$. 
\end{enumerate}
\end{lemma}

\noindent {\bf Terminology:}
Any quasi-geodesic ray as in $(1)$ of the above lemma will be referred to as a quasi-geodesic ray
joining the points $p$ and $\xi$. Similarly any quasi-geodesic as in $(2)$ of the above lemma
will be referred to as a quasi-geodesic line joining the points $\xi_1$ and $\xi_2$.

\smallskip

\noindent {\em Proof of Lemma $\ref{qi-geod-line}$:} $(1).$ We shall inductively construct a suitable sequence of points
$\{p_n\}$ such that $p_n\rightarrow \xi$ and finally show that the union $\cup[p_n,p_{n+1}]$, of the geodesic segments $[p_n,p_{n+1}]$,
is a uniform quasi-geodesic. 
Suppose $x_n\rightarrow \xi$, $x_n\in X$, for all $n$. Fix $N\geq 1$ and let $p_0=p$.
Since $x_n\rightarrow \infty$ we can find a positive integer $n_1\in \mathbb N$ such
that $(x_i.x_j)_{p_0}\geq N$ for all $i,j\geq n_1$. Let $[p_0,x_{n_1}]$ be a geodesic joining 
$p_0$ and $x_{n_1}$. Choose $p_1\in [p_0,x_{n_1}]$ such that $d(p_0,p_1)=N$. Now suppose $p_l$ has
been constructed. To construct $p_{l+1}$, let $n_{l+1}\geq max\{ n_k:1\leq k\leq l\}$ be an integer
such that $(x_i.x_j)_{p_l}\geq (l+1)N$ for all $i,j\geq n_{l+1}$. Choose $p_{l+1}\in[p_l,x_{n_{l+1}}]$
such that $d(p_l,p_{l+1})=(l+1)N$. Now, let $\alpha_N$ be the arc length parametrization of the
concatenation of the geodesic segments $[p_i,p_{i+1}]$, $i\in {\mathbb Z}^+$.

\noindent {\bf Claim:} For $N> max\{ 7\delta+1, \frac{1}{3} L_{\ref{local-global-qg}}(\delta,1,42\delta) \}$, $\alpha_N$ is a 
$\lambda_{\ref{local-global-qg}}(\delta,1,42\delta)$-quasi-geodesic.

\begin{center}
\includegraphics[height=30mm]{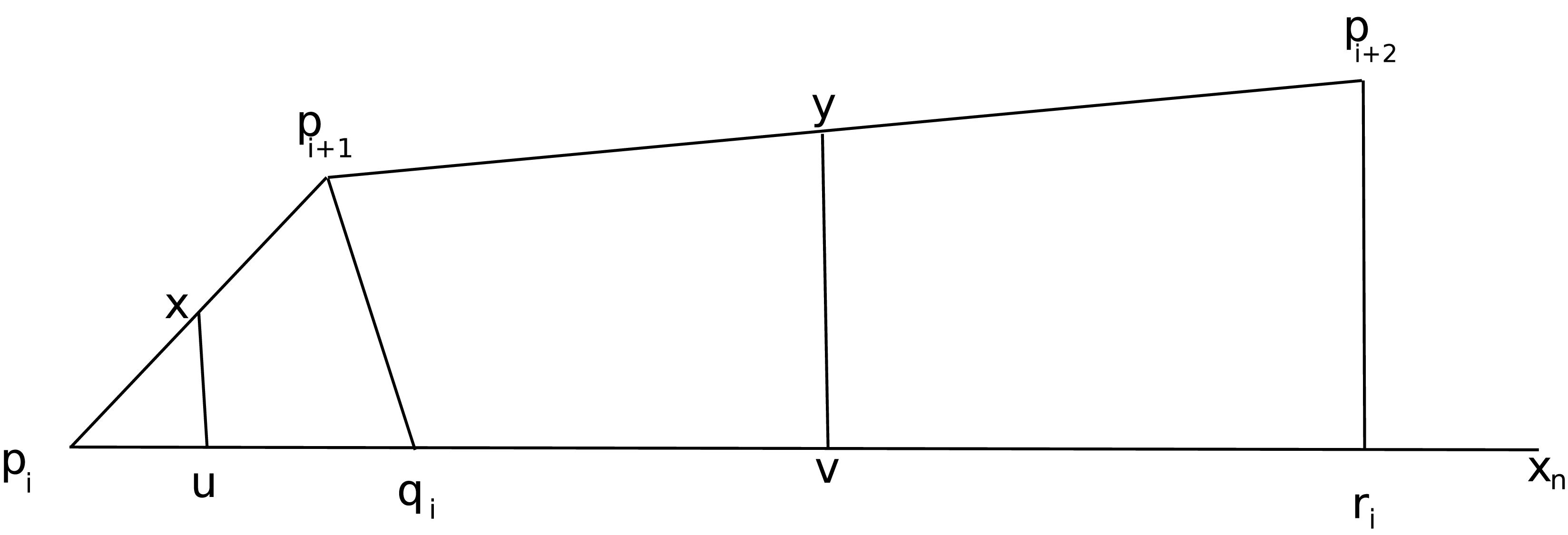}
\end{center}

First we show that $[p_i,p_{i+1}]\cup [p_{i+1},p_{i+2}]$ is a uniform quasi-geodesic for each $i$.
Let $n>n_{i+2}$. Join $p_i$ with $x_n$. Since $(x_n.x_{n_{i+1}})_{p_i}\geq (i+1)N$ and triangles in $X$
are $6\delta$-thin by Lemma $\ref{hyp-defn}(2)$, we can find a point $q_i\in [p_i,x_n]$ such that $d(p_i,q_i)=(i+1)N$ and
$d(p_{i+1},q_i)\leq 6\delta$. Similarly there is a point $q_{i+1}\in[p_{i+1},x_n]$ such that
$d(p_{i+1},q_{i+1})=(i+2)N$ and $ d(p_{i+2},q_{i+1})\leq 6\delta$. 

Consider the triangle
$\Delta p_ip_{i+1}x_n$. The point $q_{i+1}\in [p_{i+1},x_n]$ is contained in a $\delta$-neighborhood
of $[p_i,p_{i+1}]\cup [p_{i},x_n]$. Hence there exists $r_i\in [p_i,x_n]\cup [p_i,p_{i+1}]$
such that $d(r_i,q_{i+1})\leq \delta$. Since  $d(q_{i+1}, p_{i+1}) = (i+2)N$,
 it follows from the triangle inequality that $d(r_i,q_i) \geq d(q_{i+1}, p_{i+1})-d(r_i,q_{i+1}) - d(p_{i+1},q_i)
 \geq (i+2)N- \delta - 6\delta$. Again,
since $N> 7\delta+1$, it follows that $d(r_i,q_i) > (i+1) N +1$ and hence
 $r_i\not \in [p_i,q_i](\subset [p_i,x_n])$.

Next, we note that $r_i\not \in [p_i,p_{i+1}]$. Else suppose $r_i\in [p_i,p_{i+1}]$. Then 
$(i+1)N=d(p_i,p_{i+1})\geq d(r_i,p_{i+1})\geq d(p_{i+1},q_{i+1})-d(r_i,q_{i+1})\geq (i+2)N-\delta$. This is a contradiction since $N>7\delta +1$.
Thus $r_i\in [q_i,x_n]\subset [p_i,x_n]$. Also note that $d(p_{i+2},r_i)\leq d(p_{i+2},q_{i+1})+d(q_{i+1},r_i)\leq 7\delta$.

We now show that $[p_i,p_{i+1}]\cup [p_{i+1},p_{i+2}]$ is a $(1,42\delta)$-quasi-geodesic
of length at least $3N$ for $N>7\delta +1$. 
Suppose $x\in [p_i, p_{i+1}]$ and $y\in [p_{i+1},p_{i+2}]$. 
It is enough to show that $|d(x,p_{i+1})+d(p_{i+1},y)-d(x,y)|\leq 42 \delta$.

For the $\delta$-slim triangle $\bigtriangleup p_i q_i p_{i+1}$, there exists $u\in [p_i, q_i]$
such that $d(x,u)\leq 7\delta$. Similarly for
$\bigtriangleup p_{i+1}q_i r_i$ and $\bigtriangleup p_{i+1} p_{i+2} r_i$, there exists $v\in [q_i,r_i]$
such that $d(y,v)\leq  8\delta$ (the precise constant is obtained by a routine computation).

We have the following inequalities: \\
$|d(x,p_{i+1})-d(u,q_i)|\leq d(x,u) + d(p_{i+1}, q_i) \leq 6\delta + 7\delta= 13\delta$, \\
$ |d(p_{i+1},y)-d(q_i,v)|\leq d(p_{i+1}, q_i) + d(y,v) \leq 7 \delta + 8 \delta =15\delta$, \\
and $|d(u,v)-d(x,y)|\leq d(x,u) + d(y,v) \leq 6\delta + 8 \delta = 14\delta$.\\
Hence $|d(x,p_{i+1})+d(p_{i+1},y)-d(x,y)|\leq |\{d(x,p_{i+1})-d(u,q_i)\}+\{d(p_{i+1},y)-d(q_i,v)\}+\{d(u,v)-d(x,y)\}|\leq 42\delta$
and we are done.

The claim follows from Lemma $\ref{local-global-qg}$.

\medskip

Next we show that $\gamma(\infty)=\xi$. For this, by Lemma \ref{prev}, we just need to check that $\{p_n\} \sim \{x_n\}$.
Again, to show this, it is enough to check that $\{p_k\} \sim \{x_{n_{k-1}}\}$, i.e. 
$lim_{k\rightarrow \infty}(p_k.x_{n_{k-1}})_p=\infty$.
By the above proof we know that $ (\cup_{i=1}^{k-1}[p_{i-1},p_i])\bigcup[p_{k-1},x_{n_{k-1}}]$ is a
uniform quasi-geodesic. Thus, by stability of quasi-geodesics (Lemma $\ref{stab-qg}$), we can find a constant $D$ depending
only on $\delta$ such that there is a point $u\in [p,x_{n_{k-1}}]$ with $d(p_{k-1},u)\leq D$;
similarly there is a point $v\in [p,p_k]$ such that $d(p_{k-1},v)\leq D$. Therefore, we have $d(u,v)\leq 2D$ and 
$(p_k . x_{n_{k-1}})_p\geq (u.v)_p\geq d(p,u)-d(u,v)\geq d(p,p_{k-1})-d(u,p_{k-1})-d(u,v)\geq d(p,p_{k-1})-3D$. 
As $lim _{k\rightarrow \infty}d(p,p_k)= \infty$, we have $lim_{k\rightarrow \infty} (p_k .x_{n_{k-1}})_p=\infty$.
Therefore, the proof is complete by taking $K_{\ref{qi-geod-line}}(\delta)\geq \lambda_{\ref{local-global-qg}}(\delta,1,42\delta)$.

\smallskip

$(2)$ Let $\lambda:= \lambda_{\ref{local-global-qg}}(\delta,1,42\delta)$, and $D_1:=D_{\ref{stab-qg}}(\delta,\lambda)$. Now,
using the proof of $(1)$, we can construct two $\lambda$-quasi-geodesic rays
$\gamma_1,\gamma_2$, parametrized by arc length, joining a point $p\in X$ to $\xi_1$ and $\xi_2$ respectively. 
Clearly, $sup\{(x.y)_p: x\in \gamma_1,\,y\in \gamma_2\}< \infty$, else there exist $x_n\in \gamma_1$, $y_n\in \gamma_2$,
$n\in \mathbb N$, such that
$(x_n.y_n)_p \rightarrow \infty$. Since $x_n\rightarrow \gamma_1(\infty)=\xi_1$ and $y_n\rightarrow \gamma_2(\infty)=\xi_2$
by Lemma   \ref{prev}, this contradicts the fact that $\xi_1\neq \xi_2$. 

Let $N_1=sup\{(x.y)_p:x\in \gamma_1,y\in \gamma_2\}$. Let $x_i\in \gamma_i$, $i=1,2$,
be such that $(x_1.x_2)_p \geq N_1-1$. 
Let $u_i\in [p,x_i]$, $i=1,2$ be internal points of $\Delta px_1x_2$.
By Lemma $\ref{stab-qg}$ we can find $p_i\in \stackrel{\frown}{px_i}$ such that 
$d(p_i,u_i)\leq D_1$, $i=1,2$.
Now, let $\gamma^{'}_i \subset \gamma_i$ be the quasi-geodesic subray starting from $p_i$, for $i=1,2$.
We intend to show that the arc length parametrization of the concatenation of
$\gamma^{'}_1$, $\gamma^{'}_2$ and a geodesic segment $[p_1,p_2]$ joining $p_1, p_2$ is a uniform quasi-geodesic (see figure below).

\begin{center}
\includegraphics[height=50mm]{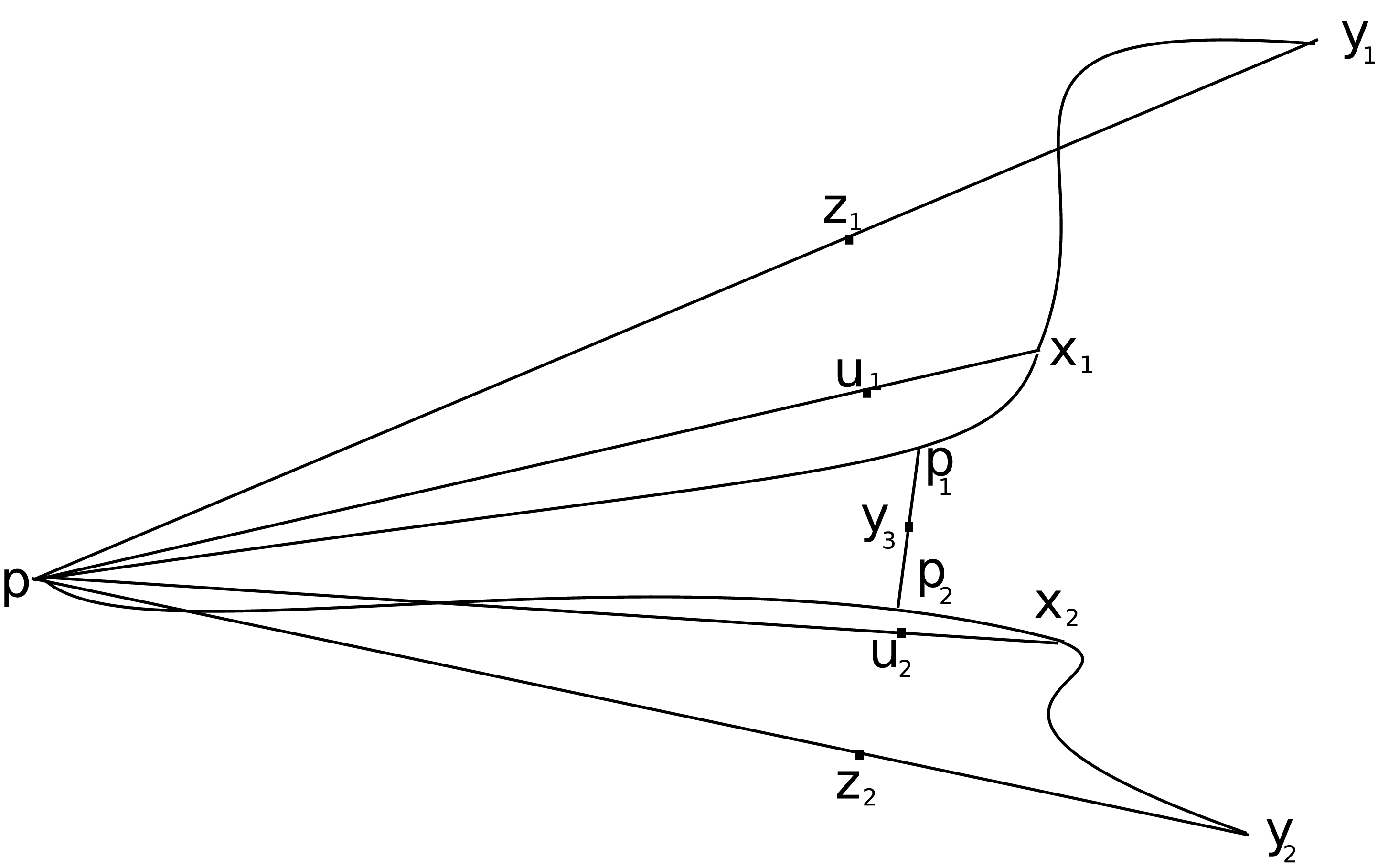}
\end{center}

Suppose $y_i\in \gamma^{'}_i$, $i=1,2$ and $y_3\in [p_1,p_2]$. 
It suffices to find $K\geq 1$ and $\epsilon \geq 0$ independent of $y_1,y_2,y_3$
such that the following conditions are satisfied.\\
Condition $(1)$ $ l(\stackrel{\frown}{p_1 y_1})+ l( \stackrel{\frown}{p_2 y_2}) +d(p_1,p_2)\leq Kd(y_1,y_2)+\epsilon$,\\
Condition $(2)$ $l(\stackrel{\frown}{p_1 y_1})+ d(p_1,y_3)\leq Kd(y_1,y_3)+\epsilon$, and\\
Condition  $(3)$ $l(\stackrel{\frown}{p_2 y_2})+ d(y_2,y_3)\leq Kd(y_2,y_3)+\epsilon$,\\
where $\stackrel{\frown}{p_i y_i}$ is the subsegement of $\gamma_i$ between $p_i$ and $y_i$ for $i=1,2$; also for a
rectifiable curve segment $\alpha$, $l(\alpha)$ denotes the length of the curve $\alpha$.
Since the proofs of Conditions $(2)$ and $(3)$ are similar we shall give proofs of Conditions  $(1)$ and $(2)$ below.

First of all, we note that $d(u_1,u_2)\leq 4\delta$ by Lemma $\ref{hyp-defn}$ and hence $d(p_1,p_2)\leq d(p_1,u_1)+d(p_2,u_2)+d(u_1,u_2)\leq 2D_1+4\delta=D_2$, say.
By Lemma $\ref{stab-qg}$ we can find $z_i\in [p,y_i]$ such that $d(p_i,z_i)\leq D_1$, $i=1,2$.

We shall first show  that the difference between $(y_1.y_2)_p$ and $(p_1.p_2)_p$ is small.
Note that $(y_1.y_2)_p\geq (z_1.z_2)_p\geq (p_1.p_2)_p -\{d(p_1,z_1)+d(p_2,z_2)\}\geq (p_1.p_2)_p-2D_1$.
Also, $|(x_1.x_2)_p - (u_1.u_2)_p|=|d(p,u_1)-(u_1.u_2)_p|=d(u_1,u_2)/2\leq 2\delta$ and $|(p_1.p_2)_p -(u_1.u_2)_p|\leq d(p_1,u_1)+d(p_2,u_2)\leq 2D_1$.
Thus $|(x_1.x_2)_p -(p_1.p_2)_p| \leq 2(D_1+\delta)$ and hence
$(y_1.y_2)_p\geq (p_1.p_2)_p -2D_1\geq (x_1.x_2)_p-(4D_1+2\delta)$.

Since $(y_1.y_2)_p\leq N_1$ and $(x_1.x_2)_p\geq N_1-1$ we have $$|(y_1.y_2)_p -(p_1.p_2)_p| \leq (1+2\delta+4D_1).$$
Next, suppose that $c_i\in [p,y_i]$, $i=1,2$ and $c\in [y_1,y_2]$ are the internal points of $\Delta py_1y_2$. 
We shall show that $d(p_i,c_i)$, $i=1,2$
are small.

Suppose $q_i\in [p,p_i]$, $i=1,2$ are internal points of $\bigtriangleup pp_1p_2$. Then  $d(p_i,q_i)\leq d(p_1,p_2)\leq D_2$.
We can choose $r_i\in [p,y_i]$ such that $d(r_i,q_i)\leq 2D_1$, by Lemma $\ref{stab-qg}$ applied to the subsegment of the quasi-geodesic $\gamma_i$
between $p$, $p_i$ and $p$, $y_i$.  
Hence $d(c_i,r_i)= |d(p,c_i)-d(p,r_i)|\leq |d(p,c_i)-d(p,q_i)|+|d(p,q_i)-d(p,r_i)|\leq (1+2\delta+4D_1)+d(q_i,r_i)$.
Hence $d(c_i,r_i)\leq (1+2\delta +6D_1)$. This gives $d(c_i,p_i)\leq d(c_i,r_i)+d(r_i,q_i)+d(q_i,p_i)\leq (1+2\delta+8D_1 +D_2)$. 
Since $d(c,c_i)\leq 4\delta$ we have $$d(c,p_i)\leq d(c,c_i)+d(c_i,p_i)\leq (1+6\delta+8D_1 +D_2).$$

Thus for any point $y_3\in [p_1,p_2]$ we have $d(c,y_3)\leq d(p_1,p_2)+d(p_1,c)\leq (1+6\delta +8D_1+2D_2)=D_3$, say.

{\em Proof of Condition} $1:$ Now, 
\[
\begin{array}{l}
\sum_{i=1}^2 l(\stackrel{\frown}{p_i y_i}) +d(p_1,p_2)\\
\leq \sum_{i=1}^2 \{\lambda d(p_i,y_i) +\lambda\} + D_2, \,\, \mbox{since} \,\, \gamma_i \,\, \mbox{are}\,\, \lambda-\mbox{quasi-geodesics.} \\
\leq \sum_{i=1}^2 \lambda \{d(y_i,c) +d(c,p_i)\} + 2\lambda + D_2\\
\leq \lambda d(y_1,y_2) + 2\lambda +D_2 +2\lambda D_3.

\end{array}
\]

{\em Proof of Condition} $2:$
\[
\begin{array}{l}
l(\stackrel{\frown}{p_1 y_1})+d(p_1,y_3)\\
\leq \{\lambda d(p_1,y_1) +\lambda\}+ d(p_1,y_3)\\
\leq \lambda \{d(y_1,y_3)+d(y_3,p_1)\} +\lambda +d(p_1,y_3)\\
\leq \lambda d(y_1,y_3)+ \lambda +(\lambda +1)d(p_1,p_2)\\
\leq \lambda d(y_1,y_3)+ (\lambda+(\lambda +1)D_2).

\end{array}
\]

As mentioned before, the proof of Condition $3$ is exactly like the proof of Condition $2$. $\Box$

\smallskip

Two quasi-geodesic rays $r_i: [0, \infty) \rightarrow X$, $i=1,2$,
in a hyperbolic metric space $(X,d)$
are said to be {\it asymptotic} if there exists $C_0$ such that $d(r_1(t), r_2(t)) \leq C_0$ for all $t \in [0, \infty)$.
Using stability of quasi-geodesics (Lemma $\ref{stab-qg}$) the proofs of the following lemma and corollary are standard
(see Lemma $1.15$, Chapter $III.H$, \cite{bridson-haefliger}). 

\begin{lemma}{\bf Asymptotic rays are uniformly close:}\label{asymp-ray}
For all $\delta \geq 0$ and $k\geq 1$ there is a constant $D_{\ref{asymp-ray}}=D_{\ref{asymp-ray}}(\delta,k)$ such that
the following holds:\\
Suppose $X$ is a $\delta$ hyperbolic metric space and $\gamma_1,\gamma_2:[0,\infty)\rightarrow X$
are two asymptotic $k$-quasi-geodesic rays. Then there exists $T\geq 0$ such that $\gamma_1(t)\in N_{D_{\ref{asymp-ray}}}(Im(\gamma_2))$ and
$\gamma_2(t)\in N_{D_{\ref{asymp-ray}}}(Im (\gamma_1))$, for all $t\geq T$.
\end{lemma}

\begin{cor}\label{stab-of-lines}
For all $\delta \geq 0$ and $K\geq 1$ there is a constant 
$D_{\ref{stab-of-lines}}=D_{\ref{stab-of-lines}}(\delta,K)$ such that
the following holds:

Suppose $X$ is a $\delta$-hyperbolic metric space and let $\gamma_1,\gamma_2$ be two
$K$-quasi-geodesic lines in $X$ joining the same pair of points $\xi_1,\xi_2\in \partial X$.
Then the Hausdorff distance between $\gamma_1$ and $\gamma_2$ is at most 
$D_{\ref{stab-of-lines}}$.
\end{cor}

 Lemma $\ref{stab-qg}$ and  Lemma $\ref{asymp-ray}$ combined with the proof
of Lemma $\ref{barycen}$(2), immediately imply the following result.
\begin{lemma}\label{barycent}
For all $\delta\geq 0$, $D^{'}\geq 0$ and $k\geq 1$ there are constants
$D=D_{\ref{barycent}}(\delta,k)$ and $L=L_{\ref{barycent}}(\delta,k,D^{'})$ 
such that we have the following:

Suppose $X$ is a $\delta$-hyperbolic metric space. Then
\begin{enumerate}
\item Let $\Delta \xi_1\xi_2\xi_3$ be a $k$-quasi-geodesic {\rm ideal triangle} in 
$X$, i.e. a union of three $k$-quasi-geodesic lines in $X$ joining the pairs of
points $(\xi_i,\xi_j)$, $i \neq j; $ $1\leq i,j\leq 3$. Let us denote the quasi-geodesic lines joining $\xi_i,\xi_j$ by
$[\xi_i,\xi_j]$.
Then there is a point $x\in X$ such that $x\in N_D([\xi_i,\xi_j])$ for
all $i\neq j$.

\item If $x,x^{'}\in X$ are two points each of which is contained within a $D^{'}$-
neighborhood of each of the sides of an ideal $k$-quasi-geodesic triangle
$\Delta \xi_1\xi_2\xi_3$ then $d(x,x^{'})\leq L$.
\end{enumerate}
\end{lemma}

If a point $x\in X$ is contained in the $D^{'}$-neighborhood of each of the sides of 
an ideal quasi-geodesic triangle $\Delta \xi_1\xi_2\xi_3$, then $x$ will be called a {\em $D^{'}$-barycenter}
of $\bigtriangleup \xi_1 \xi_2\xi_3$. 
A $D_{\ref{barycent}}$-barycenter will be simply referred to as a {\em barycenter.}

Now, Lemma  $\ref{barycent}$ along with the proof of  Lemma $\ref{barycen}(2)$ gives
 the following.
\begin{lemma}
Given $\delta\geq 0$, $D^{'}\geq 0$, $K_1\geq 1$ and $K_2\geq 0$ there exists $D=D(\delta,K_1,K_2,D^{'})$
such the following holds:

Suppose $f:X\rightarrow Y$ is a $K_1$-quasi isometric embedding of  $\delta$-hyperbolic
metric spaces. Let $\bigtriangleup \xi_1\xi_2\xi_3 \subset X$ and
$\bigtriangleup \partial(f)(\xi_1)\partial(f)(\xi_2)\partial(f)(\xi_3)\subset Y$ be $K_2$-quasi-geodesic ideal triangles. 
If $x\in X$ is a $D^{'}$-barycenter of $\Delta\xi_1\xi_2\xi_3$, then $f(x)\in Y$ is a
$D$-barycenter of  $\Delta \partial(f)(\xi_1)\partial(f)(\xi_2)\partial(f)(\xi_3)$.
\end{lemma}

\noindent {\bf The barycenter map}\\
Suppose $X$ is a $\delta$-hyperbolic metric space such that
$\partial X$ has more than two points. Let us denote the set of all distinct triples
of points in $\partial X$ by $\partial^3 X$.
Now, given $\xi=(\xi_1,\xi_2,\xi_3)\in \partial^3 X$
 we can, by Lemma $\ref{qi-geod-line}$, construct a $K_{\ref{qi-geod-line}}(\delta)$-quasi-geodesic ideal triangle, 
say $\Delta_1$, with vertices $\xi_i$, $i=1,2,3$. Then, by Lemma $\ref{barycent}(2)$ there is a coarsely well
defined barycenter of $\Delta_1$. Suppose  $b_{\xi}$ is a barycenter of $\Delta_1$. {\em Henceforth, we shall refer
to it simply as a barycenter of the triple $(\xi_1,\xi_2,\xi_3)$.}
For a different set of choices of the $K_{\ref{qi-geod-line}}(\delta)$-quasi-geodesic lines joining the pairs
$(\xi_i,\xi_j)$, suppose we obtain a new ideal triangle $\Delta_2$, and suppose $b^{'}_{\xi}$ is a barycenter
of $(\xi_1,\xi_2,\xi_3)$ defined with respect to $\Delta_2$.
Then by the stability of  quasi-geodesic lines (Corollary $\ref{stab-of-lines}$), $b^{'}_{\xi}$ is
a $D_1:=(D_{\ref{barycent}}(\delta)+D_{\ref{stab-of-lines}}(\delta,K_{\ref{qi-geod-line}}(\delta))$-barycenter of
the triangle $\Delta_1$.
Hence, by Lemma $\ref{barycent}(2)$,  $d(b_{\xi},b^{'}_{\xi})\leq L_{\ref{barycent}}(\delta, K_{\ref{qi-geod-line}}(\delta),D_1)$
and we have:

\begin{lemma}\label{bary-map}
For every $\delta\geq 0$ there is a constant $D_{\ref{bary-map}}=D_{\ref{bary-map}}(\delta)$ such that
we have the following:

Suppose $X$ is a $\delta$-hyperbolic metric space and $\xi=(\xi_1,\xi_2,\xi_3)\in \partial^3 X$.
If $b_{\xi}$ and $b^{'}_{\xi}$ are two barycenters of $\xi$, then $d(b_{\xi},b^{'}_{\xi})\leq D_{\ref{bary-map}}$.
\end{lemma}

We shall say that a map $f: U \rightarrow (V,d_V)$ satisfying properties $\PP_1 , \cdots , \PP_k$ is {\it coarsely unique}
if there exists $C>0$ such that
for any other map $g: U \rightarrow (V,d_V)$ satisfying properties $\PP_1 , \cdots , \PP_k$, and any $u \in U$,
$d_V(f(u), g(u)) \leq C$.

Thus, from Lemma \ref{bary-map} we have a coarsely unique map $\phi:\partial^3 X\rightarrow X$, $\xi\mapsto b_{\xi}$
mapping a triple of points to a barycenter.
Any such map will be referred to as {\em the barycenter map}. Now we are ready to state the main proposition of this subsection. 

\begin{prop}{\bf Existence of qi sections for metric graph bundles:}\label{existence-qi-section} For all $\delta^{'},N\geq 0$ and 
proper $f:{\mathbb{N}} \rightarrow {\mathbb{N}} $ there exists $K_0=K_0(f,\delta^{'},N)$ such that the following holds.\\
Suppose $p : X \rightarrow B$ is an $(f,K)$-metric graph bundle with the following properties:
\begin{enumerate}
\item Each of the fibers $F_b$ , $b \in \mathcal{ V}(B)$ is a $\delta^{'}$-hyperbolic metric space with respect to the
path metric $d_b$ induced from $X$.
\item The barycenter maps $\phi_b : \partial^3 F_b \rightarrow F_b$ are uniformly coarsely surjective, i.e.
$F_b$ is contained in the $N$-neighborhood of the image of $\phi_b$ for all $b \in \mathcal{V} (B)$. 
\end{enumerate}
Then there is a $K_0$-qi section through each point of $\mathcal V(X)$.
\end{prop}

Note that the constant $K$ in  Proposition \ref{existence-qi-section} above is given by
 $K=f(4)$ by Proposition \ref{def2} and hence we may write $K_0=K_0(f,K,\delta^{'},N)$ by making the implicit dependence on $K$
explicit. We also assume without loss of generality that
for all $b\in \mathcal{ V} (B)$, the image
of  $\phi_b$ is contained in $\mathcal{ V}(F_b)$.

\begin{proof} Let us fix a set $\{\phi_b\}_{b\in \mathcal{ V}(B)}$ of barycenter maps and let
$v\in \mathcal{ V}(B)$, $x\in \mathcal V(F_v)$.
First, suppose that $x$ is contained in the image of the barycenter map $\phi_v$. We will construct a qi section through $x$. 
Choose a point $\xi_v = (\xi_1 ,\xi_2 , \xi_3 )\in \partial^3 F_v$ such that
$\phi_v(\xi_v) = x$. Denote $\xi_v = \xi$ and so $\phi_v(\xi) =x$. 

Let $w,z\in \mathcal{ V}(B)$, $w\neq z$. Choose a geodesic $\gamma$ joining $w,z$ and 
let $w=w_0,w_1,\ldots, w_{n-1},w_n=z$ be the consecutive vertices on $\gamma$.
By condition $(2)(ii)$ of the definition of metric graph bundles (Proposition \ref{def2}), 
for all $i$, $0\leq i\leq n-1$,
there is a $K$-quasi-isometry $f_{w_i w_{i+1}}:F_{w_i}\rightarrow F_{w_{i+1}}$ which sends any vertex
$y_i\in \mathcal V(F_{w_i})$ to a vertex
$y_{i+1}\in \mathcal V(F_{w_{i+1}})$ where $y_i$ and $y_{i+1}$ are connected by an edge.
By composition of $n$ such maps we get a map $f_{wz} : F_w\rightarrow F_z$, which sends each point $y\in  F_w$ to
a point $y^{'}\in \mathcal{ V}(F_z)$ such that $d(y,y^{'})\leq  d_B(w,z)+1 = n+1$.
Let $f_{ww}:F_w\rightarrow F_w$ denote the identity map on $F_w$, for all $w\in \mathcal{ V}(B)$.
Now we make the following observations:

$1.$ Since the inclusion maps $F_w\hookrightarrow X$ are uniformly metrically proper, by the definition of metric graph bundles,
the map $f_{wz}$ is coarsely uniquely determined.
In fact, if $d(w,z)= n$, $n\in \mathbb N$, we have for any other map $f^{'}_{wz}$ defined in the
same way, $d(f_{wz}(y),f^{'}_{wz}(y))\leq 2(n+1)$, so that $d_z(f_{wz}(y),f^{'}_{wz}(y))\leq f(2n+2)$, for all $y\in F_w$.

$2.$ 
Since each $f_{wz}$ is obtained as a composition of $K$-quasi-isometries it is a quasi-isometry.
Now, since the spaces $F_w,w\in \mathcal{ V}(B)$, are $\delta^{'}$-hyperbolic
and since the map $f_{wz}$ is coarsely uniquely determined, we have a well defined map
$\partial (f_{wz}):\partial F_w\rightarrow \partial F_z$, by Lemma $\ref{bdry-elementary}$, and hence an induced map
$\partial^3 f_{wz}: \partial^3 F_w \rightarrow \partial^3 F_z$, $\forall w,z\in \mathcal{ V}(B)$. 

Consider the map $s = s_{\xi,x} : \mathcal{ V}(B)\rightarrow X$ given by $ s(v)=x
=\phi_v(\xi)$,
and $s(w)=\phi_w((\partial^3 f_{vw} (\xi)))$, for all $w\in \mathcal{ V}(B), w\neq v$. 
We show below that $s$ (or $s( \mathcal{ V}(B))$) is the required qi section
through $x$. 

$3.$ Writing $\xi_w = \partial^3 f_{vw} (\xi)$ for all $w\neq v$, note that for any $w, z \in \mathcal{ V}(B)$, $\partial^3 f_{wz} (\xi_w ) = \xi_z$ .
This follows from the fact that by the definitions of the maps $f_{vz}$, $f_{wz}$, $f_{vw}$ we have, for all $y\in F_v$,
$d(f_{vz}(y),f_{wz}\circ f_{vw}(y)) \leq d_B(v,z)+d_B(w,z)+d_B(v,w)+3$, and thus 
$d_z(f_{vz}(x),f_{wz}\circ f_{vw}(x)) \leq f(d_B(v,z)+d_B(w,z)+d_B(v,w)+3)$.
The claim follows from Lemma $\ref{bdry-elementary}(3)$.

$4.$ Lastly, we show that there exists $C\geq 1$ such that for any pair  $w, z$ of adjacent vertices of $B$, $d(s(w), s(z)) \leq C$.
 By Condition 2 (ii) (Proposition \ref{def2}) $f_{wz}$ is a $K$-quasi-isometry.
Let $\xi_w = (\beta_1 , \beta_2 , \beta_3 )$ and $\partial^3f_{wz}(\xi_w)=\xi_z=(\eta_1,\eta_2,\eta_3)$.
Choose $K_{\ref{qi-geod-line}}(\delta^{'})$-quasigeodesic ideal triangles $\Delta_w$ and $\Delta_z$ respectively in $F_w$ and
$F_z$, with vertices $\xi_i$'s and $\eta_i$'s; by definition of the map $s$, $s(w)$ and $s(z)$ are
$D_{\ref{barycent}}(\delta^{'})$-barycenters of these triangles. Now, the map $f_{wz}$ takes the ideal triangle
$\Delta_w$ to an ideal $K_1$-quasigeodesic triangle with vertices $\eta_i$'s, where $K_1=K_{\ref{qi-geod-line}}(\delta^{'}).K+ K$,
and $f_{wz}(s(w))$ is a $D_1:=\{D_{\ref{barycent}}(\delta^{'}).K+K\}$-barycenter of the new triangle. 
Thus, by Lemma $\ref{stab-of-lines}$, $f_{wz}(s(w))$ is a $D_2$-barycenter of the triangle $\Delta_z$, where
$D_2=D_{\ref{stab-of-lines}}(\delta^{'},K_1)+D_1$. Hence, by Lemma $\ref{barycent}$, we have
$d(s(z),f_{wz}(s(w)))\leq L_{\ref{barycent}}(\delta^{'},K_{\ref{qi-geod-line}}(\delta^{'}),D_2)$. Since
$d(s(w),f_{wz}(s(w)))=1$ we have $d(s(w),s(z))\leq C:=1+L_{\ref{barycent}}(\delta^{'},K_{\ref{qi-geod-line}}(\delta^{'}),D_2)$.

 For any $w, z \in \mathcal{ V}(B)$, 
$d(w, z)\leq d(s(w), s(z))$ by the definition of a metric graph bundle. Also from  Step $(4)$ above,
we have $d(s(w), s(z)) \leq C.d(w, z)$. Hence  $s$ is a $C$-qi section.

If $x\in \mathcal V(F_v)$ is not in the image of $\phi_v$, we can choose $x_1\in \mathcal V(F_v)$ 
such that $d(x,x_1)\leq N$ and $x_1\in Im(\phi_v)$. Now construct as above a $C$-qi section $s=s_{\xi,x_1}$,
and define a new section $s^{'}$ by setting $s^{'}(b)=s(b)$ for all $b\in \mathcal{ V}(B), b\neq v$ and $s^{'}(v)=x$. This 
is an $(N+C)$-qi section passing through $x$. Thus we can take $K_0=N+C$ to finish the proof of the proposition. \end{proof}

Applying this proposition to Example $\ref{eg-mbdl}$,
we have a different proof of the following result of Mosher \cite{mosher-hypextns}.

\begin{theorem} \label{qi-mosher} {\bf (Mosher \cite{mosher-hypextns})}
Let us consider the short exact sequence of finitely generated groups 
\medskip
\begin{center}
$1\rightarrow A\rightarrow G\rightarrow Q\rightarrow 1$.
\end{center}
such that $A$ is non-elementary word hyperbolic. Then there exists a q(uasi)-i(sometric) section $\sigma : Q \rightarrow G$.
Hence, if $G$ is hyperbolic, then so is $Q$.\end{theorem}

Let $p:X^{'}\rightarrow B^{'}$ be an $(f,c,K)$-metric bundle and let $\pi:X\rightarrow B$ be an approximating metric graph bundle 
as in Lemma $\ref{mbdl-mgbdl}$.  
As in  Lemma $\ref{mbdl-mgbdl}$ we suppose that 
the maps $\psi_X,\psi_B$ are $K_1-$ quasi-isometries. Let $\psi_{X^\prime}$ (resp. $\psi_{B^\prime}$)
be a quasi-isometric coarse inverse of the map $ \psi_X$ (resp. $\psi_B$) constructed
as in the proof of Lemma \ref{elem-lemma1} (2). We assume that these maps are   inverses of  $\psi_X$, $\psi_B$
 when restricted to the vertex sets $\mathcal V(X)$ and $\mathcal V(B)$ respectively. Moreover, we assume that
$\psi_{X^\prime}$, $\psi_{B^\prime}$ are $K_1$-quasi-isometries.
 Also we assume that the restrictions of $\psi_X$ and
$\psi_{X^{'}}$ to horizontal spaces ({\em cf.} Lemma \ref{coarse3}) are $K_1$-quasi-isometries.

\begin{prop}{\bf Existence of qi section for metric bundles:}\label{mgsec-to-mbsec}
 Let $p:X^{'}\rightarrow B^{'}$, $\pi:X\rightarrow B$ and
$\psi_B,\psi_X,\psi_{X^{'}},\psi_{B^{'}}, K_1$ be as above.
Let $V\subset B^{'}$ be the collection of points of $B^{'}$ that form the vertex set of $B$.
 Suppose we have a $k$-qi section $s:V\rightarrow X$. Then we have a 
$k^{'}=K_{\ref{mgsec-to-mbsec}}(f,c,K,K_1,k)$-qi section $s^{'}:B^{'}\rightarrow X^{'}$ such that $\psi_X\circ s=s^{'}\circ \psi_B$. 

Hence any metric bundle satisfying
the properties \\
1) horizontal spaces
are uniformly hyperbolic, and\\
 2) the
barycenter maps of these spaces are uniformly
coarsely surjective\\
admits a uniform qi section through each
point.
\end{prop}

\begin{proof}
The proof of the first part of the proposition 
is clear once we describe what the map $s^{'}$ is. For $u\in V$ define $s^{'}(u)=\psi_{B^{'}}\circ s(u)$.
Suppose $u\in B\setminus V$. Let $v\in V$, so that $d(u,v)\leq 1$. Choose $x\in F_v$  such that $x$ can be joined to $s^{'}(u)$ by
a curve in $X^{'}$ of length at most $c$ and define $s^{'}(v)=x$.

Next note  that if the fibers of a metric bundle are (uniformly)
hyperbolic, then so are the vertex spaces of an approximating
metric graph bundle. This is because the fibers of  an approximating
metric graph bundle are uniformly quasi-isometric to the fibers of the metric bundle.
Next (for the same reason) observe that if the barycenter maps of the metric bundle are uniformly coarsely surjective, then so are  the barycenter maps of 
an approximating
metric graph bundle.
The last part of the
proposition now follows from Proposition \ref{existence-qi-section}
and the first part of the proposition.
\end{proof}


\subsection{Ladders}

We use the term {\em ladder} below due to a similar ladder construction in \cite{mitra-trees}.
The term {\em girth} is taken from \cite{BF}.

\begin{defn}\label{defn-ladder} Suppose $X$ is a metric bundle (resp. a metric graph bundle) over $B$.
Suppose $X_1$ and $X_2$ are two $c_1$-qi sections of the metric bundle $X$.
For each $b\in B$ (resp. $b\in \mathcal{ V}(B)$), join the points $X_1\cap F_b$, 
$X_2\cap F_b$ by a geodesic in $F_b$. We denote the union of these 
geodesics by $C(X_1,X_2)$, and call it a {\bf ladder} formed by the sections 
$X_1$ and $X_2$.
\end{defn}

\begin{rem}\label{defn-ladder-rem}
If (as in the case of interest) the horizontal spaces are $\delta^{'}$-hyperbolic, for some 
$\delta^{'} \geq 0$, the Hausdorff distance between any pair of ladders determined 
by two given sections $X_1,X_2$ is uniformly bounded.  In such a situation,
  $C(X_1,X_2)$ will refer to any one of them, and abusing notation we refer to  $C(X_1,X_2)$ 
as {\em the} ladder determined by $X_1,X_2$. 
\end{rem}

For four qi sections $X_i$, $i=1,2,3,4$ we write  
$C(X_3,X_4)\subset C(X_1,X_2)$ to mean  $C(X_3,X_4) \cap F_b \subset C(X_1,X_2) \cap F_b $ for all $b \in B$ (or $\mathcal{V}(B)$).

\begin{defn}\label{defn-girth}
Suppose $X_1$ and $X_2$ are two $c_1$-qi sections of a metric bundle (resp. metric graph bundle) $X$ over $B$.
We define $d_h(X_1,X_2)={\rm inf}\{d_b(F_b\cap X_1,F_b\cap X_2):b\in B\} $ (resp.
${\rm inf}\{d_b(F_b\cap X_1,F_b\cap X_2):b\in \mathcal{ V}(B)\} $)
and call it the {\bf girth} of any ladder $C(X_1,X_2)$, determined by $X_1,X_2$.
\end{defn}

\begin{defn}{\bf Neck of Ladders:}
Suppose $X$ is a metric bundle (resp. metric graph bundle) over $B$ and
let $X_1,X_2$ be two qi sections. Let $C(X_1,X_2)$ be a ladder determined by $X_1,X_2$ and let $A\geq 0$.
We define $U_A(X_1,X_2)$ to be
the set $\{b\in B:\,d_b(X_1\cap F_b,X_2\cap F_b)\leq A\}$ (resp. $\{b\in \mathcal{V} (B):\,d_b(X_1\cap F_b,X_2\cap F_b)\leq A\}$) and call it the $A$-neck of the ladder $C(X_1,X_2)$.
\end{defn}

A first aim of this subsection is to show that under suitable restrictions on a metric
bundle or a metric graph bundle necks of ladders are quasi-convex
subsets of the base space. The next lemma  leads to one of the main tools (Lemma $\ref{qc-level-set2}$)
for proving the combination theorem \ref{combthm}. This lemma originally appears in \cite{hamenst-word} in the context of
metric fibrations. The proof that we give here is almost the same as that of \cite{hamenst-word},
nevertheless we include it for the sake of completeness. For convenience of exposition we suppress the dependence of the constants
(defined in the following lemma)
on the parameters $f,c,K$.

\begin{lemma}\label{qc-level-set} Let $X$ be an $(f,c,K)$-metric bundle over $B$
satisfying $(M_k,\lambda_k,n_k)$-flaring for all $k\geq 1$ (cf. Definition $\ref{defn-flare}$), and let $\mu_k$ be 
the bounded flaring function (cf. Corollary $\ref{bdd-flaring-mbdl}$). Then
for all $c_1\geq 1$ and $R>1$ there are constants $D_{\ref{qc-level-set}}=D_{\ref{qc-level-set}}(c_1,R)$
and $K_{\ref{qc-level-set}}= K_{\ref{qc-level-set}}(c_1)$ such that the following holds:\\
Suppose $X_1,X_2$ are two $c_1$-qi sections of $B$ in $X$ and let $A\geq max \{ M_{c_1}, d_h(X_1, X_2) \}$.
\begin{enumerate}
\item Let $\gamma:[t_0,t_1]\rightarrow B$ be a geodesic such that\\
a)  $d_{\gamma(t_0)}(X_1\cap F_{\gamma(t_0)},X_2\cap F_{\gamma(t_0)})=AR$.\\
b) $\gamma(t_1)\in U_A:=U_A(X_1,X_2)$ but for all $t\in [t_0,t_1)$, $\gamma(t)\not \in U_A$.\\
Then the length of $\gamma$ is at most $D_{\ref{qc-level-set}}(c_1,R)$.

\item $U_A$ is $K_{\ref{qc-level-set}}$-quasi-convex in $B$.

\item If $d_h(X_1,X_2) \geq M_{c_1}$ then the diameter of the set $U_A$ is at most 
$D^{'}_{\ref{qc-level-set}}=D^{'}_{\ref{qc-level-set}}(c_1,A)$.
\end{enumerate}
\end{lemma}

For  convenience of exposition we will write 
$\lambda$ for $\lambda_{c_1}$, $n$ for $n_{c_1}$ and $\mu$ for $\mu_{c_1}$ in the proof below.
Also $l(\alpha)$ will denote the length of a curve $\alpha$.

\begin{proof} Since $A\geq  d_h(X_1, X_2) $, $U_A \neq \emptyset$.

\smallskip

\noindent $(1)$ Let $\phi:[t_0,t_1]\rightarrow \mathbb R$ be the function
$t\mapsto d_{\gamma(t)}(X_1\cap F_{\gamma(t)},X_2\cap F_{\gamma(t)})$
and $t_1-t_0=n.L+\epsilon$ where $L\in {\mathbb Z}^+$
and $0\leq \epsilon <n$. Suppose  $L\geq 3$. 
Consider the sequence $\phi(t_0+ni)$, $i=1,\cdots,L$.
Since $\phi(t_0+n.i)\geq M_{c_1}$, for all $i\in [1,L-1]$,
\[
\lambda\phi(t_0+ni)\leq \mbox{max}\{\phi(t_0+n(i-1)),\phi(t_0+n(i+1))\} 
\]

\noindent  by the 
flaring condition.

Hence if $\phi(t_0+n)>\phi(t_0)$ then $\phi(t_0+n(i+1))\geq \lambda\phi(t_0+ni)$
for all $i\in [1,L-1]$. Then, 
$\phi(t_0+nL))\geq \lambda^{L-1}\phi(t_0)$.
Using bounded flaring (Corollary $\ref{bdd-flaring-mbdl}$) we have,
$\phi(t_0+nL)\leq \mu(n)\max\{\phi(t_1),1\}$. Putting all these together
and using the fact that $\phi(t_1)\leq A$ and $\phi(t_0)>A$, we have 
$L-1< log(\mu(n))/ log\lambda$.

Hence, $L\geq 3$ and $\phi(t_0+n)> \phi(t_0)$ implies 
\[l(\gamma)< n(L+1)\leq 2n+ n.log\mu(n)/log\lambda . \]
Suppose $\phi(t_0)\geq \phi(t_0+n)$ and let $k\leq L$ be the largest
integer such that $\phi(t_0)\geq \phi(t_0+n)\geq \cdots \geq \phi(t_0+k.n)$.
If $k\geq 2$, applying the flaring condition we get
$\phi(t_0+(i-1)n)\geq \lambda\phi(t_0+in)$ for all $i\in [1,k-1]$. 
Then $\phi(t_0)\geq \lambda^{k-1}\phi(t_0+(k-1)n)> \lambda^{k-1}A$.
Therefore $k< 1+  \{log\phi(t_0)-logA\}/log \lambda = 1+ logR/log\lambda$.
Also, by the first part of the proof, 
$l(\gamma|_{[t_0+k.n,t_1]})< 
\max \{3n, 2n+n ~log\mu(n)/log \lambda\}$.
Hence, 
\[ 
\begin{array}{rcl}
l(\gamma)& \leq & n+ n\{logR/log \lambda + max\{3n,2n+ n~log(\mu(n))/log\lambda \} \}.
\end{array}
\] 
Taking $D_{\ref{qc-level-set}}=D_{\ref{qc-level-set}}(c_1,R)$
as the right hand side of the above inequality, part (1) of the lemma is proved.

$(2)$ Suppose $\gamma:[t_0,t_1]\rightarrow B$ is a geodesic joining two points
of $U_A$, such that for all
$t\in (t_0,t_1)$, $\gamma(t)\not\in U_A$. Without loss of generality, we may
assume that $t_1-t_0>n$. Let $t_2=t_0+n$. Then by  bounded
flaring, we have $\phi(t_2)\leq \mu(n)\phi(t_0)\leq \mu(n).A$.
Again by the first part of the lemma
$l(\gamma|_{[t_2,t_1]})\leq D_{\ref{qc-level-set}}(c_1,\phi(t_2)/A)$.
Since, the function $D_{\ref{qc-level-set}}$ is increasing in the second variable,
given that the first variable is fixed, we have $l(\gamma) \leq n+D_{\ref{qc-level-set}}(c_1, \mu(n))$.
Hence, taking 
$K_{\ref{qc-level-set}}(c_1)=n+D_{\ref{qc-level-set}}(c_1, \mu(n))$
we are through. 

$(3)$ Suppose $b_1,b_2\in U_A$, $d_B(b_1,b_2)/2=L.n+\epsilon$, $0\geq \epsilon <n$.
Let $\gamma:[-(L.n+\epsilon), (L.n+\epsilon)]\rightarrow B$ be a geodesic joining $b_1,b_2$, so that
$\gamma(0)$ is the midpoint of the geodesic $\gamma$. The bounded flaring condition gives
 $d_{\gamma(t)}(F_{\gamma(t)}\cap X_1, F_{\gamma(t)}\cap X_2)\leq A.\mu(n)$ for $t= -L.n, L.n$.

As in the proof of the first part of the lemma,
$d_{\gamma(t)}(F_{\gamma(t)}\cap X_1, F_{\gamma(t)}\cap X_2)\geq \lambda^L. d_{\gamma(0)}(F_{\gamma(0)}\cap X_1, F_{\gamma(0)}\cap X_2)$ either for $t=L.n$ or for $t=-L.n$. Since 
$d_{\gamma(0)}(F_{\gamma(0)}\cap X_1, F_{\gamma(0)}\cap X_2)\geq M_{c_1}$, it follows that
$\lambda^L. M_{c_1}\leq A.\mu(n)$. Hence, $L\leq log(A.\mu(n)/M_{c_1})/log\lambda$. 
\end{proof}

This lemma has the following analog for metric graph bundles. We omit the proof
since it is an exact replica of the proof of the previous lemma (see also Remark \ref{qcrmk} below). We just need to point out
that in the proof of the first part of the lemma the function $\phi$ should have 
$[t_0,t_1]\cap \mathbb Z$ as domain and for the latter parts it is useful to recall
that in a graph, points on a geodesic refer to the vertices on the geodesic.
Also, as in Lemma  \ref{qc-level-set} above, we suppress the dependence of the constants
on the parameters $f,K$.

\begin{lemma}\label{qc-level-set-new}
Let $X$ be an $(f,K)$-metric graph bundle over $B$
satisfying $(M_k,\lambda_k,n_k)$-flaring for all $k\geq 1$ (cf. Definition $\ref{defn-flare}$), and let $\mu_k$ be 
the bounded flaring function (cf. Corollary $\ref{bdd-flaring}$). Then
for all $c_1\geq 1$ and $R>1$ there are constants $D_{\ref{qc-level-set-new}}=D_{\ref{qc-level-set-new}}(c_1,R)$
and $K_{\ref{qc-level-set-new}}= K_{\ref{qc-level-set-new}}(c_1)$ such that the following holds:\\
Suppose $X_1,X_2$ are two $c_1$-qi sections of $B$ in $X$ and let $A\geq max \{ M_{c_1}, d_h(X_1, X_2) \}$.
\begin{enumerate}
\item Let $\gamma:[t_0,t_1]\rightarrow B$ be a geodesic, $t_0,t_1\in \mathbb Z$, such that\\
a)  $d_{\gamma(t_0)}(X_1\cap F_{\gamma(t_0)},X_2\cap F_{\gamma(t_0)})=AR$.\\
b) $\gamma(t_1)\in U_A:=U_A(X_1,X_2)$ but for all $t\in [t_0,t_1)\cap \mathbb Z$, $\gamma(t)\not \in U_A$.\\
Then the length of $\gamma$ is at most $D_{\ref{qc-level-set-new}}(c_1,R)$.

\item $U_A$ is $K_{\ref{qc-level-set-new}}$-quasi-convex in $B$.

\item If $d_h(X_1,X_2) \geq M_{c_1}$ then the diameter of the set $U_A$ is at most 
$D^{'}_{\ref{qc-level-set-new}}=D^{'}_{\ref{qc-level-set-new}}(c_1,A)$.
\end{enumerate}
\end{lemma}

\begin{rmk}\label{qcrmk} {\rm  
 We note in particular that in Lemma \ref{qc-level-set} (1),
all that we need in order to make an analogous statement for a metric graph bundle is that $\phi(t)\geq M_{c_1}$, for all $t$.}
\end{rmk}

\begin{rmk}\label{qcrmk0} {\rm
It is not a priori clear that if a metric bundle satisfies a flaring condition, then an approximating metric graph 
bundle does so too (though this does follow {\it a posteriori} from Theorem \ref{combthm} and Proposition 
\ref{necflaring}). One reason is that the flaring condition is defined for any two qi lifts of a {\it geodesic} 
segment in the base. However, geodesics in the base space of the approximating metric graph bundle need not come from a geodesic in the base space of the  metric bundle.

 Lemma \ref{qc-level-set2} below  addresses this issue and proves that conclusions
similar to those of   Lemma \ref{qc-level-set} above remain true for the approximating metric graph bundle. This is the main reason
for giving explicitly a proof  of Lemma \ref{qc-level-set} here in the context of metric bundles rather than for
metric graph bundles (Lemma \ref{qc-level-set-new}).} 
\end{rmk}

Let $p:X^{'}\rightarrow B^{'}$ be an $(f,c,K)$-metric bundle 
satisfying $(M_k,\lambda_k,n_k)$-flaring for all $k\geq 1$  and let $\pi:X\rightarrow B$ be an approximating metric graph bundle 
as in Lemma $\ref{mbdl-mgbdl}$. 
 As in  Lemma $\ref{mbdl-mgbdl}$ we suppose that 
the maps $\psi_X,\psi_B$ are $K_1-$ quasi-isometries. Let $\psi_{X^\prime}$ (resp. $\psi_{B^\prime}$)
be a quasi-isometric coarse inverse of the map $ \psi_X$ (resp. $\psi_B$) constructed
as in the proof of Lemma \ref{elem-lemma1} (2). We assume that these maps are   inverses of  $\psi_X$, $\psi_B$
 when restricted to the vertex sets $\mathcal V(X)$ and $\mathcal V(B)$ respectively. Moreover, we assume that
$\psi_{X^\prime}$, $\psi_{B^\prime}$ are $K_1$-quasi-isometries.
 Also we assume that the restrictions of $\psi_X$ and
$\psi_{X^{'}}$ to horizontal spaces ({\em cf.} Lemma \ref{coarse3}) are $K_1$-quasi-isometries.

Let $B$ be $\delta_0$-hyperbolic. Suppose further that for every $k$-qi section of the approximating metric graph bundle,
we obtain a $k^{'}$-qi section of the original bundle.  
For convenience of exposition we suppress the dependence of the constants
(defined in the following lemma)
on the parameters $f,c,K,\delta_0$ etc.

\begin{lemma} \label{qc-level-set-app}
With notation as above,
let $X_1$, $X_2$ be two $k$-qi sections of the approximating metric graph bundle and $A_0\geq 0$.
 Suppose $d_h( X_1, X_2)\leq A_0$ and let $A_1=K_1.max\{A_0+K_1+1, M_{k^{'}}+K_1\}$. Then the following hold.\\
$(1)$ For $A\geq A_1$, $U_A(X_1,X_2)$ is $K_{\ref{qc-level-set-app}}= K_{\ref{qc-level-set-app}}(k,A)$-quasi-convex 
in $B$.\\
$(2)$ Suppose  $d_u(F_u\cap X_1,F_u\cap X_2)=C\geq A$ for some $u\in \mathcal{ V}(B)$.
Then $d_B(u,U_A(X_1,X_2))\leq D_{\ref{qc-level-set-app}}(k,C)$.\\
$(3)$ Suppose $d_h( X_1, X_2)\geq K_1(M_{k^{'}}+K_1)$. Then the diameter of the set $U_A(X_1,X_2)$  is at most $D^{'}_{\ref{qc-level-set-app}}(k,A)$.
\end{lemma}

\begin{proof} 
For the proof of this lemma we introduce the notation $d^{'}_u$ for the path metric on $X^{'}_u$ induced from $X^{'}$.
Also let $A_2=A_1/K_1 -1$.

\noindent $(1)$ By Proposition $\ref{mgsec-to-mbsec}$ we have 
two $k^{'}$-qi sections $X^{'}_1, X^{'}_2$  of the metric bundle $X^{'}$ (corresponding to  $X_1, X_2$ respectively) where
$ k^{'}=K_{\ref{mgsec-to-mbsec}}(f,c,K,K_1,k)$.
By choice of the constant $A_1$, we know that $U:=U_{A_2}(X^{'}_1,X^{'}_2)$ is a nonempty 
$K^{'}:=K_{\ref{qc-level-set}}(k^{'})-$quasiconvex
subset of $B^{'}$. Hence $\psi_{B^{'}}(U)\subset B$ is $D:=\{K_1.K^{'}+K_1+D_{\ref{stab-qg}}(\lambda, K_1)\}$-quasiconvex.
Also note that $\psi_{B^{'}}(U)\subset U_A(X_1,X_2)$.

Now suppose $u\in (U_A(X_1,X_2)\setminus \psi_{B^{'}}(U))$ is a point of $\mathcal V(B)$. 
Then $d^{'}_u(X^{'}_1\cap F_u, X^{'}_2\cap F_u)\leq K_1A+K_1$. It follows that
 either $u\in U$ 
 or  $d_{B^{'}}(u,U)\leq D_1=D_{ \ref{qc-level-set}}(k^{'},(A.K_1+K_1)/A_2)$ (by Lemma $\ref{qc-level-set}$(1)).
In any case, $d_B(u,\psi_{B^{'}}(U))\leq D_2=K_1D_1 +K_1$. Hence, 
$\psi_{B^{'}}(U)\subset U_A(X_1,X_2)\subset N_{D_2}(\psi_{B^{'}}(U))$. Since $\psi_{B^{'}}(U)$ is a $D$-quasi-convex set and
since $B$ is $\lambda$-hyperbolic, it follows that $U_A(X_1,X_2)$ is $K_{\ref{qc-level-set-app}}(=(2\lambda+D+D_2))-$quasi-convex.

$(2)$ If $u\in U$ then we set $D_{\ref{qc-level-set-app}}(k,C)=0$. Otherwise $d^{'}_u(X^{'}_1\cap F_u, X^{'}_2\cap F_u)\leq C.K_1 +K_1$
 since the restriction of the map $\psi_X$ to the horizontal space $F_u$
is  a $K_1$-quasi-isometry.
Hence by Lemma $\ref{qc-level-set}(1)$ we have $d_{B^{'}}(u,U)\leq D_3=D_{\ref{qc-level-set}}(k^{'},(C.K_1+K_1)/A_2)$.
Using the fact that $\psi_{B^{'}}$ is a $K_1$-quasi-isometry, we have $d_B(u,\psi_{B^{'}}(U))\leq K_1.D_3+K_1$.
Hence $d(u,U_A(X_1,X_2))\leq K_1D_3+K_1$. Set $D_{\ref{qc-level-set-app}}(k,C)=K_1+ K_1.D_{\ref{qc-level-set}}(k^{'},A_3)$
where $A_3=max\{1, (C.K_1+K_1)/A_2\}$.

$(3)$  By the given condition, for all $z\in U$,
$d^{'}_z(X^{'}_1\cap F_z, X^{'}_2\cap F_z)\geq M_{k^{'}}$ and so 
 we can apply the flaring condition. Now let $b_1,b_2\in U$ and let $b\in [b_1,b_2]$; then 
$d^{'}_b(X^{'}_1\cap F_b, X^{'}_2\cap F_b)\leq \mu_{k^{'}}(K_{\ref{qc-level-set}}(k^{'}))=D_4$, say, by Corollary $\ref{bdd-flaring-mbdl}$. 

Finally, as noted in Remark \ref{qcrmk},  what we really used in the proof of Lemma $\ref{qc-level-set}(1)$
is the fact 
that the value of the function $\phi$ is always greater than or equal to $M_{k^{'}}$. Thus in the same way we have
$d_{B^{'}}(b_1,b_2)\leq D_{\ref{qc-level-set}}(k^{'},D_4/A_2)$. Taking 
$D^{'}_{\ref{qc-level-set-app}}(k,A)= K_1 +K_1.D_{\ref{qc-level-set}}(k^{'},D_4/A_2)$
 completes the proof of the lemma. 
\end{proof}

We unify the content of the last two lemmas in the following lemma in the form that shall use later.

\begin{lemma}\label{qc-level-set2}
Given a function $f:\mathbb N\rightarrow \mathbb N$, $c_1\geq 1$ and $A_0 \geq 0$, there exist
$A^{'}_{\ref{qc-level-set2}}=A^{'}_{\ref{qc-level-set2}}(f,c_1,A_0)\geq A_0, A^{''}_{\ref{qc-level-set2}}=A^{''}_{\ref{qc-level-set2}}(f,c_1)$
and three functions $K_{\ref{qc-level-set2}}, D_{\ref{qc-level-set2}}:[A^{'},\infty)\rightarrow \mathbb R^{+}$,
$D^{'}_{\ref{qc-level-set2}}:[A^{''},\infty)\rightarrow \mathbb R^{+}$ such that the following hold:

Suppose $X$ is an $(f,K)$-metric graph bundle over $B$ such that\\
$1.$ either it satisfies a flaring condition\\
$2.$ or it is an approximating metric graph bundle of a metric bundle that satisfies a flaring condition.

Suppose $B$ is a hyperbolic metric space. Let $C(X_1,X_2)$ be a ladder formed by two $c_1$-qi sections $X_1,X_2$.
Let $d_h(X_1,X_2)\leq A_0$.
\begin{enumerate}
\item If $A\geq A^{'}_{\ref{qc-level-set2}}$ then $U_A(X_1,X_2)$ is $K_{\ref{qc-level-set2}}(A)$-quasi-convex.
Suppose  $d_u(F_u\cap X_1,F_u\cap X_2)=C\geq A$ for some $u\in \mathcal{ V}(B)$.
Then $d(u,U_A(X_1,X_2))\leq D_{\ref{qc-level-set2}}(C)$.
\item If $d_h(X_1,X_2) \geq A^{''}_{\ref{qc-level-set2}}$ then the diameter of the set $U_A(X_1,X_2)$ is at most $D^{'}_{\ref{qc-level-set2}}(A)$.
\end{enumerate}

\end{lemma}

The dependence of the functions $K_{\ref{qc-level-set2}}, D_{\ref{qc-level-set2}}, 
D^{'}_{\ref{qc-level-set2}}$  on $c_1$, (which is implicit here) will be made explicit
in the next section. Also we shall suppress the dependence of $A^{'}_{\ref{qc-level-set2}}, A^{''}_{\ref{qc-level-set2}}$ on $f$.

\section{Construction of Hyperbolic Ladders}

In this section we prove the main technical result leading to the combination theorem \ref{combthm}. A brief sketch follows:
For a metric bundle,  we first replace it with its approximating metric graph bundle. Then we work exclusively with
metric graph bundles.
In section $\ref{big-ladder}$ we prove that, under suitable hypotheses, ladders in a metric graph bundle are hyperbolic metric spaces when the 
metric graph bundle satisfies the properties of Lemma \ref{qc-level-set2}. To achieve this we first prove this result in
section $\ref{small-ladder}$ when the ladder is of small girth.
Then, to prove hyperbolicity in the general case, a ladder is decomposed into small-girth ladders
using qi sections. This gives a finite sequence of  hyperbolic metric spaces 
and we check that the conditions of Corollary $\ref{hamenstadt}$ are satisfied.

{\bf Notation and conventions:} 
We fix the following notation and conventions to be used till the end of section $4$.
For us $p:X \rightarrow B$ will be either an $f-$ metric graph bundle satisfying a flaring condition, or an
approximating ($f-$) metric graph bundle obtained from a metric bundle satisfying a
flaring condition.  \\
The symbols $g$, $\mu_k$ will have the same connotation as in Lemma $\ref{condition3}$ and Corollary $\ref{bdd-flaring}$
respectively.
We shall assume that $B$ is $\delta$-hyperbolic and each of the horizontal spaces
$F_b$ is $\delta^{'}$-hyperbolic for all vertices $b$ of $B$. We assume that the barycenter maps
$\partial^3F_b\rightarrow F_b$ are (uniformly) coarsely  surjective. Thus by  Proposition 
$\ref{existence-qi-section}$ we know that
the metric graph bundle admits a uniform ($K_0$, say)qi section through any point of $X$. 
 Lastly, often the dependence on these functions and constants will not be
explicitly stated if it is clear from context. 
By points in a graph we shall always mean vertices, unless otherwise specified.

\begin{lemma}\label{qi-section} For all $c_1 \geq 1$, there exists $C_{\ref{qi-section}}(=C_{\ref{qi-section}}(c_1))$ such that
the following holds.\\
Suppose $X_1$ and $X_2$ are two $c_1$-qi sections. 
Then through each point $x\in C(X_1,X_2)$ there exists a 
$C_{\ref{qi-section}}$-qi section
contained in $C(X_1,X_2)$.
\end{lemma}

\begin{proof} We already know that there is a $K_0$-qi section, say 
$Y_1$, through $x$ in $X$. Now define a new section $Y_2$ as follows:
 let 
$Y_2\cap F_b$ be  a nearest point projection, in the intrinsic metric on the horizontal space $F_b$,
of $Y_1\cap F_b$ onto the
horizontal geodesic $C(X_1,X_2)\cap F_b$. This defines a set theoretic
section. We need to check that this is indeed a qi section. For this
it is enough to check that $\forall \,b_1, b_2\in \mathcal{ V}(B)$, with
$d(b_1,b_2)= 1$, the distance between $F_{b_1}\cap Y_2$ and
$F_{b_2}\cap Y_2$ is uniformly bounded. This in turn follows immediately from  Lemma
$\ref{condition3}$, and  Lemma $\ref{qi-comm-proj}$ by choosing
$C_{\ref{qi-section}}:=c^{\prime}+ D_{\ref{qi-comm-proj}}(\delta^{\prime},g(c^{\prime}))$,
where $c^{\prime}=2 ~~max\{K_0,c_1\}$. \end{proof}

The proof of the previous lemma parallels a construction of  \cite{mitra-ct}. In our setting this
can be stated as follows:
Let $X_1, X_2$ be two $c_1-$qi-sections of a metric graph bundle $p: X \rightarrow B$, where 
each fiber (but not necessarily
the base $B$) is uniformly $\delta$-hyperbolic. Let $C(X_1, X_2)$ be the associated ladder.
By construction, $\lambda_b:= C(X_1, X_2) \cap F_b$ is a geodesic in the metric space
$(F_b, d_b)$. Define  $\pi_b:F_b \rightarrow \lambda_b$ as the nearest point projection of $F_b$ onto
the geodesic $\lambda_b$
 in the metric $d_b$.  Let $\Pi_{X_1,X_2}: X \rightarrow C(X_1, X_2)$ be given by 
$\Pi_{X_1,X_2} (x) = \pi_b (x), ~~~~\forall x \in F_b.$ Extend $\Pi_{X_1,X_2}$ to all other edges in the usual way by sending
the interior of an edge
to the image of one of its end-points.
The main technical theorem of \cite{mitra-ct} states

\begin{theorem} \cite{mitra-ct} For $X, B, p$ as above, and  $c_1 \geq 1$,
there exists $C \geq 1$ such that for two $c_1-$qi sections $X_1,X_2$, 
and $\forall x, y \in X$, $$d( \Pi_{X_1,X_2}(x), \Pi_{X_1,X_2}(y)) \leq Cd(x,y) + C.$$ 
Equivalently, $\Pi_{X_1,X_2}$ is a coarse Lipschitz retract
of $X$ onto $C(X_1,X_2)$.\label{retract} \end{theorem}

{\bf Simplification of Notation:}
We fix the following conventions and notation to be followed in
the rest of this section. Fix $c_1\geq K_0$.
Let $c_{i+1}=C^i_{\ref{qi-section}}(c_1)$, $i=1,2,3$ where $C^i_{\ref{qi-section}}$ is the
$i$-th iterate of the function $C_{\ref{qi-section}}$.
Note that if $Y$ is a $k$-qi section, and $k\leq c_4$, then it is also
a $c_4$-qi section. We know that our metric graph bundle satisfies the bounded flaring condition.
 We shall denote the function $\mu_{c_4}$ (see Corollary $\ref{bdd-flaring}$)  simply by $\mu$.

For two qi-sections $X_1,X_2$ in $X$ and $D\geq 0$ we shall denote
by $C_D(X_1,X_2)$ the $D$-neighborhood of the ladder $C(X_1,X_2)$ in $X$. 

From the definition of a ladder, we see that a ladder in a metric graph bundle is {\it not} connected. 
However, the first part of the next lemma says that a large enough neighborhood of a ladder is connected.

\begin{lemma}\label{important-remark} Let $X_1,X_2$ be  two $c_1-$qi-sections  in $X$. Then \\
1) For any $D\geq 2c_1$, $C_D(X_1,X_2)$ is connected.\\
2) Let $\gamma$ be a geodesic in $B$ and let $\tilde{\gamma}$ be its lift in  $X_1$.
Then, for any $D\geq 2c_1$, $l(\tilde{\gamma})\leq 2c_1.l(\gamma)$ where $l(\tilde{\gamma})$ is the length computed
in the $D$-neighborhood of the qi section $X_1$.\\
3) Let $X_3$ be a $c_2$-qi section lying inside  $C(X_1,X_2)$.
Then, for any $D\geq 2c_2$,  $X_3$ is the image of a $2c_2$-Lipschitz 
map from $\mathcal{V} (B)$ into $C_D(X_1,X_2)$ equipped with the 
path metric  induced from $X$.
In particular, it is a $2c_2$-qi section in  $C_D(X_1,X_2)$.
\end{lemma}

\begin{proof}
Let $s:\mathcal V(B)\rightarrow X$ be a $c_1$-qi section and let $b_1,b_2\in \mathcal V(B)$ be adjacent
vertices in $B$. Then $d(s(b_1),s(b_2))\leq c_1.d_B(b_1,b_2)+c_1=2c_1$. Conclusion (1) follows. \\
2) follows from (1). \\
3) The first statement follows by taking $c_2$ in place of $c_1$ in (1).
 Since the projection map
of the metric graph bundle $X$ to its base space $B$ is $1$-Lipschitz by definition,
it follows that a $c_2$-qi section lying inside $C(X_1,X_2)$ is a $2c_2$-qi section inside $C_D(X_1,X_2)$, where the latter is
equipped with the 
path metric  induced from $X$.
\end{proof}

\subsection{Hyperbolicity of ladders: Special case}\label{small-ladder}

This subsection is devoted to proving the hyperbolicity of small girth
ladders.  
Let $X_1,X_2$ be two $c_1$-qi sections in $X$ and let $d_h(X_1,X_2)\leq A_0$, say.
Let $A:= A^{'}_{\ref{qc-level-set2}}(c_4,A_0)$.
We further assume,
with reference to Lemma $\ref{qc-level-set2}$, that for any two $k$-qi sections 
$X_3,X_4$, $k\leq c_4$, lying inside the ladder $C(X_1,X_2)$, the set $U_A(X_3,X_4)\subset B$ is 
$K$-quasi-convex. We shall write simply $U(X_3,X_4)$ instead of $U_A(X_3,X_4)$
in what follows. Dependence of
constants in the various lemmas and propositions below on the  constants associated with the bundle  will be implicit rather than explicit.

The rest of this subsection is   devoted to proving the following: 
\begin{prop}\label{main-lem}  For all $L \geq 2c_4$, and $c_1, A_0$ as above,
there exist $\delta_{\ref{main-lem}}(=\delta_{\ref{main-lem}}(c_1,A_0,L))
\geq 0$, $K_{\ref{main-lem}}(=K_{\ref{main-lem}}(c_1,A_0,L)) \geq 0$, $D_{\ref{main-lem}}(=D_{\ref{main-lem}}(c_1,A_0,L)) \geq 0$
 such that we have the following:\\
(1) $C_L(X_1,X_2)$ is $\delta_{\ref{main-lem}}$-hyperbolic with the path metric  induced from $X$, and 
$X_1$,$X_2$ are $K_{\ref{main-lem}}$-quasiconvex in $C_L(X_1,X_2)$. \\
(2) If $d_h(X_1,X_2)\geq A^{''}_{\ref{qc-level-set2}}(c_1)$, then $X_1$, $X_2$ are
$D_{\ref{main-lem}}$-cobounded in $C_L(X_1,X_2)$.
\end{prop}

{\bf Idea of the proof}:
The proof of this proposition is rather long. 
Therefore, we shall break it up into several lemmas.
The idea  is as follows. We define a set of discrete paths $c(x,y)$, one
for each pair of points $x,y\in \mathcal V(X)\cap C(X_1,X_2)$ and check that they
satisfy the three properties of Corollary $\ref{hyp-lemma}$. 
Given $x,y\in \mathcal V(X)\cap C(X_1,X_2)$ first we construct two qi sections
through them. Then, $c(x,y)$ consists of three parts: two of them are in
the two sections containing $x,y$ and the other one is a horizontal 
geodesic of uniformly bounded length. Then any problem of length
computation is transferred to the sections. For instance
computing the Hausdorff distance between two paths or proving slimness of
triangles becomes easy when we apply this strategy to the parts of the paths that already
lie in a quasi-isometric section of the hyperbolic base space $B$.
 Lemma $\ref{qc-level-set2}$ and the bounded flaring condition are the main tools of
the proof.

We denote by $d^{'}$ the path metric on a neighborhood of a ladder induced from $X$.
Also $Hd^{'}$ will denote the Hausdorff distance between sets in a neighborhood of ladder and $Hd_B$ will 
denote Hausdorff distance between sets in $B$. 

{\bf Definition of  path family:} Let $x,y\in C(X_1,X_2)$ be two vertices.
By Lemma $\ref{qi-section}$ we can choose two
$c_2$-qi sections $X_3$ and $X_4$ through $x$ and $y$ respectively in $C(X_1,X_2)$.
Recall that $U(X_3,X_4)\subseteq \mathcal V(B)$ is a $K$-quasi-convex subset of $B$. 
Join $p(x)$ to $U(X_3,X_4)$ by a shortest geodesic $\gamma_{x,y}$ in $B$ ending at
 $b_{x,y} \in U(X_3,X_4)$. Let $\tilde{\gamma}_{x,y}$ be the lift of 
$\gamma_{x,y}$ in $X_3$, ending at $s_{x,y}$. 
Let $t_{x,y}$ be the lift of $b_{x,y}$ in $X_4$.
We note that $d_{b_{x,y}}(t_{x,y},s_{x,y})\leq A$. Now let $\beta_{x,y}$ be 
a geodesic in $B$ joining $p(y)$ and $b_{x,y}$, and let
$\tilde{\beta}_{x,y}$ be the lift of $\beta_{x,y}$ in $X_4$. We define
$c(x,y)$ to be the union of the three paths:
$\tilde{\gamma}_{x,y}$, $\tilde{\beta}_{x,y}$ and the sequence of consecutive vertices on the geodesic segment
$F_{b_{x,y}}\cap C(X_1,X_2)$ between $t_{x,y}$ and $s_{x,y}$.
We see that there is an asymmetry in the definition of $c(x,y)$
and a number of choices are involved.
However, for each unordered pair $\{x,y\}$ make the choices once and for all
and choose either $c(x,y)$ or $c(y,x)$ as the path joining the points
$x,y$. (See figure below.)

\medskip

\begin{center}

\includegraphics[height=4cm]{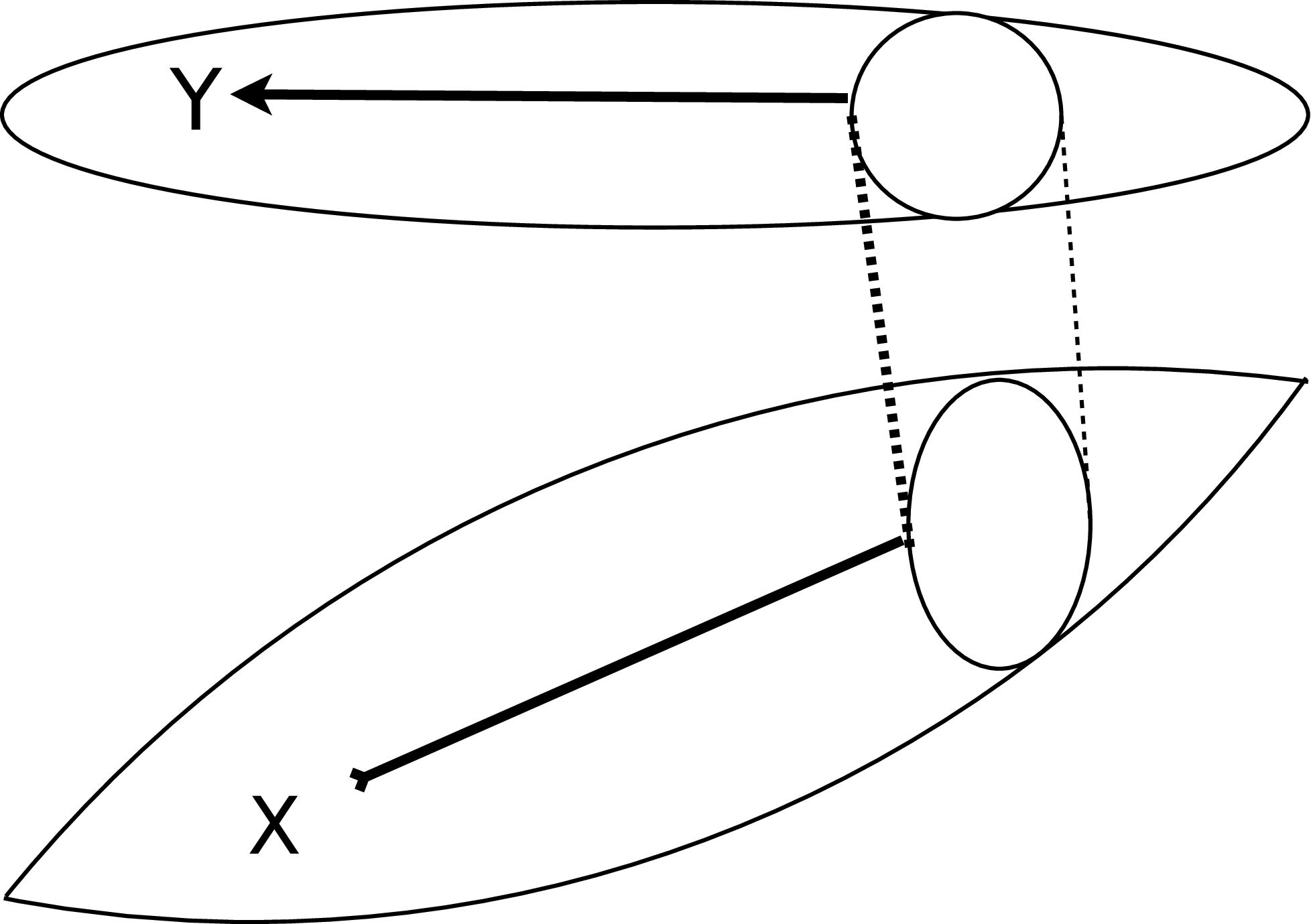}

\underline{{\it Path families: Special case} }
\end{center}

\begin{lemma}\label{property1}
Given $D_1\geq0$ there exist constants 
$D_{\ref{property1}}=D_{\ref{property1}}(c_1,A,D_1)$ and $D^{'}_{\ref{property1}}=D^{'}_{\ref{property1}}(c_1,A,D_1)$
such that the following holds:

Let $x,y\in \mathcal V(X)\cap C(X_1,X_2)$ with $d(x,y)\leq D_1 $. Then, $d^{'}(x,y)$- the distance between $x,y$ in 
the path metric on  $C_L(X_1,X_2)$, is bounded by  $D^{'}_{\ref{property1}}$. Moreover, the length of the path 
$c(x,y)$ is at most $D_{\ref{property1}}$. 
\end{lemma} 

\begin{proof}
Let $\tilde{y}$ be the lift of $p(x)$ in $X_4$. Since $p$ is a $1$-Lipschitz map,
 $d_B(p(y),p(\tilde{y}))\leq D_1$. Hence, $d(y,\tilde{y})\leq c_2.D_1+c_2$, since $X_4$ is a $c_2$-qi section. Therefore
 $d(\tilde{y},x)\leq c_2.D_1+D_1+c_2$. 
Then, since inclusions of the fibers of the map $p$ are uniformly metrically proper embeddings as measured
by $f$, we have  $d_{p(x)}(\tilde{y},x)\leq f(c_2.D_1+D_1+c_2)$.

By  Lemma $\ref{important-remark}(2)$,  $d^{'}(y,\tilde{y})\leq 2c_2.D_1$.
Thus $d^{'}(x,y)\leq 2c_2.D_1 + f(c_2.D_1+D_1+c_2)$ and the first part of the lemma is proved, with
$D^{'}_{\ref{property1}}:=2c_2.D_1 + f(c_2.D_1+D_1+c_2)$.

Next by Lemma $\ref{qc-level-set2}(2)$, we have
\[d_B(p(x),b_{x,y})\leq D^{'}_1:= D_{\ref{qc-level-set2}}(c_2,max\{A, f(c_2D_1+D_1+c_2)\}).\]
Thus $d_B(p(y),b_{x,y})\leq D_1+D^{'}_1$.
From this and Lemma $\ref{important-remark}(2)$, the second part of the lemma follows, with
$D_{\ref{property1}}:=A+2c_2.D_1+4c_2.D^{'}_1$. \end{proof}

\begin{rmk} {\em Note that in the first part of Lemma \ref{property1}, we have} not {\em assumed that
$C(X_1,X_2)$ is  of small girth.} \end{rmk}

We next show that the path family is {\bf coarsely well-defined}, i.e. ambiguities in
the definition of the paths can be ignored. More precisely, 
the different choices of paths joining the same pair of points are at a uniformly bounded Hausdorff
distance from each other.

Suppose $X_3,X^{'}_3$ are two $k$-qi sections in $C(X_1,X_2)$
containing $x$; and $X_4$ is a $k$-qi section containing $y$, where 
$k\leq c_4$.
Consider the two paths $c(x,y)$ 
and $c^{'}(x,y)$  joining $x,y$  defined using $X_3,X_4$ and $X^{'}_3,X_4$
respectively (defined as before).
 
Let $V:=U(X_1,X_2)$, $W:=U(X_3,X_4)$ and $W^{'}:=U(X^{'}_3,X_4)$. 
Then $V\subset W$, $V\subset W^{'}$. Join $p(x)$ to $V$ by a shortest 
geodesic $\gamma$ in $B$ and let $\tilde{\gamma}$, $\tilde{\gamma}^{'}$ be the
lifts of $\gamma$ in $X_3$ and $X^{'}_3$ respectively. Similarly join
$p(x)$ to $W,W^{'}$ respectively by shortest geodesics $\gamma_{x,y}$ and 
$\gamma^{'}_{x,y}$ and let $\tilde{\gamma}_{x,y}$ and $\tilde{\gamma}^{'}_{x,y}$
be their lifts in $X_3$ and $X^{'}_3$ respectively. Let $s_{x,y}$, 
$s^{'}_{x,y}$
be the end points of $\tilde{\gamma}_{x,y}$, $\tilde{\gamma}^{'}_{x,y}$
respectively, and let $b_{x,y}$, $b^{'}_{x,y}$ be the end points of 
$\gamma_{x,y}$ and $\gamma^{'}_{x,y}$.

\begin{lemma}\label{ladder-step1} With notation (in particular $k, A$) as above, there exists $D_{\ref{ladder-step1}}(=D_{\ref{ladder-step1}}(k,A))$
such that $d_B(b_{x,y},b^{'}_{x,y})$ is bounded by
$D_{\ref{ladder-step1}}$.
\end{lemma}
\begin{proof} Note that $V\subset W\cap W^{'}$, and that $V,W,W^{'}$ are all $K$-quasiconvex
subsets of $B$. Therefore, by Lemma $\ref{subqc-elem}$ (2), concatenating  $\gamma_{x,y}$ (resp. $\gamma^{'}_{x,y}$) 
with a geodesic joining $b_{x,y}$ (resp. $b^{'}_{x,y}$) to the terminal point of $\gamma$,
we obtain  $(3+2K)$-quasi-geodesics. These quasi-geodesics have the same end points as those of $\gamma$.
Since $B$ is a $\delta$-hyperbolic graph, by Lemma $\ref{stab-qg}$ we can find 
$b,b^{'}\in \gamma \cap \mathcal V(B)$, such that $d_B(b_{x,y},b)\leq D_3$, 
$d_B(b^{'}_{x,y},b^{'})\leq D_3$, where $D_3:=D_{\ref{stab-qg}}(\delta,3+2K)$.
If $b\in [p(x),b^{'}]\subset \gamma$ then 
$b_{x,y}\in N_{2.D_3+\delta}(\gamma^{'}_{x,y})$. Otherwise, $b^{'}\in [p(x),b]$,
so that $b^{'}_{x,y}\in N_{2.D_3+\delta}(\gamma_{x,y})$. Without loss of 
generality, let us assume that $b\in [p(x),b^{'}]$.

The end points of $\gamma$ are in $U(X_3,X^{'}_3)$ which is a $K$-quasi-convex
set in $B$. Hence by the bounded flaring condition (Corollary $\ref{bdd-flaring}$), we know that for
all points $b_2\in \gamma$, $ d_{b_2}(X_3\cap F_{b_2},X^{'}_3\cap F_{b_2})\leq A.\mu(K).$
In particular,  $ d_b(X_3\cap F_b,X^{'}_3\cap F_b)  \leq A.\mu(K)$. Similarly,
 $d_b(X_3\cap F_b,X_4\cap F_b)\leq A.\mu(D_3)$. Thus,
\[
\begin{array}{cl}
 d_b(X^{'}_3\cap F_b,X_4\cap F_b) & \leq d_{b}(X_3\cap F_b,X^{'}_3\cap F_b)+ d_b(X_3\cap F_b,X_4\cap F_b)\\
 & \leq A.\mu(K)+ A.\mu(D_3).
\end{array}
\]
We know that $[p(x),b^{'}] \subset N_{\delta + D_3}(\gamma^{'}_{x,y})$. 
Let $b^{'}_1\in \gamma^{'}_{x,y}\cap \mathcal V(B)$ be such that $d_B(b,b^{'}_1)\leq \delta + D_3$.
Then $d_{b^{'}_1}(X^{'}_3\cap F_{b^{'}_1},X_4\cap F_{b^{'}_1})\leq \mu(\delta +D_3).max\{d_b(X^{'}_3\cap F_{b},X_4\cap F_{b}),1\}$
and hence $d_{b^{'}_1}(X^{'}_3\cap F_{b^{'}_1},X_4\cap F_{b^{'}_1})\leq A.\mu(\delta +D_3)\{
\mu(D_3) +\mu(K)\}$.

Denoting the right hand side of the preceding inequality by $D^{'}$, we have,
by Lemma $\ref{qc-level-set2}(1)$, $d_B(b^{'}_1,b^{'}_{x,y})\leq D_{\ref{qc-level-set2}}(k,D^{'}).$
Since,
$d_B(b_{x,y},b^{'}_{x,y})\leq d_B(b_{x,y},b)+d_B(b,b^{'}_1)+d_B(b^{'}_1,b^{'}_{x,y})$,
therefore
\[ d_B(b_{x,y},b^{'}_{x,y})\leq D_3+ (\delta + D_3)  +D_{\ref{qc-level-set2}}(k,D^{'})=\delta+2D_3+D_{\ref{qc-level-set2}}(k,D^{'}). \]
Taking $D_{\ref{ladder-step1}}:=\delta+2D_3+D_{\ref{qc-level-set2}}(k,D^{'})$ completes the proof of the lemma. \end{proof}

\begin{lemma}\label{property2.1} With $k, A$ as above there exists $D_{\ref{property2.1}}(=D_{\ref{property2.1}}(k,A))$
such that
the Hausdorff distance between $c(x,y)$ and $c^{'}(x,y)$ is bounded by
 $D_{\ref{property2.1}}$. 
\end{lemma}

\begin{proof} 
{\bf Step $1$:} By Lemma \ref{ladder-step1} we have
$d_B(b_{x,y},b^{'}_{x,y})\leq D_{\ref{ladder-step1}}(k,A)$.
Hence, by  $\delta$-hyperbolicity of $B$, 
$Hd_B(\beta_{x,y},\beta^{'}_{x,y})\leq \delta +D_{\ref{ladder-step1}}(k,A)$.
Since $X_4$ is a $k$-qi section, we have, by Lemma $\ref{important-remark}(2)$,
\[Hd^{'}(\tilde{\beta}_{x,y},\tilde{\beta}^{'}_{x,y})\leq 2k.(\delta +D_{\ref{ladder-step1}}(k,A)).\]

{\bf Step $2$:}
Similarly,
\[
Hd^{'}([s_{x,y},t_{x,y}],[s^{'}_{x,y},t^{'}_{x,y}])\leq A+2k.D_{\ref{ladder-step1}}(k,A). \]
where $[s_{x,y},t_{x,y}],[s^{'}_{x,y},t^{'}_{x,y}]$ are the horizontal
geodesic segments of $c(x,y)$ and $c^{'}(x,y)$ respectively, each of length
at most $A$. 

{\bf Step $3$:}
Now we calculate the Hausdorff distance between 
$\tilde{\gamma}_{x,y}$ and $\tilde{\gamma}^{'}_{x,y}$. 
Let $\tilde{\gamma}^{''}_{x,y}$ be the lift of $\gamma_{x,y}$ in
$X^{'}_3$. Then, as in Step $1$, we have 
$Hd^{'}(\tilde{\gamma}^{'}_{x,y},\tilde{\gamma}^{''}_{x,y})\leq 2k.(\delta +D_{\ref{ladder-step1}}(k,A))$.
Since $\gamma$ joins two points of $U(X_3,X^{'}_3)$ which is $K$-quasi-convex in $B$, it follows that 
$d_{b_2}(X_3\cap F_{b_2},X^{'}_3\cap F_{b_2})\leq A.\mu(K)$ for all points $b_2\in \gamma$
 by the bounded flaring condition. 
Since there is a point $b\in \gamma$ such that
$d_B(b,b_{x,y})\leq D_3:=D_{\ref{stab-qg}}(\delta,3+2K)$, we have the following using the boundedness of the flaring condition again:
\[
\begin{array}{rcl}
d_{b_{x,y}}(X_3\cap F_{b_2},X^{'}_3\cap F_{b_2}) & \leq & \mu(D_3).max\{d_b(X_3\cap F_b,X^{'}_3\cap F_b),1\} \\
&\leq & A.\mu(D_3).\mu(K).
\end{array}
\]
Let $A_1=A.\mu(D_3).\mu(K)$; then $U_{A_1}(X_3,X^{'}_3)$ is $K^{'}:=K_{\ref{qc-level-set2}}(c_2,A_1)$-quasi-convex.
Note that $\gamma_{x,y}$ joins two points of $U_{A_1}(X_3,X^{'}_3)$.
Therefore, by Lemma $\ref{qc-level-set2}$ (1) and the bounded flaring condition, we have for all
$ b_1\in \gamma_{x,y}$, 
$d_{b_1}(X_3\cap F_{b_1},X^{'}_3\cap F_{b_1})\leq \mu(K^{'}).A_1$. 
Hence 
$Hd^{'}(\tilde{\gamma}_{x,y},\tilde{\gamma}^{''}_{x,y})\leq \mu(K^{'}).A_1$, and  therefore
\[
\begin{array}{rcl}
Hd^{'}(\tilde{\gamma}_{x,y},\tilde{\gamma}^{'}_{x,y}) &\leq & Hd^{'}(\tilde{\gamma}_{x,y},\tilde{\gamma}^{''}_{x,y})+Hd^{'}(\tilde{\gamma}^{'}_{x,y},\tilde{\gamma}^{''}_{x,y})\\
& \leq & \mu(K^{'}).A_1+  2k.(\delta +D_{\ref{ladder-step1}}(k,A)).
\end{array}
\]
Finally, since 
\[
\begin{array}{l}
Hd^{'}(c(x,y),c^{'}(x,y))\\ 
\leq \mbox{max}\{Hd^{'}(\tilde{\beta}_{x,y},\tilde{\beta}^{'}_{x,y}), Hd^{'}([s_{x,y},t_{x,y}],[s^{'}_{x,y},t^{'}_{x,y}]),Hd^{'}(\tilde{\gamma}_{x,y},\tilde{\gamma}^{'}_{x,y})\},
\end{array}
\] 
the lemma follows, taking $D_{\ref{property2.1}}:=\mu(K^{'}).A_1+2k.(\delta +D_{\ref{ladder-step1}}(k,A))$. \end{proof}

\begin{lemma}\label{property2.2} With notation  (in particular $k, A$)  as above, there exists 
 $D_{\ref{property2.2}}(=D_{\ref{property2.2}}(k,A))$  such that
if $c(x,y),c(y,x)$ are defined using two $k$-qi sections $X_3,X_4$ where $k\leq c_4$, $x\in X_3$ and $y\in X_4$,  then
$Hd(c(x,y),c(y,x))$ is bounded by  $D_{\ref{property2.2}}$.  
\end{lemma}

\begin{proof}
Let $\alpha$ be a geodesic in $B$ joining $b_{x,y}$ and $b_{y,x}$.
Since $\alpha$ joins two points of $U(X_3,X_4)$ which is a $K$-quasiconvex subset of $B$,  
we have:\\
i)  $d_b(F_b\cap X_3, F_b\cap X_4)\leq \mu(K)A$ for all $ b\in\alpha$,
by  the bounded flaring
condition for metric graph bundles.\\
ii) $\gamma_{x,y}\cup \alpha$
is a $(3+2K)$-quasi-geodesic by Lemma $\ref{subqc-elem}$ (2).
Hence,
$Hd_B(\gamma_{x,y}\cup \alpha,[p(x),b_{y,x}])\leq D_{\ref{stab-qg}}(\delta,3+2K)$,
by Lemma $\ref{stab-qg}$. 
Similarly, $Hd_B(\gamma_{y,x}\cup \alpha,[p(y),b_{x,y}])\leq D_{\ref{stab-qg}}(\delta,3+2K)$.

Therefore, for all $z\in \tilde{\gamma}_{y,x}$, $p(z)$ is in the $D_{\ref{stab-qg}}(\delta, 3+2K)$-neighborhood
of $[p(y),b_{x,y}]$. Thus $z$ is contained in the $2k.D_{\ref{stab-qg}}(\delta, 3+2K)$-neighborhood of 
$\tilde{\beta}_{x,y}\subset c(x,y)$ using the fact that $X_4$ is a $k$-qi section and Lemma 
\ref{important-remark} (2).  

Again for all
$z\in F_{b_{y,x}}\cap \mathcal V(C(X_1,X_2))$, $d_{b_{y,x}}(z, s_{y,x})\leq A$. 
It follows that in this case $z$ is contained
in the $(A+2k.D_{\ref{stab-qg}}(\delta, 3+2K))$-neighborhood of $c(x,y)$. 

Now, suppose 
$z\in \tilde{\beta}_{y,x}$. Since $B$ is $\delta$-hyperbolic, $p(z)\in N_{\delta}(\gamma_{x,y}\cup \alpha)$.
If $p(z)\in N_{\delta}(\gamma_{x,y})$ then $z$ is contained in the $2k.\delta$-neighborhood of
$\tilde{\gamma}_{x,y}\subset c(x,y)$. Otherwise, $p(z)\in N_{\delta}(\alpha)$. Suppose $b_1\in \alpha$ such that
$d_B(p(z),b_1)\leq \delta$. As in the first paragraph of the proof we have 
$d_{b_1}(F_{b_1}\cap X_3, F_{b_1}\cap X_4)\leq \mu(K)A$. Using the fact that 
$Hd_B(\gamma_{x,y}\cup \alpha,[p(x),b_{y,x}])\leq D_{\ref{stab-qg}}(\delta,3+2K)$ 
we see that $z$ is contained in the $(\mu(K)A+ 2k(\delta+D_{\ref{stab-qg}}(\delta,3+2K)))$-neighborhood of
$\tilde{\beta}_{x,y}\subset c(x,y)$.


It follows that $c(y,x)$ is contained in the $(\mu(K)A+ 2k(\delta+D_{\ref{stab-qg}}(\delta,3+2K)))$-neighborhood of
$c(x,y)$. Similarly it follows that $c(x,y)$ is contained in the $(\mu(K)A+ 2k(\delta+D_{\ref{stab-qg}}(\delta,3+2K)))$-neighborhood of $c(y,x)$. Hence 
$Hd^{'}(c(x,y),c(y,x))\leq D_{\ref{property2.2}}:=\mu(K)A+ 2k(\delta+D_{\ref{stab-qg}}(\delta,3+2K)).  $
\end{proof}

\begin{cor}\label{curve-well-defined} With notation (in particular $k, A$) as above, there exists 
$D_{\ref{curve-well-defined}}(=D_{\ref{curve-well-defined}}(k,A))$ such that the following holds. \\
Let $x,y\in C(X_1,X_2)$. Then the Hausdorff distance
between any pair of paths joining $x,y$ defined in the same way as 
 $c(x,y)$ using $k$-qi sections passing through $x,y$, is at most
$D_{\ref{curve-well-defined}}$.
\end{cor}

\begin{proof} Choose $D_{\ref{curve-well-defined}}:=2(D_{\ref{property2.1}}+ D_{\ref{property2.2}})$. \end{proof}

\begin{lemma}\label{triangle-thin}  With notation (in particular $k, A$) as above, there exists 
$D_{\ref{triangle-thin}}(=D_{\ref{triangle-thin}}(k,A))$ such that the following holds. \\
Suppose  $X_3,X_4,X_5$ are $k$-qi sections
in $C(X_1,X_2)$ such that $z\in X_5$, $y\in X_4 \subset C(X_1,X_5)\subset C(X_1,X_2)$ and 
$x\in X_3\subset C(X_1,X_4)\subset C(X_1,X_2)$.
Then the triangle formed by the paths $c(x,y),c(y,z),c(x,z)$, defined
using the pairs $X_3,X_4$; $X_4,X_5$ and $X_3,X_5$ respectively,
is $D_{\ref{triangle-thin}}-$slim.
\end{lemma}

 \begin{proof}
We have $U(X_3,X_5)\subset U(X_4,X_5)\cap U(X_3,X_4)$ and we know that all of these three sets are
$K$-quasi-convex in $B$.

{\bf Case $1$:} {\em Suppose $x,y$ are in the same horizontal space
and $d_{p(x)}(x,y)\leq A$}. \\ 
Then  $p(x)\in U(X_3,X_4)$.
Since $\gamma_{x,z}$ ends in $U(X_3,X_5)\subset U(X_3,X_4)$, it
joins two points of $U(X_3,X_4)$ which we know is $K$-quasi-convex.
Hence, by Corollary $\ref{bdd-flaring}$, we have
for all $b^{'}\in \gamma_{x,z}$,
$d_{b^{'}}(X_3\cap F_{b^{'}},X_4\cap F_{b^{'}})\leq A.\mu(K)$.

Now we show that $d_B(b_{x,z},b_{y,z})$ is small. Recall that 
$b_{x,z}\in U(X_3,X_5)\subset U(X_4,X_5)$ and $b_{y,z}\in U(X_4,X_5)$.
Thus $\gamma_{y,z}\cup [b_{y,z}, b_{x,z}]$ is a $(3+2K)$-quasi-geodesic in $B$,
by Lemma $\ref{subqc-elem}$ (2).
Hence, there is a point $b_2\in \gamma_{x,z}$,
such that $d_B(b_{y,z},b_2)\leq D_{\ref{stab-qg}}(\delta, 3+2K)$, by Lemma 
$\ref{stab-qg}$.
Since $d_{b_{y,z}}(F_{b_{y,z}}\cap X_4,F_{b_{y,z}}\cap X_5)\leq A$,
we have by  bounded flaring,
$d_{b_2}(F_{b_2}\cap X_4,F_{b_2}\cap X_5)\leq A.\mu(D_{\ref{stab-qg}}(\delta, 3+2K))$.
Therefore, $
d_{b_2}(F_{b_2}\cap X_3,F_{b_2}\cap X_5)=
d_{b_2}(F_{b_2}\cap X_3,F_{b_2}\cap X_4)+d_{b_2}(F_{b_2}\cap X_4,F_{b_2}\cap X_5)
\leq A.\{\mu(K)+\mu(D_{\ref{stab-qg}}(\delta,3+2K)\}$.
Now, by  Lemma $\ref{qc-level-set2}$ (1), we get
\[d_B(b_2,b_{x,z})\leq D_{\ref{qc-level-set2}}(k,A.\mu(K)+A.\mu(D_{\ref{stab-qg}}(\delta,3+2K)).
\]
Hence 
\[d_B(b_{x,z},b_{y,z})\leq D_{\ref{stab-qg}}(\delta,3+2K)+ D_{\ref{qc-level-set2}}(k,A.\mu(K)+A.\mu(D_{\ref{stab-qg}}(\delta,3))).
\]
It follows by arguments similar to that of Lemma \ref{property2.1} that
\[ Hd^{'}(c(x,z),c(y,z))\leq A.\mu(K)+ 2k.(\delta+d_B(b_{x,z},b_{y,z})). \]
Replacing $d_B(b_{x,z},b_{y,z})$ in the right hand expression  with its upper bound obtained above,  we get a constant  
$D^{'}_{\ref{triangle-thin}}$. Hence the lemma follows in this case by choosing
$D_{\ref{triangle-thin}}\geq D^{'}_{\ref{triangle-thin}}$.

{\bf Case $2$:} We consider the general case.
For the rest of the proof, we shall assume that all the paths
of the form $c(u,v)$ ($u,v \in X_3\cup X_4 \cup X_5$),
are constructed using the sections $X_3, X_4, X_5$ only, unless otherwise specified.
We first show that $Hd^{'}(c(x,z),\tilde{\gamma}_{x,y}\cup c(s_{x,y},z))$
is bounded by a constant depending on $k$ and $A$.

Let $\bar{b}$ be a nearest point projection of $b_{x,y}$ on $U(X_3,X_5)$.
By Lemma $\ref{subqc}$,
$d_B(\bar{b},b_{x,z})\leq D_{\ref{subqc}}=D_{\ref{subqc}}(\delta,K)$.
Let $\gamma_2$
be a geodesic joining $b_{x,y}$ to $\bar{b}$ and let $\widetilde{\gamma_2}$
be a lift of $\gamma_2$ in $X_3$. Note that
$\gamma_{x,y}\cup \gamma_2$ is a $(3+2K)$-quasi geodesic in $B$.
Thus the Hausdorff distance between $\gamma_{x,z}$ and $\gamma_{x,y}\cup \gamma_2$
is at most $\delta+D_{\ref{subqc}}+D_{\ref{stab-qg}}(\delta,3+2K)$. 
Hence the Hausdorff distance between $\tilde{\gamma}_{x,y}\cup c(s_{x,y},z)$
and $c(x,z)$, in $C_L(X_1,X_2)$, is at most
$2k.\{\delta+D_{\ref{subqc}}+D_{\ref{stab-qg}}(\delta,3+2K)\}+A=D_1$, say. 

Again by case $1$, we know that
$Hd(c(s_{x,y},z),c(t_{x,y},z))\leq D^{'}_{\ref{triangle-thin}}$.
Hence, $Hd(c(x,z),   \tilde{\gamma}_{x,y}\cup [s_{x,y},t_{x,y}]\cup c(t_{x,y},z))$ is
at most $A+D_1+D^{'}_{\ref{triangle-thin}}$. 
Also, if we define the paths $c(z,t_{x,y})$, $c(z,y)$
with respect to the sections $X_4,X_5$ by taking
$\gamma_{z,t_{x,y}}=\gamma_{z,y}$, the triangle formed by the 
paths $c(z,t_{x,y})$, $c(z,y)$ and $\tilde{\beta}_{y,t_{x,y}}$ is
$2k\delta$-slim. 

Thus by Corollary $\ref{curve-well-defined}$,
the triangle formed by the
paths $\tilde{\beta}_{y,t_{x,y}}$, $c(t_{x,y},z)$ and $c(y,z)$ is $D_2$-slim
where $D_2=2k.\delta+ 2.D_{\ref{curve-well-defined}}$.
Taking $D_{\ref{triangle-thin}}:= A+D_1+D^{'}_{\ref{triangle-thin}}+D_2$,
the  lemma follows. \end{proof}

\bigskip

\noindent {\bf Proof of Proposition $\ref{main-lem}$}

\smallskip

We verify that the set of paths $\{c(x,y)\}$ defined earlier in this
section satisfies the properties of Corollary $\ref{hyp-lemma}$. Then, as per the notation of Corollary $\ref{hyp-lemma}$,
let $D=L$, $C_1:=2c_2$,
$\Phi (N)=D_{\ref{property1}}(c_2,A,N)$ and $C_2= D_{\ref{triangle-thin}}(c_4,A)+2.D_{\ref{curve-well-defined}}(c_4,A)$.

{\bf Proof of properties $1$ and $2$:} These follow from Lemma $\ref{important-remark}(2)$ and Lemma
$\ref{property1}$ respectively.

{\bf Proof of property $3$:} Suppose $x,y\in C(X_1,X_2)$.
If $x^{'},y^{'}\in c(x,y)$ then  the segment of
$c(x,y)$ between $x^{'},y^{'}$, say $c(x,y)|_{[x^{'},y^{'}]}$, is
a possible candidate for the definition of $c(x^{'},y^{'})$.
Hence by Corollary $\ref{curve-well-defined}$, the Hausdorff distance
of $c(x,y)|_{[x^{'},y^{'}]}$ and $c(x^{'},y^{'})$ is bounded by
$D_{\ref{curve-well-defined}}(c_2,A)\leq C_2$.

{\bf Proof of property $4$:} Let $x,y,z\in C(X_1,X_2)$.
Then using Lemma $\ref{qi-section}$ we may assume, without loss
of generality, that $x,y,z$ are contained in three $c_4$-qi
sections $X_3,X_4,X_5$ respectively, where 
$X_4\subseteq C(X_1,X_5)$, $X_3\subseteq C(X_1,X_4)$. Now,
the triangle formed by the paths $c(x,y)$, $c(y,z)$, $c(x,z)$
defined using these sections is  $D_{\ref{triangle-thin}}(c_4,A)$-slim
by Lemma $\ref{triangle-thin}$. Hence, by  Corollary
$\ref{curve-well-defined}$,  any triangle with vertices $x, y, z$
formed by such paths is
$\{D_{\ref{triangle-thin}}(c_4,A)+2.D_{\ref{curve-well-defined}}(c_4,A)\}$-slim.
It follows from Corollary $\ref{hyp-lemma}$ 
 that $C_L(X_1,X_2)$ is  $\delta_{\ref{main-lem}}$-hyperbolic for some  $\delta_{\ref{main-lem}}\geq 0$.

By Lemma \ref{important-remark}(3), it follows that $X_1,X_2$ are the images of 
$2c_1-$quasi-isometric embeddings of $B$ into the
 $\delta_{\ref{main-lem}}$-hyperbolic metric
space $C_L(X_1,X_2)$. Thus, they are
$K_{\ref{main-lem}}:=D_{\ref{stab-qg}}(\delta_{\ref{main-lem}},2c_1)$-quasiconvex in $C_L(X_1,X_2)$.
This completes the proof of the first statement of the proposition.

From the given conditions it follows by Lemma $\ref{qc-level-set2}$ (2) that  $U(X_1,X_2)$ is  bounded. 
Hence, for any  $x\in X_1$
and $y\in X_2$ the $K_{\ref{hyp-lemma}}$-quasi-geodesic $c(x,y)$ passes through the (uniformly)
bounded set $p^{-1}(U(X_1,X_2))\cap C(X_1,X_2)$ (by Corollary $\ref{hyp-lemma}$). Since $C_L(X_1,X_2)$ has been
proven to be hyperbolic,  stability of quasi-geodesics (Lemma \ref{stab-qg})
completes the proof of the second statement of the proposition. 
$\Box$

\subsection{Hyperbolicity of ladders: General case}\label{big-ladder}

\begin{lemma}\label{easy-lemma}
There is a function $D_{\ref{easy-lemma}}:\mathbb R^{+}\rightarrow \mathbb R^{+}$
such that the following holds.\\ Suppose $I,J$ are intervals in $\mathbb R$ 
and $\phi:I\rightarrow J$ is a $k$-quasi-isometric embedding. Let 
$x_1,x_2,x_3\in I$, $x_1\leq x_2\leq x_3$, and suppose $\phi(x_1)$ belongs to
the interval with end points  $\phi (x_2)$, $\phi (x_3)$.
Then $x_2-x_1\leq D_{\ref{easy-lemma}}(k)$. 
\end{lemma}

\begin{proof} Without loss of generality, we may assume that
$\phi(x_2)\leq \phi(x_1)\leq \phi(x_3)$.
Let $x_4=inf\{y\in [x_2,x_3]:\phi(y)\geq \phi (x_1)\}$.

If $x_2=x_4$ then $\exists x^{'}\in [x_2,x_2+1]\cap [x_2,x_3]$
such that $\phi (x^{'})\geq \phi(x_1)$. Now $x^{'}-x_2\leq 1$ implies
$|\phi(x^{'})-\phi(x_2)|\leq 2k$, since $\phi$ is a $k$-quasi-isometric embedding.
Therefore,  $\phi(x_1)-\phi(x_2)\leq 2k$. Thus we have
$x_2-x_1\leq 3k^{2}$.

If $x_2<x_4$ we choose $x^{'}\in [x_2,x_4)$ and $x^{''}\in [x_4,x_3]$
such that $x_4-x^{'}\leq 1$ and $x^{''}-x_4\leq 1$ with
$\phi(x^{''})\geq \phi(x_1)$. Now $x^{''}-x^{'}\leq 2$ implies
$|\phi(x^{'})-\phi(x^{''})|\leq 3k$. Thus $\phi(x^{''})-\phi(x_1)\leq 3k$,
since $\phi(x^{'})< \phi(x_1)\leq \phi(x^{''})$ by the choices of $x^{'}, x^{''}$. Hence $x_2-x_1\leq x^{''}-x_1\leq 4k^2$.
Therefore, in any case, we may choose $D_{\ref{easy-lemma}}(k)=4k^2$.
\end{proof}

\begin{lemma}\label{ladder-decomp}  Given 
$f: {\mathbb{N}} \rightarrow  {\mathbb{N}}$, $k \geq 1, D \geq 2C_{\ref{qi-section}}(k)$, there exists  $D^{'}_{\ref{ladder-decomp}}
=D^{'}_{\ref{ladder-decomp}}(f,k,D)\geq 1$
such that the following holds.\\
Suppose $p:X\rightarrow B$ is an $f$-metric graph bundle
and $X_1,Y,X_2$ are $k$-qi sections in $X$.
Also suppose that $Y$ is contained in the ladder $C(X_1,X_2)$.
Then the $D$-neighborhood of  each of the spaces $Y$, $C(X_1,Y),C(Y, X_2)$ is a connected
subgraph of $X$ and the intersection of the spaces
$C_D(X_1,Y)$ and $C_D(Y,X_2)$ is contained in the $D^{'}_{\ref{ladder-decomp}}$-neighborhood of $Y$ in the
path metric of both $C_D(X_1,Y)$ and $C_D(Y,X_2)$.
\end{lemma}

\begin{proof}
Since $X_1,X_2,Y$ are $k$-qi sections and $D \geq 2C_{\ref{qi-section}}(k)$, it follows from
Lemma $\ref{important-remark}(1)$ that the $D$-neighborhood of  each of the spaces $Y$, $C(X_1,Y),C(Y, X_2)$ is connected.

Now, let $y\in C_D(X_1,Y) \cap C_D(Y,X_2)$. Let us denote the path metric on $C_D(X_i,Y)$ induced
from $X$ by $d_i$ and suppose $y_i\in C(X_i,Y)$ be such that $d_i(y,y_i)\leq D$, for $i=1,2$.
Then $d_B(p(y_1),p(y_2))\leq 2D$. 
We need to prove the statements: \\ ${\mathcal{P}}_j:$ {\em Any point of $C_D(X_1,Y)\cap C_D(Y,X_2)$ is
contained in a $D^{'}$-neighborhood of $Y$ in $C_D(X_j,Y)$, for $j=1,2$.} 

Since the proofs of ${\mathcal{P}}_1, {\mathcal{P}}_2$ are similar, we shall only prove ${\mathcal{P}}_2$.
We know that there exists a $k^{'}=C_{\ref{qi-section}}(k)$-qi section $Y_2$ say, through 
$y_2\in C(X_2,Y)$ contained in $C(X_2,Y)$.
Join $y_2$ to the point $y^{'}_1=Y_2\cap F_{p(y_1)}$, by the lift of a geodesic in $B$ joining
$p(y_1)$ and $p(y_2)$. The length of this path is at most $4Dk^{'}$ by Lemma \ref{important-remark}(2).  Then
$d(y_1,y^{'}_1)\leq 2D+4Dk^{'}$ and hence their horizontal distance is at most
$f(2D+4Dk^{'})$ by the bounded flaring condition for metric graph
bundles. Thus choosing $D^{'}_{\ref{ladder-decomp}}$ to be $D+f(2D+4Dk^{'})$, we are through. \end{proof}


Suppose $X_1,X_2$ are any two $c_1$-qi sections in $X$.
Let us define the notation $c_{i+1}=C^{i}_{\ref{qi-section}}(c_1)$, $i\geq 1$, as in the proof of 
Proposition $\ref{main-lem}$. Then we have the following. 

\begin{prop}\label{main-prop} 
For any $L\geq 2c_6$, and $c_1 \geq 1$ as above, there exists  $\delta_{\ref{main-prop}}=\delta_{\ref{main-prop}}(c_1,L)$ such that
$C_L(X_1,X_2)$ is a $\delta_{\ref{main-prop}}$-
hyperbolic metric space with respect to the path metric induced from $X$.
\end{prop}

\begin{proof} Let $A= A^{''}_{\ref{qc-level-set2}}(c_3)+ D_{\ref{easy-lemma}}(2g(c_3))+ f(2L+4c_3L)$.
The idea of the proof is to break the ladder $C(X_1,X_2)$
into a finite number of subladders. Then by Proposition $\ref{main-lem}$ and, if necessary, by a simple application of
Corollary $\ref{hamenstadt}$ we show that these subladders are hyperbolic. 
Finally we apply Corollary $\ref{hamenstadt}$ again to the ladder assembled out of subladders to finish the proof.

{\bf Step $1:$} {\it Defining  subladders.}\\
Fix a horizontal geodesic ${\mathcal I} = F_{b_0}\cap C(X_1,X_2)$.
 The two end points of $\mathcal I$ lie in $X_1$ and $X_2$. 
Choose a parametrization $\alpha: [0,l]\rightarrow \mathcal I$ 
by arc length so that $\alpha(0)\in X_1$ and $\alpha(l)\in X_2$. 
We shall inductively construct a finite sequence of integers
$0=s_0<s_1<\cdots <s_m=l$,  and a sequence
of $c_2$-qi sections $X^{'}_i$ contained in $C(X_1,X_2)$ such that $X^{'}_i$ passes through $\alpha(s_i)$
for each $i=1,\cdots ,m-1$. Let $X^{'}_0=X_1$. 
Suppose $s_i$ has been obtained, $s_i<l$ and $X^{'}_i$ has been constructed.
If $d_h(X^{'}_i,X_2)\leq A$ then define $s_{i+1}=l$, $X^{'}_{i+1}=X_2$
and the construction is over. Otherwise, consider the set
\[
S_{i+1}=\{t\in [s_i,l]\cap \mathbb N: \exists \,\mbox{a} \,c_2-\mbox{qi section} \, X^{'} \mbox{through}\, \alpha(t) \,\mbox{with} \,d_h(X^{'},X^{'}_i)\leq A\}
\]
Let $u_{i+1}=\mbox{max}\, S_{i+1}$. If $\exists t\in S_{i+1}$ such that
there is a $c_2$-qi section $X^{'}$ inside $C(X_1, X_2)$ through $\alpha(t)$ with
$d_h(X^{'},X^{'}_i)= A$, define $s_{i+1}=t$ and $X^{'}_{i+1}= X^{'}$.
Otherwise define $s_{i+1}=min\{l,u_{i+1}+1\}$ and let $X^{'}_{i+1}$ be
any $c_2$-qi section inside $C(X_1, X_2)$  through $\alpha(s_{i+1})$.
The construction of these sections stops at the $m$-th step if
$d_h(X^{'}_{m-1},X_2)\leq A$, so that we must have $X^{'}_m=X_2$ and
$s_m=l$.
It follows from the above construction of the sections $X^{'}_i$
that for each $i$, $1\leq i\leq m-1$, we have $d_h(X^{'}_{i-1},X^{'}_i)\geq A$
and in case $d_h(X^{'}_i,X^{'}_{i+1})> A$, there is a section $X^{''}_i$ through
a point $\alpha(t_i)$, $t_i\in [s_i,s_{i+1}]$ with 
$d_h(X^{'}_j,X^{''}_i)\leq A$, $j=i,i+1$.

{\bf Step $2:$} {\it Subladders form a decomposition of  $C(X_1,X_2)$.}\\
In this step, we will show that 
$C(X_1,X_2)=\cup_{i=0}^{m-1}C(X^{'}_i,X^{'}_{i+1})$ and that $C(X^{'}_{i-1},X^{'}_i)\cap C(X^{'}_i, X^{'}_{i+1})=X^{'}_i$.

Note that the first assertion follows from the second and the construction in Step 1. For the second
assertion, it is enough to show the following:\\
{\bf Claim:} $X^{'}_{i+1} \subseteq C(X^{'}_i,X_2)$, for all $i$, $1\leq i\leq m-2$. 

Consider the triples of points 
$(X_1\cap F_b, X^{'}_i\cap F_b, X^{'}_{i+1}\cap F_b)$, $b\in \mathcal V(B)$. They are
contained in the geodesic $F_b\cap C(X_1,X_2)$. For $b=b_0$ we know, by  the construction in Step 1, that 
$X^{'}_i\cap F_b\in [X_1\cap F_b, X^{'}_{i+1}\cap F_b]$. 

We now argue by contradiction.  Suppose
$X^{'}_{i+1}\not \subseteq C(X^{'}_i,X_2)$. Then for some point
$b^{'}\in \mathcal V(B)$, we must have
$X^{'}_{i+1}\cap F_{b^{'}}\in [X_1\cap F_{b^{'}}, X^{'}_i\cap F_{b^{'}}]$.
Therefore there exist $b_1,b_2\in \mathcal V(B)$ with $d(b_1,b_2) = 1$,
such that
$X^{'}_i\cap F_{b_1}\in [X_1\cap F_{b_1}, X^{'}_{i+1}\cap F_{b_1}]$
but $X^{'}_{i+1}\cap F_{b_2}\in [X_1\cap F_{b_2}, X^{'}_i\cap F_{b_2}]$.
We know that $X^{'}_i,X^{'}_{i+1}$ are $c_2$-quasi-isometric sections, and 
$X_1$ is a $c_1$-quasi-isometric section. Hence 
$d( X^{'}_i\cap F_{b_1},X^{'}_i\cap F_{b_2})\leq 2c_2$,
$d( X^{'}_{i+1}\cap F_{b_1},X^{'}_{i+1}\cap F_{b_2})\leq 2c_2$ and 
$d(X_1\cap F_{b_1},X_1\cap F_{b_2})\leq 2c_1\leq 2c_2$. 

By Lemma $\ref{qi-section}$, the definition of $c_3$ (at the beginning of the proof of this proposition) and
Lemma $\ref{condition3}$, we have a $g(2c_3)$-quasi-isometric embedding
$[X_1\cap F_{b_1}, X^{'}_{i+1}\cap F_{b_1}]\rightarrow
[X_1\cap F_{b_2}, X_2\cap F_{b_2}]$ which sends each of the points
$X_j^{'}\cap F_{b_1}$ to $X_j^{'}\cap F_{b_2}$, $j=i,i+1 $ and $X_1\cap F_{b_1}$ to $X_1\cap F_{b_2}$. 
By Lemma $\ref{easy-lemma}$ we get
\[ d_{b_1}( X^{'}_i\cap F_{b_1}, X^{'}_{i+1}\cap F_{b_1})\leq
D_{\ref{easy-lemma}}(g(2c_3)).\]
By the choice of the constant $A$, and the definition of $X^{'}_i$'s
this gives rise to a contradiction, completing the proof of Step 2.

{\bf Step $3:$} {\it Subladders are uniformly hyperbolic.}\\
Next we show that there are constants $\delta_1$, $k_1$ and $D$ such that
$(i)$ each $C_L(X^{'}_i, X^{'}_{i+1})$ is $\delta_1$-hyperbolic and
$X^{'}_i, X^{'}_{i+1}$ are  $k_1$-quasi-convex in $C_L(X^{'}_i, X^{'}_{i+1})$ for each
$i,\,0\leq i\leq m-1$. $(ii)$ Also we shall show that the sets $X^{'}_i, X^{'}_{i+1}$ are mutually
$D$-cobounded in $C_L(X^{'}_i, X^{'}_{i+1})$, for $0\leq i\leq  m-1$.

$(i)$ Since $X^{'}_i,X^{'}_{i+1}$ are $c_2$-qi sections in $X$, it follows that
they are the images of  $2c_2-$quasi-isometric embeddings  in $C_L(X^{'}_i,X^{'}_{i+1})$ (Lemma \ref{important-remark}(3)). 
Hence, they will be $D_{\ref{stab-qg}}(\delta_1,2c_2)$-quasiconvex in
$C_L(X^{'}_i,X^{'}_{i+1})$ provided we can show that $C_L(X^{'}_i, X^{'}_{i+1})$ is $\delta_1$-hyperbolic. 

If $d_h(X^{'}_i,X^{'}_{i+1})\leq A$ then, by Proposition $\ref{main-lem}$,
each  $C_L(X^{'}_i,X^{'}_{i+1})$ is $\delta_{\ref{main-lem}}(c_2,A,L)$-hyperbolic;
moreover, in this case, unless $i=m-1$, we have $d_h(X^{'}_i,X^{'}_{i+1})=A$ and $X^{'}_i,X^{'}_{i+1}$ are then mutually
$D_{\ref{main-lem}}(c_2,A,L)$-cobounded.

Suppose $d_h(X^{'}_i,X^{'}_{i+1})>A$. Recall that $X^{'}_j$ passes
through $\alpha(s_j)$, $j=i,i+1$. In this case, we can find 
$t_i\in [s_i,s_{i+1}]$ such that there is a $c_2$-qi section
$X^{''}_i$ in $C(X_1,X_2)$, passing through $\alpha(t_i)$, so that 
$d_h(X^{'}_j,X^{''}_i )\leq A$, $j=i,i+1$. Now, as in the proof of Lemma
$\ref{qi-section}$,
we project points of $X^{''}_i$ into the horizontal geodesics of 
$C(X^{'}_i,X^{'}_{i+1})$ and get a $c_3$-qi section $Y^{'}_i$
through $\alpha(t_i)$. Note that we still have 
$d_h(X^{'}_j,Y^{'}_i )\leq A$ for $j=i,i+1$. By Proposition $\ref{main-lem}$,
$C_L(X^{'}_i,Y^{'}_i)$, and $C_L(X^{'}_{i+1},Y^{'}_i)$ are both
$\delta_{\ref{main-lem}}(c_3,A,L)$-hyperbolic. Also we see that 
$C_L(X^{'}_i,Y^{'}_i)\cap C_L(X^{'}_{i+1},Y^{'}_i)$ contains a $2c_3$-neighborhood of  $Y^{'}_i$ which is
connected. Since $Y^{'}_i$ is 
a $c_3$-quasi-isometric image of $B$ in $X$,  therefore it is a $2c_3$-quasi-isometric image in both 
$C_L(X^{'}_i,Y^{'}_i)$ and $ C_L(X^{'}_{i+1},Y^{'}_i)$.

Now, we apply  Lemma $\ref{ladder-decomp}$ followed by Corollary $\ref{hamenstadt}$. Here the total space
is $C_L(X^{'}_i,X^{'}_{i+1})$ and we have just two subspaces:
$C_L(X^{'}_i,Y^{'}_i)$ and  $C_L(X^{'}_{i+1},Y^{'}_i)$. Also their intersection contains  a $2c_3$-neighborhood
of $Y^{'}_i$, denoted by $Y_i$, say. We see that the rest of the conditions of Corollary $\ref{hamenstadt}$ are easily verified.

Thus, $C_L(X^{'}_i,X^{'}_{i+1})$
is $\delta_{\ref{hamenstadt}}(\delta_{\ref{main-lem}}(c_3,A,L),D^{'}_{\ref{ladder-decomp}}(f,c_3,L),1,2c_3)$-hyperbolic. 
Choosing
$$\delta_1:=max\{\delta_{\ref{main-lem}}(c_2,A,L), \delta_{\ref{hamenstadt}}(\delta_{\ref{main-lem}}(c_3,A,L),D^{'}_{\ref{ladder-decomp}}(f,c_3,L),1,2c_3)\}$$
completes the proof of Step 3(i).

$(ii)$ We next show that the quasi-convex sets  $X^{'}_i,X^{'}_{i+1}$ are mutually cobounded in
$C_L(X^{'}_i,X^{'}_{i+1})$. 

Since the sets $U(X^{'}_j,Y^{'}_i)$, $j=i,i+1$ are $K(=K_{\ref{qc-level-set2}}(c_3,A))$-quasiconvex in $B$, the lift $Y_{ij}$ (say) of
$U(X^{'}_j,Y^{'}_i)$ in $Y^{'}_i$ is a $C_1:=(2Kc_3+D_{\ref{stab-qg}}(\delta_1,2c_3))$-quasi-convex set
in $C_L(X^{'}_i,X^{'}_{i+1})$. 

{\bf Claim:} {\em There are constants $R=R(\delta_1,C_1)$, $D_1=D_1(\delta_1,C_1)$ such that if $Y_{ij}$, $j=i,i+1$ are $R$-separated then the sets
$X^{'}_j$, $j=i,i+1$ are $D_1$-cobounded.}

{\em Proof of Claim:}  We show that the projection of  $X^{'}_{i+1}$ on  $X^{'}_i$ is uniformly bounded. By a symmetric
 argument the projection of  $X^{'}_{i}$ on  $X^{'}_{i+1}$ is uniformly bounded.

Suppose $x\in X^{'}_{i+1}$ and let $y\in X^{'}_i$ be a nearest point projection of $x$ on $X^{'}_i$. 
Let $x_1$ be a
nearest point projection of $x$ on $Y^{'}_i$ and let $y_1$ be a
nearest point projection of $y$ on $Y^{'}_i$. 

{\bf Sub-claim 1:} The curve $[x,x_1]\cup [x_1,y_1] \cup [y_1,y]$ is a uniform quasi-geodesic if  $R$ is sufficiently large.

{\em Proof of Sub-claim 1:} By Lemma \ref{subqc-elem} (2) the unions $[x,x_1]\cup [x_1,y_1]$ and $[x_1,y_1]\cup [y_1,y]$ are
$(3+2C_1)$-quasi-geodesics.  Sub-claim 1 will follow from the fact that
$d(x_1,y_1)\geq L_{\ref{local-global-qg}}(\delta_1,3+2C_1,3+2C_1)$ for large enough $R$ (by Lemma \ref{local-global-qg}).

By Lemma \ref{cobdd-lemma}, if the sets $Y_{ij}$ are $R$-separated, $R\geq R_{\ref{cobdd-lemma}}(\delta_1,C_1)$
then there are points $y_{ij}\in Y_{ij}$, $j=i,i+1$ such that every geodesic connecting the sets $Y_{ij}$, $j=i,i+1$ passes through
the $D_{\ref{cobdd-lemma}}(\delta_1,C_1)$-neighborhood of $y_{ij}$, $j=i,i+1$.
Applying this  to the geodesic $[x_1,y_1]$, Sub-claim 1 follows from the following. 

{\bf Sub-claim 2:} Suppose $x^{'}_j\in X^{'}_j$, and let $y^{'}_j\in Y^{'}_i$ be its nearest point projection on  $Y^{'}_i$ for $j=i,i+1$ .
Then $y_{ij}$ is uniformly close to the geodesic $[x^{'}_j, y^{'}_j]$, $j=i,i+1$.

{\em Proof of Sub-claim 2:} Since the proofs are similar, let us  prove the statement for $j=i$. Let $b$ be a nearest
point projection of $p(x^{'}_i)$ on $U(X^{'}_i,Y^{'}_i)$. Let $\alpha$ be a geodesic in $B$ joining $p(x^{'}_i)$ and $b$.
Let $\beta$ be a geodesic joining $p(y^{'}_i)$ and $b$. Let $\tilde{\alpha} $ and $\tilde{\beta}$ be the
lifts of $\alpha$ and $\beta$ in $X^{'}_i$ and $Y^{'}_i$ respectively. Let $\tilde{\alpha}\cap p^{-1}(b) =z_i$ and
$\tilde{\beta}\cap p^{-1}(b) =w_i$. Then $d_b(z_i,w_i)\leq A$. 
The paths $\tilde{\alpha}$ and $\tilde{\beta}$ are $2c_1$-quasi-geodesics in $C_L(X^{'}_i,X^{'}_{i+1})$.
Hence, by hyperbolicity of  $C_L(X^{'}_i,X^{'}_{i+1})$
there exist
$x^{''}_1\in [x^{'}_i,y^{''}_i]$, $x^{''}_2\in \tilde{\alpha}$, $x^{''}_3\in \tilde{\beta}$ which are uniformly close to each other
(cf. Lemmas \ref{stab-qg},  \ref{hyp-defn}).
Then, it follows as in the first paragraph of the
proof of Lemma \ref{property1} that $d_{p(x^{'}_i)}(X^{'}_i,Y^{'}_i)$ is uniformly bounded. Hence $y_{ii}$ is close to $x^{'}_i$ by
Lemma \ref{qc-level-set2} (1). Sub-claim 2 follows. $\Box$

Since $C_L(X^{'}_i,X^{'}_{i+1})$ is hyperbolic the Hausdorff distance between the quasi-geodesic $[x,x_1]\cup [x_1,y_1] \cup [y_1,y]$ and 
the geodesic $[x,y]$ is uniformly bounded. Hence the points $y_{ii}$ and $y_{ii+1}$ are uniformly close to the
geodesic $[x,y]$ by Sub-claim 2. The Claim follows. $\Box$

Finally, note that if $Y_{ij}$, $j=i,i+1$ are {\em not} $R$-separated
then there exists a pair of points in $X^{'}_i$ and $X^{'}_{i+1}$  which are at a distance  of
at most $A^{'}_1:=(2A+R)$ from each other. It follows as in the first paragraph of the proof of Lemma \ref{property1} 
that $d_h(X^{'}_i,X^{'}_{i+1})\leq A_1:=f(2A^{'}_1 c_2+ A^{'}_1)$. 
Hence, by Proposition $\ref{main-lem}$, $X^{'}_i,X^{'}_{i+1}$ are $D_{\ref{main-lem}}(c_2,A_1,L)$-cobounded.

It follows that any geodesic joining $X^{'}_j$, $j=i, i+1$ passes close to the end points of this coarsely unique geodesic
and step $3$ follows.

{\bf Step $4:$} {\it The final step:}\\
Finally we use  Lemma $\ref{ladder-decomp}$ in conjunction with  Corollary $\ref{hamenstadt}$.
Here the total space is  $C_L(X_1,X_2)$, and
the sequence of subspaces are
$C_L(X^{'}_{i},X^{'}_{i+1})$, $i=0,1,\ldots, m-1$. We check to see that the hypotheses
of Corollary $\ref{hamenstadt}$ are satisfied:\\
$(1)$ Each of the subspaces $C_L(X^{'}_{i},X^{'}_{i+1})$ is $\delta_1$-hyperbolic by step $3$;\\ 
$(2)$ by choice of the constant
$A> f(2L+4c_3L)$ (see 
Lemma \ref{property1}) we know that only the consecutive ones intersect nontrivially;\\
 $(3)$ for $i=2,\ldots, m$, the intersection of  
two consecutive subspaces $C_L(X^{'}_{j},X^{'}_{j+1})$, $j=i-1,i$, contains the $2c_2$-neighborhood $Y_i$ (say), of $X^{'}_i$. Also $Y_i$ is connected
(Lemma \ref{property1}). Further the intersection is contained in the
$D^{'}_{\ref{ladder-decomp}}(f,c_3,L)$-neighborhood of $Y_i$ in the spaces $C_L(X^{'}_{j},X^{'}_{j+1})$, $j=i-1,i$;\\
$(4)$ To check Condition (4) of Corollary $\ref{hamenstadt}$ it is enough to show the following:
Suppose $Z\subset X$ is a connected subgraph such that $Y_i\subset Z$. Then the inclusion $Y_i\hookrightarrow Z$ is 
uniform qi embedding.

The inclusion of $Y_i$ in the space $Z$ is clearly distance
decreasing. Let $x,y \in Y_i$ and choose $x_1,y_1\in X^{'}_i$ such that $d(x,x_1)\leq 2c_2$,
$d(y,y_1)\leq 2c_2$. Suppose $d_Z(x,y)=n$. 
Then $d_X(x_1,y_1)\leq d_Z(x_1,y_1) \leq n+4c_2$. Hence $d_B(p(x_1),p(y_1))\leq d_X(x_1,y_1)\leq n+4c_2$. 
Since $X^{'}_i$ is a $c_2$-qi section in $X$, 
by Lemma \ref{important-remark}(2) there is a path of length $2c_2(n+4c_2)$ joining $x_1$ and $y_1$ contained in 
$Y_i$. Hence we have $d_{Y_i}(x,y)\leq 2c_2(n+4c_2)+4c_2=2c_2.n+ 12c_2$. This proves (4). 
\\ 
$(5)$ the sets $X^{'}_{i},X^{'}_{i+1}$ are uniformly cobounded in $C_L(X^{'}_{i},X^{'}_{i+1})$ for $i=1,2,\ldots, m-2$   as proved in Step 3.

The proposition follows. \end{proof}


\section{The Combination Theorem}

As  in Section $3$, we assume the following for the purposes of this section:\\
1) $p:X \rightarrow B$ will be either an $f-$ metric graph bundle satisfying a flaring condition, or an
approximating ($f-$) metric graph bundle obtained from a metric bundle satisfying a
flaring condition.\\
2) $B$ is $\delta$-hyperbolic and  the horizontal spaces
$F_b$ are $\delta^{'}$-hyperbolic for all vertices $b$ of $\mathcal{V} (B)$. \\
3) The barycenter maps
$\partial^3F_b\rightarrow F_b$ are (uniformly) coarsely  surjective. Thus by  Proposition 
$\ref{existence-qi-section}$ we know that
the metric graph bundle admits uniform ($K_0$, say) qi sections through each point of $X$. \\
In this section we prove the main theorem of our paper which says that
{\it a metric (graph) bundle satisfying the above conditions
has  hyperbolic  total space.}

Here is an outline of the main steps of the proof:\\
For each pair of points $x,y\in X$, choose a ladder $C(X_1,X_2)$ containing $x,y$ and choose a
geodesic $c(x,y)$ in $C_D(X_1,X_2)$ joining $x,y$ (with $D$ large enough but fixed). This gives a family
of curves. We shall show that the family satisfies the conditions of Corollary $\ref{hyp-lemma}$. Proofs of conditions $1$ and $2$
follow from the results of the last section. Proofs of conditions $3$ and $4$ follow from 
Proposition $\ref{final-prop0}$ below, which contains the statement that large neighborhoods of `tripod bundles' 
are hyperbolic. Proposition $\ref{final-prop0}$ in turn follows from Proposition $\ref{main-lem}$ and Corollary $\ref{hamenstadt}$.

\begin{defn}
For three qi sections $X_1,X_2,X_3$ in a metric graph bundle $X$ over $B$ a {\bf tripod bundle} determined by 
these qi sections, denoted $C(X_1,X_2,X_3)$, is defined to be the union of the ladders $C(X_1,X_2)$,
$C(X_2,X_3)$, $C(X_3, X_1)$.
\end{defn}
The convention that we adopted in Remark \ref{defn-ladder-rem}  applies here as well; namely, since the
Hausdorff distance between any two tripod bundles determined by three qi sections is uniformly bounded (by hyperbolicity of the
fibers),
we denote by 
$C(X_1,X_2,X_3)$ any tripod bundle determined by the qi sections $X_1,X_2,X_3$.
 Also for any qi sections $X_1,X_2,X_3$ in $X$ and $D\geq 0$ we denote by $C_D(X_1,X_2,X_3)$
the $D$-neighborhood of the tripod bundle $C(X_1,X_2,X_3)$ in $X$.


The main technical tool of this section is the following:

\begin{prop} \label{final-prop0} Let $X$ over $B$ be an $(f,K)$-metric graph bundle such that\\
i) $X$ is either a metric graph bundle
satisfying a flaring condition or one obtained as an approximating metric graph bundle of a metric bundle
satisfying a flaring condition; \\
ii) $B$ is $\delta$-hyperbolic and  the horizontal spaces
$F_b$ are $\delta^{'}$-hyperbolic for all vertices $b$ of $\mathcal{V} (B)$. \\
iii) the barycenter
maps $\partial^3F_b\rightarrow F_b$ are (uniformly) coarsely  surjective. \\
Given $c_1 \geq 1$,  
there exists $L_0, \delta_{\ref{final-prop}}, K_{\ref{final-prop}} \geq 0$
such that the following holds.\\
Let $X_1,X_2,X_3$ be $c_1$-qi sections and $L\geq L_0$. Then
\begin{enumerate}
\item 
$C_L(X_1,X_2,X_3)$ is $\delta_{\ref{final-prop}}(= \delta_{\ref{final-prop}}(c_1,L))$-hyperbolic with the path metric
 induced from $X$ and each of
$C_L(X_i,X_j)$, $i\neq j$ is $K_{\ref{final-prop}}(=K_{\ref{final-prop}}(c_1,L))$-quasi-convex in
$C_L(X_1,X_2,X_3)$.
\item there exists 
 $D_{\ref{ladder-qi-embed}}(= D_{\ref{ladder-qi-embed}}(c_1,L))$ such that if
$x,y \in C_L(X_1,X_2)$,    $\gamma_1$ is  a geodesic in $C_L(X_1,X_2,X_3)$ joining $x,y$  and
$\gamma_2$  is a geodesic
in $C_L(X_1,X_2)$  joining $x,y$, then the Hausdorff distance $Hd(\gamma_1,\gamma_2)\leq D_{\ref{ladder-qi-embed}}$. 
\item there exists $D_{\ref{curv-defn}}
(=D_{\ref{curv-defn}}(c_1,L))$  such that if
 $X_i,X^{'}_i$, $i=1,2$ are $c_1$-qi sections and $x_i\in X_i \cap X^{'}_i$, $i=1,2$,
then the Hausdorff distance between the geodesics joining $x_1,x_2$ in the subspaces
$C_L(X_1,X_2)$ and $C_L(X^{'}_1,X^{'}_2)$ is at most $D_{\ref{curv-defn}}(c_1,L)$. 
\end{enumerate}
\end{prop}

We postpone the proof of Proposition \ref{final-prop0} to Section \ref{tech}. Conclusions (1), (2), (3) above
form the content of Proposition \ref{final-prop}, 
 Corollary \ref{ladder-qi-embed} and 
 Corollary \ref{curv-defn} respectively. We give the proof of the main combination Theorem assuming Proposition \ref{final-prop0}.

\begin{theorem}\label{combthm}
Suppose $p:X\rightarrow B$ is a metric bundle (resp. metric graph bundle) such that\\
$(1)$ $B$ is a $\delta$-hyperbolic metric space.\\
$(2)$ Each of the fibers $F_b$, $b\in B$ (resp. $b\in \mathcal{ V}(B)$) is a $\delta^{'}$-hyperbolic 
metric space with respect to the path metric induced from $X$. \\
$(3)$ The barycenter maps $\partial^3F_b\rightarrow F_b$, $b\in B$ (resp. $b\in \mathcal{ V}(B)$) are (uniformly) coarsely  surjective.\\
$(4)$ A flaring condition is satisfied.\\
Then $X$ is a hyperbolic metric space.
\end{theorem}

\begin{proof} If $X$ is a metric bundle, we first replace $X$ by an approximating metric graph bundle. Abusing notation slightly, we continue to
call the approximating metric graph bundle $X$. 
By Proposition \ref{existence-qi-section}, there exists
 $c_1\geq 1$  such that there is a $c_1$-qi section through each point of $\mathcal V(X)$.

Let $L=L_0$ be the constant given by Proposition \ref{final-prop0} (1). 
We shall now define a set of curves joining  pairs
of points $x,y\in X$.

\noindent {\bf Definition of  curve family:} For each pair of points $x$, $y$ in $\mathcal{V}(X)$,
choose, once and for all, two $c_1$-qi sections $X_1$, $X_2$ passing through 
$x$ and $y$ respectively. Now define $c(x,y)$ to be consecutive vertices on a geodesic
in $C_L(X_1,X_2)$ joining $x$, $y$. We show that the family $\{ c(x,y) \}$ satisfies properties (1)-(4) of 
Corollary $\ref{hyp-lemma}$ to complete the proof. As per the notation of Corollary $\ref{hyp-lemma}$, set $D=L$.

\begin{itemize}
\item {\bf Proof of property $1$:} This follows by taking $C_1= 1$.
\item {\bf Proof of property $2$:} By the first part of Lemma $\ref{property1}$, Property $2$ follows.
\item {\bf Proof of property $3$:} 
This follows from  Conclusions (1) and (2) of Proposition $\ref{final-prop0}$.
\item {\bf Proof of property $4$:} Given $x,y,z\in X$
choose three $c_1$-qi sections $X_3,X_4,X_5$ containing $x,y,z\in X$ respectively and define
the curves $c^{'}(x,y)$, $c^{'}(x,z)$ and $c^{'}(y,z)$ using these sections
in the same way as the curves $c(x,y)$ are defined.
It follows from Conclusion (2) of Proposition $\ref{final-prop0}$ that the triangle formed by 
$c^{'}(x,y)$, $c^{'}(x,z)$ and $c^{'}(y,z)$ is $(\delta_{\ref{final-prop}}(c_1,L)+2D_{\ref{ladder-qi-embed}}(c_1,L))$-slim.
Conclusion (3) of Proposition $\ref{final-prop0}$ now gives property $4$.
\end{itemize}

Hyperbolicity of $X$ now follows from Corollary $\ref{hyp-lemma}$.
\end{proof}

\begin{rmk} 
{\em
Note that the conditions of Theorem \ref{combthm} are inherited by induced metric graph
bundles over quasi-isometrically embedded subsets of $B$. Hence the induced bundles over quasi-isometrically embedded subsets of $B$
are also hyperbolic. }\end{rmk}

\subsection{Proof of Proposition \ref{final-prop0}} \label{tech}
Suppose that $X_1$,$X_2$ and $X_3$ are three $c_1$-qi sections in $X$. The main aim of this subsection is to
 show that for large
$D\geq 0$, $C_D(X_1,X_2,X_3)$ is hyperbolic.
For this, we  first show that taking a nearest point projection of $X_3\cap F_b$ onto the horizontal geodesic
$C(X_1,X_2)\cap F_b$ (for all $b\in \mathcal{V} (B)$) 
we get a qi section $X_4$. (See figure below.) Then we have a genuine 'tripod bundle' $C(X_1,X_2)\cup C(X_3,X_4)$,
such that $C_D(X_1,X_2,X_3)$ is quasi-isometric to an $L$-neighborhood of $C(X_1,X_2)\cup C(X_3,X_4)$, where $L$
depends on $D$ and the bundle. The quasi-isometry is provided by projecting any point $z$ of 
 $C_L(X_1,X_2)\cup C_L(X_3,X_4)$ onto a nearest point in  $C_D(X_1,X_2,X_3)$ lying in the same horizontal fiber as $z$ (Here,
the nearest point-projection
is taken in the metric on the horizontal fiber to which $z$ belongs.)  
Hyperbolicity of the space $C_L(X_1,X_2)\cup C_L(X_3,X_4)$, and quasi-convexity of $C(X_1,X_2)$ in this space
essentially follow from Proposition $\ref{main-prop}$ and Corollary $\ref{hamenstadt}$.

\medskip

\begin{center}

\includegraphics[height=40mm]{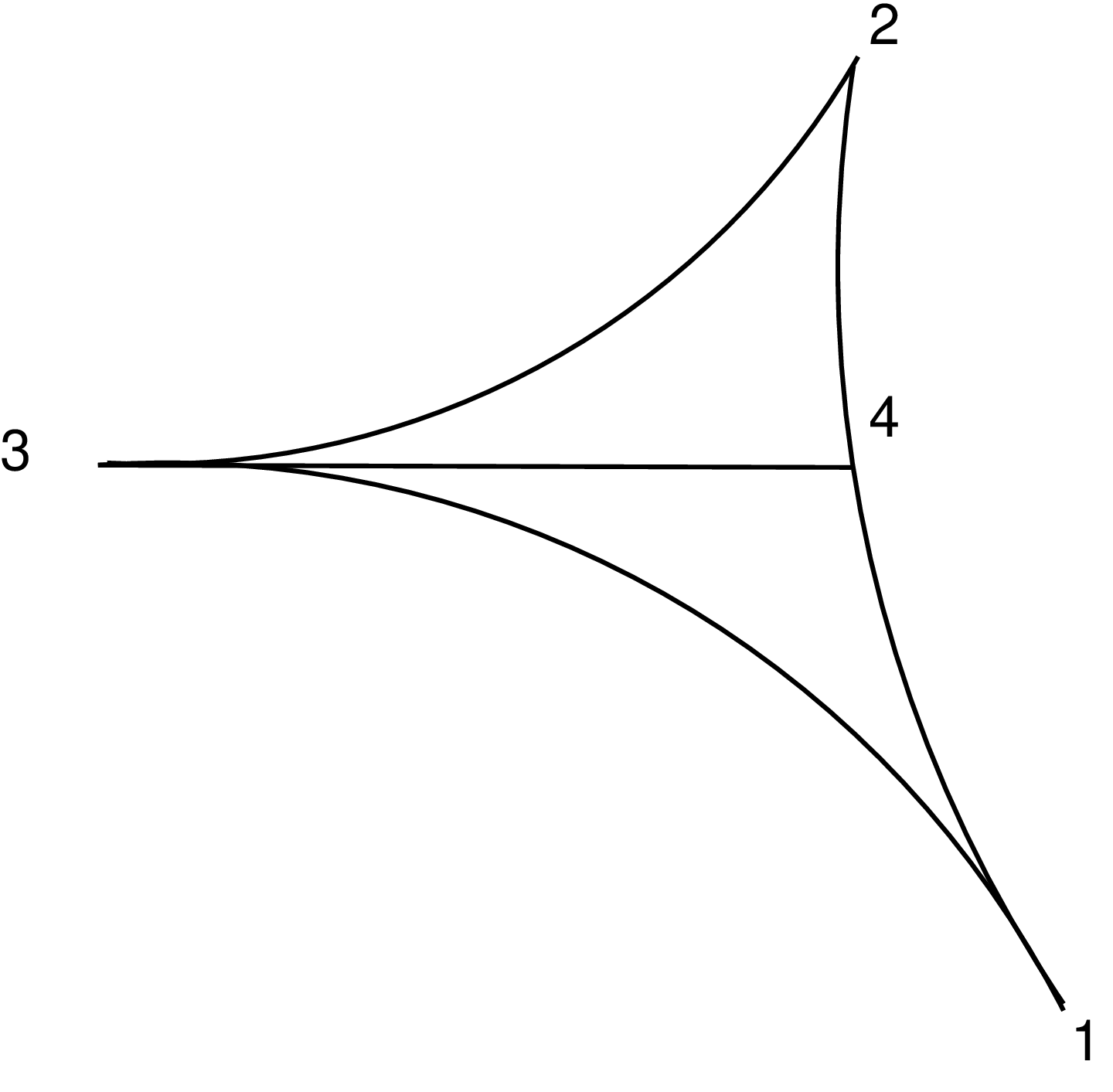}

\underline{{\it Tripod} }

\end{center}

Conclusion (1) of Proposition \ref{final-prop0} is given by the following.

\begin{prop}\label{final-prop} Given $c_1 \geq 1$,  
there exists $L_0, \delta_{\ref{final-prop}}, K_{\ref{final-prop}} \geq 0$
such that the following holds.\\
Let $X_1,X_2,X_3$ be $c_1$-qi sections and $L\geq L_0$. Then
$C_L(X_1,X_2,X_3)$ is $\delta_{\ref{final-prop}}(= \delta_{\ref{final-prop}}(c_1,L))$-hyperbolic with the induced path metric from $X$ and each of
$C_L(X_i,X_j)$, $i\neq j$ is $K_{\ref{final-prop}}(=K_{\ref{final-prop}}(c_1,L))$-quasi-convex in
$C_L(X_1,X_2,X_3)$.
\end{prop}

For ease of exposition, we break the proof up into several lemmas, many of which
 will be minor modifications of results we have shown already. 

For $b_1,b_2 \in \mathcal{V} (B)$ with $d(b_1,b_2)= 1$,
we have a $g(2c_1)$-quasi-isometry $F_{b_1}\rightarrow F_{b_2}$
by Lemma $\ref{condition3}$, which sends $X_i\cap F_{b_1}$ to $X_i\cap F_{b_2}$ for $i=1,2,3$.
Therefore, by Lemma $\ref{qi-comm-proj}$, choosing a nearest point projection of $X_3\cap F_b$
onto the horizontal geodesic $[X_1\cap F_b,X_2\cap F_b]$, for all $b\in B$, we get a
$c^{'}_1$-qi section of $B$ in $X$ where
$c^{'}_1:=2c_1+ D_{\ref{qi-comm-proj}}(\delta^{'}, g(2c_1))$.
Let us call this section $X_4$. Let $c^{'}_{i+1}:=C^i_{\ref{qi-section}}(c^{'}_1)$, $i\geq 1$.

Now we have the following analog of Lemma $\ref{ladder-decomp}$.
\begin{lemma}\label{ladder-decomp2}
For all $L\geq 2c^{'}_2$, there exists $D_{\ref{ladder-decomp2}}(=D_{\ref{ladder-decomp2}}(L))$ such that
the intersection $C_L(X_1, X_2)\cap C_L(X_3,X_4)$ is contained in the
$D_{\ref{ladder-decomp2}}$-neighborhood of $X_4$.
\end{lemma}

\begin{proof}
The proof is an exact copy of that of the proof of Lemma $\ref{ladder-decomp}$.
The only  observation we need to make is that the curve $[X_3\cap F_b, X_4\cap F_b]\cup [X_4\cap F_b, X_i\cap F_b]$,
$i=1,2$ is a $(3,0)$-quasi-geodesic in $F_b$ (Lemma \ref{subqc-elem} (1)).
\end{proof}

\begin{lemma}\label{T-space}
For all $c_1$ as above and  $L\geq 2c^{'}_6$, there exist $D_{\ref{T-space}}(=D_{\ref{T-space}}(c_1,L))$ 
and  $K_{\ref{T-space}} (=K_{\ref{T-space}}(c_1,L))$ such that
 the space $C_L(X_1,X_2)\cup C_L(X_3,X_4)$ is $D_{\ref{T-space}}$-hyperbolic 
and $C(X_1,X_2)$ is $K_{\ref{T-space}}$-quasi-convex in this space.
\end{lemma}

\begin{proof}
The first part of the lemma follows as an application of Proposition $\ref{main-prop}$ and
Corollary $\ref{hamenstadt}$
(the proof is a replica of Step $3$ of the proof of Proposition $\ref{main-lem}$ which shows that  large girth subladders 
are hyperbolic). For completeness we briefly check the conditions of Corollary $\ref{hamenstadt}$.
\begin{enumerate}
\item Here we have only two subgraphs
$C_L(X_1,X_2)$ and $ C_L(X_3,X_4)$ which are hyperbolic by Proposition $\ref{main-lem}$.
\item Condition $(2)$ is trivially satisfied.
\item The intersection $C_L(X_1,X_2)\cap C_L(X_3,X_4)$ contains the $c^{'}_2$-neighborhood, say $Y$, of $X_4$
which is connected and the rest follows from Lemma $\ref{ladder-decomp2}$ above.
\item Since $X_4$ is $2c^{'}_1$-quasi-isometrically embedded in $C_L(X_1,X_2)\cup C_L(X_3,X_4)$, $Y=N_{2c^{'}_1}(X_4)$ is
also quasi-isometrically embedded.
\item Condition $5$ is trivially satisfied.
\end{enumerate}

For the second part of the lemma we note that any geodesic joining two points of $C(X_1,X_2)$ in 
$C_L(X_1,X_2)\cup C_L(X_3,X_4)$ and which leaves $C_L(X_1,X_2)$ must join two points
in a (uniformly bounded) neighborhood of $X_4$, by Lemma $\ref{ladder-decomp2}$.
Since $X_4$ is the image of a quasi-isometric embedding of $B$ in the
hyperbolic space $C_L(X_1,X_2)\cup C_L(X_3,X_4)$ it is quasi-convex also. The lemma follows.
\end{proof}

Clearly for all $b\in \mathcal V(B)$, $C(X_1,X_2,X_3)\cap F_b$ is $\delta^{'}$-quasi-convex in $F_b$.
Define a map $\Pi : Z=C_L(X_1,X_2)\cup C_L(X_3,X_4) \rightarrow X$ by sending any point 
$x \in Z\cap F_b$  to a nearest point in $C(X_1,X_2,X_3)\cap F_b$ (in the $d_b-$metric).

\begin{lemma}\label{pi-lip}  Given $c_1 \geq 1$ there exists $D_{\ref{pi-lip}}(=D_{\ref{pi-lip}}(c_1))$ such that 
the map $\Pi$ is $D_{\ref{pi-lip}}-$coarsely Lipschitz.
\end{lemma}

\begin{proof}
We need to check that for any two adjacent vertices in the domain of $\Pi$, the image vertices are at
a uniformly bounded distance.
This breaks up into two cases. 

When the vertices are in the same horizontal space $F_b$ then since 
$C(X_1,X_2,X_3)\cap F_b$ is  (uniformly) quasiconvex in $F_b$, and since nearest point projections onto quasiconvex sets
in hyperbolic metric spaces are coarsely Lipschitz (cf. Lemma 3.2 of \cite{mitra-trees})
the claim follows.

When the vertices are not in same horizontal space then the same argument as in 
Lemma $\ref{qi-comm-proj}$ (also see \cite{mitra-trees, bowditch-ct}) shows that nearest-point projections and quasi-isometries
almost commute. The rest of the proof is a replica  of Theorem \ref{retract} 
\cite{mitra-ct}.
\end{proof}

\begin{rmk} In fact, 
 $\Pi$ restricted to $C(X_1, X_2)$ is simply an inclusion map.
Hence by Lemma \ref{property1}, $\Pi$ is a qi-embedding of $C(X_1, X_2)$ into any sufficiently large neighborhood of $C(X_1,X_2,X_3)$
 equipped with a path metric induced from $X$. \label{pi-lip-rmk} \end{rmk}

\begin{lemma}\label{proj-dist} Given $c_1 \geq 1$ and $L \geq 0$ as above there exists $D_{\ref{proj-dist}}(=D_{\ref{proj-dist}}(c_1,L))$
such that the following holds. \\
For all $x \in C_L(X_1,X_2)\cup C_L(X_3,X_4)$ the horizontal distance between $x$ and $\Pi(x)$  is at most $D_{\ref{proj-dist}}$.
\end{lemma}

\begin{proof}
This follows from the fact that in any $\delta'-$ hyperbolic metric space (=$F_b$ in our case) the Hausdorff distance between
 a triangle with vertices $x, y, z$ and the tripod $[x,w]\cup[y,z]$ (where $w \in [y,z]$ is a nearest point projection of 
$x$ onto $[y,z]$) is bounded by  $\delta'$.
\end{proof}

\noindent {\bf Proof of Proposition \ref{final-prop}:}

Set $L_0=2c^{'}_6 +D_{\ref{pi-lip}}(c_1)$;  by assumption $L\geq L_0$. Let $L_1 = L + \delta'$.

First for every pair of points in $x,y \in C(X_1,X_2,X_3)$ we choose a geodesic
in (the  path metric induced on) $C_{L_1}(X_1,X_2)\cup C_{L_1}(X_3,X_4)$ joining $x, y$ and project it into $C(X_1,X_2,X_3)$  by $\Pi$.
 This defines a path in $C(X_1,X_2,X_3)$ say $c(x,y)$ joining $x,y$.  Note that by Lemma
$\ref{proj-dist}$ the paths $c(x,y)$ are {\em uniform quasi-geodesics in} $C_{L_1}(X_1,X_2)\cup C_{L_1}(X_3,X_4)$.
Now we need to check the conditions of Corollary $\ref{hyp-lemma}$. 

Here the whole space is $C_L(X_1,X_2,X_3)$ and the discrete set is the set of vertices
contained in $C(X_1,X_2,X_3)$. As per the notation of Corollary $\ref{hyp-lemma}$, set $D=L$. Next we note the following:
\begin{enumerate}
\item Condition $(1)$ of Corollary $\ref{hyp-lemma}$ follows from Lemma $\ref{pi-lip}$.
\item Condition $(2)$ of Corollary $\ref{hyp-lemma}$ follows from the observation that $C_L(X_1,X_2,X_3)$ is contained in 
$C_{L_1}(X_1,X_2)\cup C_{L_1}(X_3,X_4)$.
\item Conditions $(3)$, $(4)$ of Corollary $\ref{hyp-lemma}$ follow from Lemma $\ref{proj-dist}$, 
since the space $C_{L_1}(X_1,X_2)\cup C_{L_1}(X_3,X_4)$ is (uniformly) hyperbolic.
\end{enumerate}

Hence $C_L(X_1,X_2,X_3)$ is hyperbolic. From Lemmas $\ref{T-space}$, and $\ref{proj-dist}$ it follows
that $C(X_1,X_2)$ is the image of the quasi-convex set $C(X_1,X_2) \subset C_{L_1}(X_1,X_2)\cup C_{L_1}(X_3,X_4)$
under the quasi-isometric embedding $\Pi$
(cf. Remark \ref{pi-lip-rmk}). Hence it is quasi-convex in $C_L(X_1,X_2,X_3)$ and thus so is
$C_L(X_1,X_2)$.
This completes the proof. $\Box $

\smallskip

Conclusion (2) of Proposition \ref{final-prop0} is given by the  next Corollary, which
 is an immediate consequence of the fact that $C_L(X_1,X_2,X_3)$ is hyperbolic (cf. Proposition \ref{final-prop}) and
that the inclusion $C_L(X_1,X_2) \hookrightarrow C_L(X_1,X_2,X_3)$ is a qi embedding (cf. Remark \ref{pi-lip-rmk}).

\begin{cor}\label{ladder-qi-embed} Given $c_1 \geq 1$ and $L \geq L_0$ (where $L_0$ is as in Proposition \ref{final-prop}) there exists 
 $D_{\ref{ladder-qi-embed}}(= D_{\ref{ladder-qi-embed}}(c_1,L))$ such that if
 $x,y \in C_L(X_1,X_2)$,   $\gamma_1$ is  a geodesic in $C_L(X_1,X_2,X_3)$ joining $x,y$ and
$\gamma_2$  is a geodesic
in $C_L(X_1,X_2)$  joining $x,y$, then the Hausdorff distance $Hd(\gamma_1,\gamma_2)\leq D_{\ref{ladder-qi-embed}}$. 
\end{cor}

Conclusion (3) of Proposition \ref{final-prop0} is given by the following.

\begin{cor}\label{curv-defn} Given $c_1 \geq 1$ and $L \geq L_0$ (cf. Proposition \ref{final-prop}) there exists $D_{\ref{curv-defn}}
(=D_{\ref{curv-defn}}(c_1,L))$  such that the following holds. \\
Suppose $X_i,X^{'}_i$, $i=1,2$ are $c_1$-qi sections and $x_i\in X_i \cap X^{'}_i$, $i=1,2$.
Then the Hausdorff distance between the geodesics joining $x_1,x_2$ in the subspaces
$C_L(X_1,X_2)$ and $C_L(X^{'}_1,X^{'}_2)$ is at most $D_{\ref{curv-defn}}(c_1,L)$. 
\end{cor}

\begin{proof}
This follows from Proposition $\ref{final-prop}$ and Corollary $\ref{ladder-qi-embed}$
applied successively to the tripod bundles $C_L(X_1,X^{'}_1,X_2)$ and $C_L(X^{'}_1,X_2,X^{'}_2)$.
\end{proof}


\noindent {\bf Concluding the proof of  Proposition \ref{final-prop0}:}
 Proposition \ref{final-prop}, 
 Corollary \ref{ladder-qi-embed} and 
Corollary \ref{curv-defn} together give precisely the statement of Proposition \ref{final-prop0}. $\Box$

\section{Consequences and Applications} A number of consequences of Theorem \ref{combthm} are collected together in this section.

\subsection{Sections, Retracts and Cannon-Thurston maps}

We shall say that an exact sequence of finitely generated groups 
$1\rightarrow K\rightarrow G\rightarrow Q\rightarrow 1$
satisfies {\it bounded flaring} if the associated metric graph bundle (cf. Example \ref{eg-mbdl}) of Cayley graphs does. An immediate consequence of Theorem \ref{combthm}
coupled with the existence of  qi-sections from Theorem \ref{qi-mosher} is the following {\it converse} to (the second part of) Mosher's Theorem \ref{qi-mosher}.

\begin{theorem} \label{combthmgps} 
Suppose that the short exact sequence of finitely generated groups 

\begin{center}
$1\rightarrow K\rightarrow G\rightarrow Q\rightarrow 1$.
\end{center}
satisfies a flaring condition such that $K, Q$ are word hyperbolic and $K$ is non-elementary. 
Then $G$ is hyperbolic. \end{theorem}

 Theorem \ref{qi-mosher} was generalized by Pal \cite{pal-piams} as follows.

\begin{theorem} \label{qi-pal} {\bf (Pal \cite{pal-piams})} 
Suppose we have a short exact sequence of pairs of finitely generated groups 
\[
1\rightarrow (K,K_1)\rightarrow (G,N_G(K_1))\stackrel{p}{\rightarrow}(Q,Q_1)\rightarrow 1
\]
with $K$ strongly  hyperbolic relative to a subgroup $K_1$ such that $G$ preserves cusps, i.e.
 for all $g\in G$ there exists $h\in K$ with
 $gK_1g^{-1}=hK_1h^{-1}$. Then there exists a $(k,\epsilon)-$quasi-isometric section $s\colon Q \to G$ for some constants $k\geq1$,
$\epsilon\geq 0$. 
Further, $Q_1=Q$ and there is a quasi-isometric section $s\colon Q\to N_G(K_1)$ satisfying 
\[
\frac{1}{R}d_Q(q,q')-\epsilon \leq d_{N_G(K_1)}(s(q),s(q'))\leq Rd_Q(q,q')+\epsilon
\]
where $q,q'\in Q$ and $R\geq1$,
$\epsilon \geq 0$ are constants. In addition, if $G$ is weakly hyperbolic relative to $K_1$, then $Q$ is hyperbolic.
\end{theorem}

The setup of Theorem \ref{qi-pal} naturally gives a metric graph bundle $P: X \rightarrow Q$ of spaces, where $Q$ is the quotient group and fibers
are isometric to the coned off spaces $\hhat K$ obtained by electrocuting copies of $K_1$ in $K$.

We shall now use  Theorem \ref{retract}.
Theorem \ref{retract} is proven in \cite{mitra-ct} in the context of an exact sequence of finitely generated
groups $1\rightarrow N \rightarrow G \rightarrow Q\rightarrow 1$, with $N$ hyperbolic; but all that the proof requires is the existence of qi sections
(which follows in the context of groups by the qi section  Theorem \ref{qi-mosher} of Mosher).

As in \cite{mitra-ct}, the existence of a qi-section through each point of $X$ guarantees, via Theorem \ref{retract}, the existence of a continuous extension to the
boundary (also called a  Cannon-Thurston map \cite{CTpub} \cite{CT}) of the map $i_b: F_b \rightarrow X$ provided
$X$ is hyperbolic. The proof is identical to that in \cite{mitra-ct} and we omit it here,
 referring the reader to \cite{mitra-ct} for details. Combining this fact with Theorem \ref{combthm} we have the following.
\begin{theorem}  \label{ct}
Suppose $p:X\rightarrow B$ is a metric (graph) bundle with the following properties:\\
$(1)$ $B$ is a $\delta$-hyperbolic metric space.\\
$(2)$ Each of the fibers $F_b$, $b\in B$ ($b \in \mathcal{V} (B)$) is a $\delta^{'}$-hyperbolic 
metric space with respect to the induced path metric from $X$.\\
$(3)$ The barycenter maps $\partial^3F_b \rightarrow F_b$, $b\in B$ ($b \in \mathcal{V} (B)$) are uniformly coarsely  surjective.\\
$(4)$ The metric  (graph) bundle satisfies a flaring condition.\\
Then the inclusion $i_b: F_b \rightarrow X$ extends continuously to a map
$\hat{i} : \widehat{F_b} \rightarrow \widehat{X}$ between the Gromov compactifications.
\end{theorem}

\subsection{Hyperbolicity of base and flaring}
In our main combination theorem $\ref{combthm}$ flaring was a sufficient 
condition. In this subsection and the next
we investigate its necessity. This issue is closely linked with 
hyperbolicity of the base space $B$.
We study it with special attention to hyperbolic and relatively 
hyperbolic groups as in Theorems \ref{qi-mosher} and Theorem \ref{qi-pal}.

A Theorem of Papasoglu (cf. \cite{papasoglu-thesis}, Lemma 3.8 of \cite{papasoglu-qiinv}) states the following.


\begin{theorem} \cite{papasoglu-thesis, papasoglu-qiinv} Let $G$ be a finitely generated group and let $\Gamma$
be the
 Cayley graph of $G$ with respect to  a finite generating set. If there is an $\epsilon$
such that geodesic bigons in $\Gamma$ are $\epsilon$-thin then $G$ is hyperbolic. \\
Similarly, let $X$ be a geodesic metric space such that
for every $K$ there exists $C$ such that $K$-quasigeodesic bigons are $C$-thin, then $X$ is hyperbolic. \\
In fact there is some (universal) constant $C>0$ such that
if $G$ is finitely generated and non-hyperbolic, then $\forall  R > 0$ there is some $R^\prime > R$ and a
$(C,C)$-quasi-isometric embedding of
a Euclidean circle of radius $R^\prime$ in $\Gamma$.\label{bigon} \end{theorem}

We now look at short exact sequences of finitely generated groups.

\begin{prop} Consider a short exact sequence of finitely generated groups 
\begin{center}
$1\rightarrow K\rightarrow G\rightarrow Q\rightarrow 1$.
\end{center}
such that $K$ is non-elementary word hyperbolic but $Q$ is not hyperbolic. Then the short exact sequence cannot satisfy
a flaring condition. \label{converseflare}\end{prop}

\begin{proof}
By Theorem \ref{bigon}, $Q$ contains $(C,C)$ qi embeddings
of Euclidean circles  of  arbitrarily large radius. Now, given any $l, A_0$
construct \\
a) a $(C,C)$ qi embedding $\tau_l$ of a  Euclidean circle $\sigma$ of circumference
$> 4l$ in $Q$\\
b) two qi sections $s_1, s_2$ of $Q$ into $G$ by Theorem \ref{qi-mosher}
such that $d_h(s_1\circ \tau_l (\sigma ), s_2\circ \tau_l(\sigma )) > A_0$.

Let $q\in \sigma$ be such that the horizontal distance $d_q(s_1\circ \tau_l (q ), s_2\circ \tau_l
(q))$ in the fiber $F_q$ over $q$ is maximal. Let the two arcs of length $l$ in $\tau_l$ starting
at $q$ (in opposite directions) end at $q_1, q_2$. Let $\overline{q_1qq_2}$ denote the union
of these arcs. Then the two quasigeodesics
$s_1\circ \tau_l (\overline{q_1qq_2} ), s_2\circ \tau_l (\overline{q_1qq_2} )$ violate flaring
as the  horizontal distance  achieves a maximum at the midpoint $q$. 
\end{proof}

We next turn to the relatively hyperbolic situation described in Theorem \ref{qi-pal} with $Q$ non-hyperbolic, i.e. we 
assume that $K$ is (strongly) hyperbolic relative to $K_1$. We have an analog of Proposition \ref{converseflare}
in this situation too. The proof is the same as that of the above proposition.
The existence of qi sections in this case, follows from Theorem $\ref{qi-pal}$.

\begin{lemma}
Suppose we have a short exact sequence of finitely generated groups 
\[
1\rightarrow (K,K_1)\rightarrow (G,N_G(K_1))\stackrel{p}{\rightarrow}(Q,Q_1)\rightarrow 1
\]
such that $K$ strongly hyperbolic relative to the cusp subgroup $K_1$ and $G$ preserves cusps, 
but $Q$ is not hyperbolic. Let $P: X \rightarrow Q$
be the associated metric graph bundle of spaces, where $Q$ is the quotient group and fibers $F_q$
are the coned off spaces $\hhat K$ obtained by electrocuting copies of $K_1$ in $K$. 
Then $X$ does not satisfy flaring.
\label{converseflarerh}\end{lemma} 

The rest of this subsection is devoted to proving the following.

\begin{prop}
Suppose we have a short exact sequence of finitely generated groups 
\[
1\rightarrow (K,K_1)\rightarrow (G,N_G(K_1))\stackrel{p}{\rightarrow}(Q,Q_1)\rightarrow 1
\]
with $K$  (strongly) hyperbolic  relative to the cusp subgroup $K_1$ such that $G$ preserves cusps. Suppose further that 
$G$ is  (strongly) hyperbolic  relative to $N_G(K_1)$. Then $Q$ is  hyperbolic. \label{mosher-genlzn} \end{prop}

\begin{proof}
We shall argue by contradiction.
Suppose $Q$ is not hyperbolic. 

Let $X$ be a Cayley graph of $G$ with respect to a finite generating set $S$ containing a finite generating set of $K$
(and, for good measure, a finite generating set of $K_1$).  Let $B$ be
the Cayley graph of $Q$ with respect to  $p(S)\setminus \{ 1 \}$.
Then the quotient map $G\rightarrow Q$ gives rise to a metric graph bundle $p:X\rightarrow B$ as before.
This metric graph bundle admits uniform qi sections through each point of $X$ by Theorem $\ref{qi-pal}$.
Also  $B$ is not a hyperbolic metric space. By Theorem \ref{bigon}, there exists $C > 0$ such that for all $r > 0$
we can construct a $(C,C)$-qi embedding $\tau_r$ of  a Euclidean circle $\sigma_r$ of radius bigger than $r$
 in $B$.

\noindent {\bf Claim:} Given
 $k > 0$ there exists $D=D(k)$  such that for any $k$-qi section $s:\mathcal{V} (B)\rightarrow X$ of the metric 
graph bundle $p:X\rightarrow B$, 
$s\circ \tau_r (\sigma_r)$ is contained in a $D$-neighborhood of a coset of $N_G(K_1)$.

\noindent {\em Proof of  claim:}  
Let $\tau=s\circ \tau_r$. Then $\tau$ is
a $k_1:=(kC+k)$-quasi-isometric embedding of $\sigma_r$ in $s(B)$.
Let $u,v$ be a pair of antipodal points of the circle and $a=\tau(u),b=\tau(v)$. 
Let $\sigma_r^1, \sigma_r^2$ be the two arcs of $\sigma_r$ joining $u, v$. Then $\tau(\sigma_r^1), \tau(\sigma_r^2)$
are $k_1-$quasigeodesics joining $a, b$.

Let $d_r$ denote the intrinsic path metric on $\sigma_r$. Since $\tau$ is a qi embedding, it follows that
 for all $C_1 \geq 0$,  there  exists $C_2 \geq 0$ such that for all $r> 0$ and $x\in \sigma_r^1, y \in \sigma_r^2$,
$d_r (x, \{ a, b\}) \geq C_2$ and  $d_r (y, \{ a, b\}) \geq C_2$ implies that $d_X(x,y) \geq C_1$.  Hence
$\tau(\sigma_r^1)\cup \tau(\sigma_r^2)$ is a `thick' quasigeodesic bigon, i.e. except for initial and final subsegments
of length $k_1C_2$, $\tau(\sigma_r^1)$ and $ \tau(\sigma_r^2)$ are separated from each other by at least $\frac{C_1}{k_1}$.

Since $G$ is strongly hyperbolic relative to $N_G(K_1)$, thick quasigeodesic bigons lie in a bounded 
neighborhood of a coset of $N_G(K_1)$ (see Definition \ref{brp} or \cite{farb-relhyp}). The claim follows. $\Box$

We continue with the proof of the proposition. For any $k-$qi section $s:\mathcal{V} (B)\rightarrow X$,
we shall refer to $s\circ \tau_r (\sigma_r) = \tau (\sigma_r)$ as a qi section of the   circle $\sigma_r$.
Let $Y_1,Y_2$ be two $k$-qi sections of a large Euclidean circle $\sigma_r$ in $B$, such that
$Y_1$ and $Y_2$ lie $D-$ close to two {\it distinct} cosets of $N_G(K_1)$ (with $D=D(k)$ as in the Claim above).
Let $W(Y_1,Y_2)$ be the union $\cup_{q\in \tau_r (\sigma_r)}\lambda_q$, where 
$\lambda_q$ is a horizontal geodesic in  $F_q$ joining $Y_1\cap F_q$ to $Y_2\cap F_q$.
Suppose $b,b^{'}$ are images (under $\tau_r$) of antipodal points on $\sigma_r$.
As in the proof of Lemma $\ref{qi-section}$, we know that there exists $k_1(=k_1(k))$
such that for each point $z$ of $\lambda_b$ there exists a
$k_1-$quasi-isometric section of $\sigma_r$ in $W(Y_1,Y_2)$; any such qi section is $D_1(=D_1(k_1))-$close to a coset of $N_G(K_1)$
by the Claim above.

Since $Y_1$, $Y_2$ are close to distinct
cosets of $N_G(K_1)$, we can find (as in Step 1 of Proposition \ref{main-prop}) 
two $k_1-$qi sections $Y^{'}_1$, $Y^{'}_2$ of $\sigma_r$ passing through
$z_1, z_2 \in \lambda_b$  with $d(z_1, z_2) = 1$ such that  $Y^{'}_1$, $Y^{'}_2$ are \\
a) $D_1-$close to two distinct cosets of $N_G(K_1)$,\\
b) both contained in $W(Y_1,Y_2)$.

Suppose $Y^{'}_1,Y^{'}_2$ intersect $\lambda_{b^{'}}$ in $z^{'}_1$ and $z^{'}_2$  respectively.
If $d(z^{'}_1,z^{'}_2)$ is large, in the same way as before, we can construct two $k_2(=k_2(k_1))-$qi sections $Y_3,Y_4$ of $\sigma_r$ contained in 
$W(Y^{'}_1,Y^{'}_2)$ such that \\
a) $Y_3,Y_4$  are $D_2(=D_2(D_1))-$ close to two distinct cosets of $N_G(K_1)$\\
b) $d(Y_3\cap \lambda_{b^{'}},Y_4\cap \lambda_{b^{'}})=1$. 

Thus we have two $k_2-$qi sections of long subarcs 
of $\sigma_r$ that start and end close by in $X$ but lie close to distinct cosets of $N_G(K_1)$.
Since  $r$ can be chosen to be arbitrarily large, this violates strong relative hyperbolicity
of $G$ with respect to the cosets of $N_G(K_1)$, proving the proposition.  
\end{proof}

\subsection{Necessity of Flaring}

In this subsection we prove that flaring is a necessary condition for hyperbolicity of a metric (graph) bundle:

\begin{prop} Let $P: X \rightarrow B$ be a metric (graph) bundle such that \\
$1.$ $X$ is $\delta$-hyperbolic \\
$2.$ There exist $\delta_0$ such that each of the fibers $F_z$, $z\in B$  $ (\mathcal{V}(B))$ is $\delta_0$-hyperbolic
equipped with the path metric  induced from $X$. \\
Then the metric bundle satisfies a flaring condition. 
In particular,  any exact sequence of finitely generated groups 
$1 \rightarrow N \rightarrow G \rightarrow Q \rightarrow 1$
with $N, G$ hyperbolic satisfies a flaring condition.
\label{necflaring} \end{prop}

The proof will occupy the entire subsection.
Suppose $\gamma: [-L,L]\rightarrow B$ is a geodesic and $\alpha,\beta$ are two $K_1$-qi lifts of $\gamma$.
As in the construction of ladders, we define $Y$ to be the union of horizontal geodesics
$[\alpha(t),\beta(t)]\subset F_{\gamma(t)}$, $t\in [-L,L]$, and refer to it as the {\em ladder} formed by $\alpha$ and $\beta$. Let $\eta:[0,M]\rightarrow F_{\gamma(0)}$ be the geodesic $Y\cap F_{\gamma(0)}$.

A crucial ingredient is the following lemma which is a specialization
to our context of the fact that geodesics in a hyperbolic space diverge exponentially.
(See Proposition 2.4 and the proof of Theorem 4.11 in \cite{mitra-endlam}).

\begin{lemma}\label{diverge}
 Given $K_1 \geq 1, D \geq 0$ there exist $b=b(K_1,D)> 1$, 
$A=A(K_1,D)> 0$ and $C=C(K_1,D)> 0$ such that the following holds:

If $d(\alpha(0),\beta(0)) \leq D$ and there exists $T \in [0,L]$ with
$d(\alpha(T),\beta(T)) \geq C$ then any path joining $\alpha(T+t)$
to $\beta(T+t)$ and lying outside the union of the
$\frac{T+t-1}{2K_1}$-balls around $\alpha(0), \beta(0)$ has length
greater than $Ab^{t}$ for all $t \geq 0$ such that $T+t \in [0,L]$.
In particular, the horizontal distance between $\alpha(T+t)$ and $\beta(T+t)$ is greater than $Ab^{t}$
for all $t \geq 0$ such that $T+t\in [0,L]$.

\end{lemma}

Now, we use Lemma \ref{diverge} to show that the ladder $Y$ flares in at least one direction of $\gamma$.
We start the proof by showing this in two special cases.
A general ladder is then broken into subladders of the special types 
by qi lifts of $\gamma$ as in Step 1 of the proof of Proposition \ref{main-prop}. 
(Recall that we get exactly two types of subladders in this way. 
This motivates us to consider the two types of special ladders here.)
We point out that

(1) the first type of ladder is of uniformly small (but not too small) girth;

(2) the second type of ladder is not necessarily of small girth but any qi lift
of $\gamma$ divides it into two subladders of small girth.\\
The proof of flaring for all ladders follows from this.

We shall need the following lemma also.
\begin{lemma} \label{diverge2}
1) Given $d_1,d_2, \delta \geq 0$ and $k\geq 1$ there are constants $C=C(d_1,d_2,k,\delta)$ and $D=D(k,\delta)$
such that the following holds:

Let  $X$ be a $\delta$-hyperbolic metric space and let $\alpha_1,\alpha_2:[-L,L]\rightarrow X$ be
$k$-quasi-geodesics.
Let $[a,b]\subset [-L,L]$ and suppose $d_1= d(\alpha_1(a),\alpha_2(a))$ and $d_2=d(\alpha_1(b),\alpha_2(b))$.
If $[t-C,t+C]\subset [a,b]$ for some $t\in [a,b]$ then $d(\alpha_1(t),\alpha_2(t))\leq D$. \\
2) Through each point of a ladder $Y$
formed by $K_1$-qi lifts of a geodesic $\gamma$ in $B$ there is a $C_{\ref{qi-section}}(K_1)$-qi lift
of $\gamma$ contained in $Y$. 
\end{lemma}

\begin{proof} (1) follows easily from stability of quasi-geodesics and slimness of triangles in $X$.
(See Lemma $1.15$ of Chapter $III.H$, \cite{bridson-haefliger} for instance).\\
(2) This is a replica of the proof of Lemma \ref{qi-section}.
\end{proof}

\begin{rem} \label{refrmk}
 We shall assume $L$ to be sufficiently large for the following arguments to go through. We give the proof for metric bundles.
The same proof works mutatis mutandis (replacing $B$ by $(\mathcal{V}(B))$ for instance) for  metric graph bundles.
\end{rem}

\noindent {\bf Flaring of ladders in special cases:}

Let $D=D_{\ref{diverge2}}(K_1,\delta)$ and $D^{'}_1=C_{\ref{diverge}}(K_1,D)$.
Since the horizontal spaces in $X$ are uniformly properly embedded in $X$ there is a $D_1$
such that for all $v\in B$    and $x,y\in F_v$ if $d_v(x,y)\geq D_1$ then $d(x,y)\geq D^{'}_1$.
Let $K_{i+1}=C^{i}_{\ref{qi-section}}(K_1)$, $i=1,2,3$. Also suppose that $d_{\gamma(0)}(\alpha(0),\beta(0))=M$.

\begin{lemma}\label{type1} {\bf Ladders of type $1$:} For $K_1,D, D_1$ as above and $M\geq D_1$,
there exists $n_1=n_1(K_1,M)$ 
such that 
$max\{d_{\gamma(-t)}(\alpha(-t),\beta(-t)),d_{\gamma(t)}(\alpha(t),\beta(t))\}\geq 8M$ for all
$t\geq n_1$. 
\end{lemma}

\begin{proof}
Let $D_2=C_{\ref{diverge}}(K_1,M)$ and let $C_1:=1+2.C_{\ref{diverge2}}(M,D_2,K_1,\delta)$.
If $d(\alpha(C_1),\beta(C_1))\geq D_2$ then 
for all $t\geq 0$ the length of the horizontal geodesic joining $\alpha(C_1+t)$ to $\beta(C_1+t)$
is greater than or equal to $A_1.b^{t}_1$ for some $A_1=A_{\ref{diverge}}(K_1,M),\, b_1=b_{\ref{diverge}}(K_1,M)$.
Choose $t_1>0$ such that for all $t\geq t_1$, $A_1.b^{t_1}_1\geq 8M$.

Else, suppose $d(\alpha(C_1),\beta(C_1))<D_2$.
In this case, by Lemma \ref{diverge2}, $d(\alpha(\frac{C_1-1}{2}),\beta(\frac{C_1-1}{2}))\leq D$.
By the choice of the constants $D,D_2$ we can again apply Lemma $\ref{diverge}$ so that for all
$t\geq 0$ the length of a horizontal geodesic $Y\cap F_{\gamma(-t)}$ is greater than or equal to
$A_2.b^{t}_2$, where the constants $A_2,b_2$ depend on $K_1$ and $D$. Choose $t_2>0$ such that
for all $t\geq t_2$, $A_2.b^t_2\geq 8M$. Now let $n_1= max\{C_1+t_1,t_2\}$. Thus we have
$max\{d_{\gamma(-t)}(\alpha(-t),\beta(-t)),d_{\gamma(t)}(\alpha(t),\beta(t))\}\geq 8M$ for all
$t\geq n_1=n_1(M,K_1)$.   
\end{proof}

\begin{lemma}{\bf Ladders of type $2$:}\label{type2}
Suppose $l>0$ and that for any $s\in [0,M-1]$ there is a $K_2$- qi
lift $\alpha_1$ of $\gamma$ in $Y$ through $\eta(s)$ such that  $d(\alpha(t),\alpha_1(t))\leq l$
for some $t\in [-L,L]$. There are $n_2=n_2(K_1,l)$  and $D_4=D_4(K_1,l)$ such that 
for all $t\geq n_2$ we have
\[max\{ d_{\gamma(t)}(\alpha(t),\beta(t)),d_{\gamma(-t)}(\alpha(-t),\beta(-t))\}\geq 8M \,\,\mbox{if} \,\,M\geq D_4+1. \]
\end{lemma}

\begin{proof}
Let $C_3=C_{\ref{diverge}}(K_3,l)$, $A_3=A_{\ref{diverge}}(K_3,l), b_3=b_{\ref{diverge}}(K_3,l)$.
Let $m_0 := min \{m\in \mathbb N : A_3.b^m_3\geq D_{\ref{easy-lemma}}(g(2K_3))\}$, where $g$ refers
to the function defined in  Lemma $\ref{condition3}$. 
It follows easily from the bounded flaring condition that there is a constant $D^{'}_4$ such that
the following is true:
 
Suppose we have two $K_3$-qi lifts $\alpha^{'},\alpha^{''}:[-L,L]\rightarrow X$ of the geodesic
$\gamma :[-L,L]\rightarrow B$ with $d_{\gamma(0)}(\alpha^{'}(0),\alpha^{''}(0))\geq D^{'}_4$ then 
$d(\alpha^{'}(t),\alpha^{''}(t))\geq D_{\ref{easy-lemma}}(g(2K_3))$ for all
$t\in [0,m_0]$.

Let $D_4=max\{D^{'}_4, C_3\}$ and let $M-1=N.D_4+r$ where $N\in \mathbb N$ and $0\leq r<D_4$.
Now construct a $K_2$-qi section $\beta_1$ in the ladder $Y$
such that $d_{\gamma(0)}(\beta(0),\beta_1(0))=r+1$ and $d(\alpha(t_0),\beta_1(t_0))\leq l$ for some $t_0\in [-L,L]$.
Without loss of generality we assume that $t_0\in [-L,0]$.
We now use Lemma \ref{diverge2} (2) to break the subladder of $Y$, formed by $\alpha$ and $\beta_1$,
by $K_3$-qi lifts $\alpha_0=\alpha,\alpha_1,\cdots,\alpha_N=\beta_1$ of $\gamma$ such that
$d_{\gamma(0)}(\alpha_i(0),\alpha_{i+1}(0))=D_4$. We have $d(\alpha_i(t_0),\alpha_{i+1}(t_0))\leq l$.
Thus by the choice of the constant $D_4$,
$d_{\gamma(t)}(\alpha_i(t),\alpha_{i+1}(t))\geq max\{D_{\ref{easy-lemma}}(g(2K_3)), A_3.b^t_3\}$ for all $t\geq 0$. Also
 (as in Step 2 of the proof of Proposition \ref{main-prop}) $\cup [\alpha_i(t),\alpha_{i+1}(t)]$
is a partition of the horizontal geodesic segment
$[\alpha(t),\beta_1(t)]\subset Y\cap F_{\gamma(t)}$,   for all $t\in [0,m_0]$.
Therefore, we can choose $n_2=n_2(K_1,l)$ such that for all $t\geq n_2$,
\[max\{ d_{\gamma(t)}(\alpha(t),\beta(t)),d_{\gamma(-t)}(\alpha(-t),\beta(-t))\}\geq 8M \,\,\mbox{if} \,\,M-1\geq D_4=D_4(K_1,l). \] 
\end{proof}

{\bf Flaring of general ladders:} In the general case, first of all, we break the ladder $Y$
into subladders of special types as described above (see figure below where horizontal and vertical directions
have been interchanged for aesthetic reasons). 

Let us assume that $Y$ is bounded by $K-$ qi lifts $\alpha$, $\beta$ of a geodesic
$\gamma:[-L,L]\rightarrow X$. 
Let $\eta:[0,M]\rightarrow F_{\gamma(0)}$ be the geodesic $Y\cap F_{\gamma(0)}$.
Let $K_i=C^i_{\ref{qi-section}}(K)$, and $l=D_{\ref{easy-lemma}}(g(2.K_2))$.
Let $M_K:=max\{D_1(K_1),D_4(K_1,l)\}$, and $n_K:=max\{n_1(K_1,M_K),n_2(K_1,l)\}$
where the functions $D_1,D_4,n_1,n_2$ are as in the proof of the flaring for the special ladders.

{\bf Claim:} If $M\geq M_K$ then we have
\[
max\{d_{\gamma(-n_K)}(\alpha(-n_K),\beta(-n_K)),d_{\gamma(n_K)}(\alpha(n_K),\beta(n_K))\}\geq 2.d_{\gamma(0)}(\alpha(0),\beta(0)).
\]

To show this we inductively construct $K_1$-qi sections $\alpha_0=\alpha,\alpha_1,\cdots,\alpha_i=\beta$
in $Y$ to decompose it into subladders of the two types we mentioned above. This is done as in Step 1 of the proof of
Proposition \ref{main-prop}. Nevertheless we include a sketch of the argument for completeness.

Since $M\geq M_K$, therefore
by Lemma \ref{refrmk} (2), we can construct a $K_1$-qi section $\alpha_1$ through $\eta(M_K)$. 
Now, suppose $\alpha_1,\cdots,\alpha_j$ has been
constructed through the points $\eta(s_1),\cdots,\eta(s_j)$ respectively. 
If $d_{\gamma(0)}(\alpha_j(0),\beta(0))\leq M_K$ define $\alpha_{j+1}=\beta$. Otherwise, if there is a $K_1$-qi
section through $\eta(M_K+s_j)$ in the ladder formed by $\alpha_j$ and $\beta$, define it to be $\alpha_{j+1}$.
If neither  happens then consider the following set: {\small
\[ {\mathcal{T}}_j=
\{t\geq s_j+M_K:\,\, \exists \,\mbox{a}\,K_1 \mbox{-qi section through} \,\,\eta(t) \,\mbox{ entering the ladder formed by } \,\alpha_j \,\mbox{and} \, \alpha\} 
\] }
Let $t_j = \mbox{sup} {\mathcal{T}}_j$ be the supremum of this set. Define $\alpha_{j+1}$ to be a $K_1$-qi section 
through $s_{j+1}:=t_j+1$, in the ladder formed
by $\alpha_j$ and $\beta$ that does not enter the ladder formed by $\alpha_j$ and $\alpha$.

\begin{center}

\includegraphics[height=40mm]{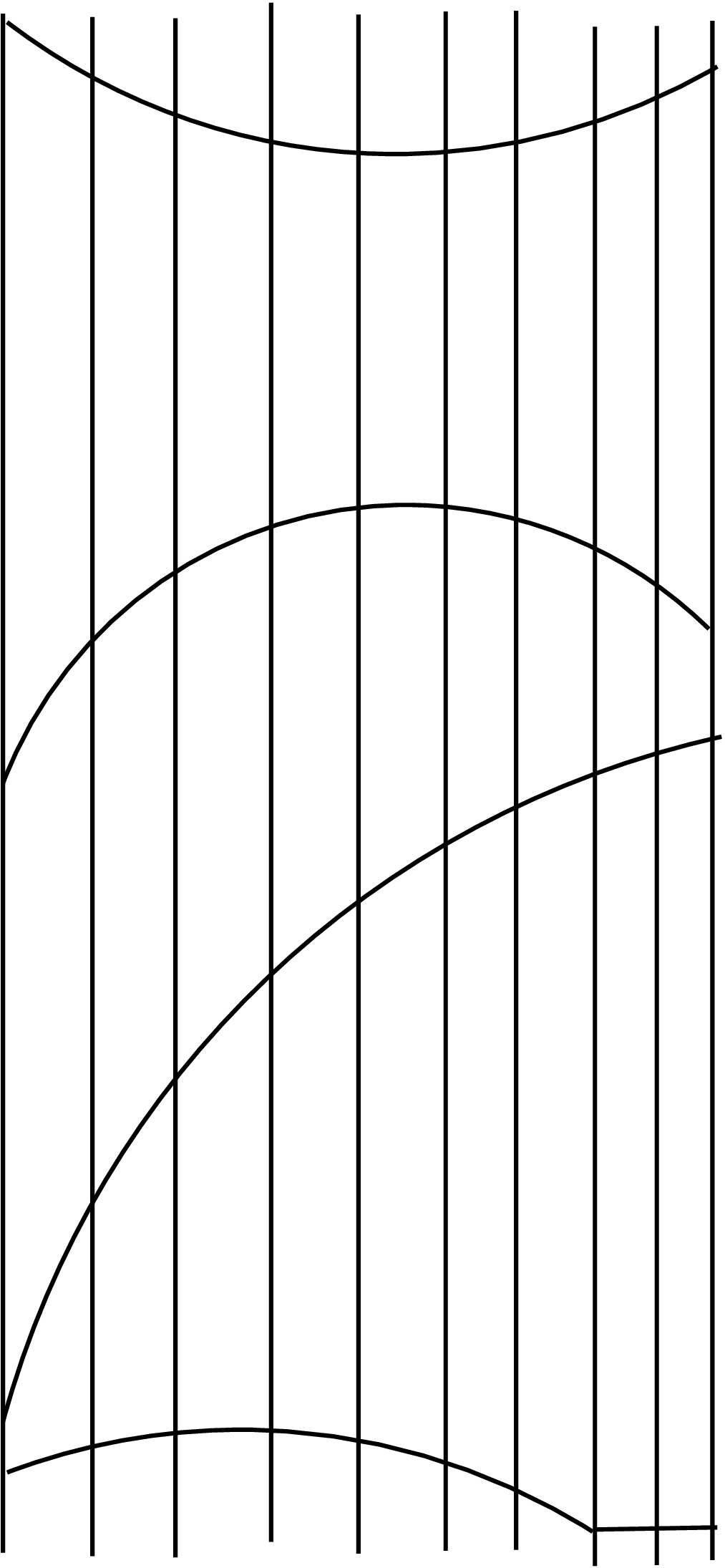}\\
 {\it Flaring subladders } 

\end{center}

For each $j$, $\alpha_j$ and $\alpha_{j+1}$ form a special ladder (except possibly for the last one) and hence it must flare.
Thus $\eta$ can be expressed as the disjoint union of subsegments that flare to the left and the union of the
subsegments that flare to the right respectively. The total length of one of these
types  must be at least one-fourth of the length of $\eta$.
The claim follows. $\Box$ 

The first statement of Proposition \ref{necflaring} follows immediately.
The last statement follows from Example \ref{eg-mbdl} and the first part of this Proposition. $\Box$

\subsection{An Example}
 Let $(Teich(S), d_T)$ be the Teichmuller space  of a closed surface $S$ equipped with the Teichmuller metric $d_T$. 
Teichmuller space can also be equipped with an electric metric $d_e$ by electrocuting the thin parts
(see \cite{farb-relhyp} for details on electric geometry and the introduction to this paper for a quick summary
and relevant notation).
Note (as per work of Masur-Minsky \cite{masur-minsky}, see also \cite{mahan-jrms}) that $(Teich(S), d_e)$ is quasi-isometric to the curve complex  $CC(S)$.
Let $E: (Teich(S), d_T) \rightarrow (Teich(S), d_e)$ be the identity map from  the
Teichmuller space of $S$ equipped with the Teichmuller metric $d_T$ to  the
Teichmuller space of $S$ equipped with the electric metric $d_e$. 

We shall need the following Theorem due to Hamenstadt \cite{hamenst-gd} which used an idea of Mosher \cite{mosher-gt03} in its proof.

\begin{theorem} { \bf Hamenstadt \cite{hamenst-gd}:} For every $L > 1$ there exists $D>0$ such that the following holds. \\
Let $f: {\mathbb{R}} \rightarrow (Teich(S), d_T)$ be a Teichmuller  $L$-quasigeodesic such that $E \circ f:     {\mathbb{R}} \rightarrow (Teich(S), d_E)$ 
is also an  $L$-quasigeodesic.
Then for all $a, b \in \mathbb{R}$ there is a
Teichmuller geodesic $\eta_{ab}$ joining $f(a), f(b) \in Teich(S)$ such that the Hausdorff distance $d_H (f([a,b]), \eta_{ab} ) \leq D$.
\label{stability} \end{theorem}

We are now in a position to prove a rather general combination  proposition for metric bundles over quasiconvex
subsets of $CC(S)$.
For $j: K \rightarrow (Teich(S), d_T)$ a map, let
$U(S,K)$ denote the pullback (under $j$) of the universal curve over $Teich(S)$
equipped with the natural path metric. Also, the universal cover of the  universal curve over $Teich(S)$
is a hyperbolic plane bundle over $Teich(S)$. Let  $\widetilde{U(S,K)}$ denote the pullback to $K$ of this hyperbolic plane bundle.

\begin{prop}
Let $(K,d_K)$ be a hyperbolic metric space satisfying the following:\\
There exists $C > 0$ such that
for any two points $u, v\in K$, there exists a bi-infinite $C$-quasigeodesic $\gamma \subset K$ with $d_K(u, \gamma ) \leq C$
and $d_K(v, \gamma ) \leq C$. \\
Let $j: K \rightarrow (Teich(S), d_T)$ be a quasi-isometric embedding
such that $E\circ j: K \rightarrow (Teich(S), d_e)$ is also a quasi-isometric embedding.  Then
$\widetilde{U(S,K)}$ is a hyperbolic metric space.  \label{eg} \end{prop}

 \begin{proof} Clearly, $\widetilde{U(S,K)}$ is a metric bundle over $K$ (since the universal curve over 
$Teich(S)$ is topologically a product $S \times Teich(S)$ and the latter is equipped with a foliation by
totally
geodesic copies of $Teich(S)$). Hence, by Theorem \ref{combthm} 
it suffices to prove flaring.  Let $S_x$ denote the fiber
of $U(S,K)$ over $x \in i(K)$.

Let $[a,b]$ be a geodesic segment
of sufficiently large length in $K$. By the hypothesis on $K$, there exists a bi-infinite geodesic passing within a bounded neighborhood
of $[a,b]$. Hence without loss of generality, we may assume that $a, b$ lie on a  
bi-infinite geodesic in $K$.

Since $j$ and $E\circ j$ are both quasi-isometric embeddings, it follows that there exists $\epsilon > 0$ such that
for all $x \in K$, the injectivity radius of $j(x) \in Teich(S)$ is greater than $\epsilon$. We shall refer to geodesics
lying in the $\epsilon -$ thick part of $ Teich(S)$ as fat Teichmuller geodesics.
By Theorem \ref{stability} we may assume that $j(a), j(b)$ lie in a uniformly bounded neighborhood of a fat Teichmuller geodesic $\eta_{ab}$ whose end-points in the Thurston boundary
$\partial Teich(S)$ are two singular foliations $\FF_+, \FF_-$. Let $d_s$ be the singular Euclidean metric on $S$ induced by the pair
of singular foliations $\FF_+, \FF_-$.

The rest of the argument follows an argument of Mosher \cite{mosher-hbh}. Let $x$ be some point
on the fat Teichmuller geodesic $\eta_{ab}$
 obtained in the previous paragraph. Given any geodesic segment $\lambda \subset \widetilde{S_x}$ of length $l(\lambda )$,
there are two projections $\lambda_+$ and $\lambda_-$ onto (the universal covers of) $\FF_+, \FF_-$ in $\widetilde{S_x}$. At least one of these
projections is of length at least $\frac{l(\lambda )}{2}$. If $u, v$ are two points on either side of $x$ such that $d_T (u, x) \geq m$
and $d_T (v, x) \geq m$, then the length of $\lambda$ in at least one of  $\widetilde{S_u}$ and  $\widetilde{S_v}$
is greater than $ \frac{l(\lambda )}{2}(e^m)$. 

Since all the surfaces with piecewise Euclidean metric involved in the above argument 
can be chosen to have uniformly bounded diameter, their universal
covers are all uniformly quasi-isometric to a fixed Cayley graph of $\pi_1(S)$. 
Flaring follows. \end{proof}

The same proof goes through for $S^h -$ a finite area surface with cusps, {\it provided we equip the cusps of $S^h$ with
a zero metric}.  (This is the electric metric on cusps in the terminology of \cite{farb-relhyp}.) Flaring in this situation
is proved in Section 4.4 of \cite{mahan-reeves}. The next proposition states this explicitly assuming that  $S^h$ (resp. the universal cover
$\widetilde{S^h}$) comes equipped with the zero metric on cusps (resp. lifts to $\widetilde{S^h}$).

\begin{prop} \label{egrmk}
Let $(K,d_K)$ be a hyperbolic metric space satisfying the following:\\
There exists $C > 0$ such that
for any two points $u, v\in K$, there exists a bi-infinite $C$-quasigeodesic $\gamma \subset K$ with $d_K(u, \gamma ) \leq C$
and $d_K(v, \gamma ) \leq C$. \\
Let $j: K \rightarrow (Teich(S^h), d_T)$ be a quasi-isometric embedding
such that $E\circ j: K \rightarrow (Teich(S^h), d_e)$ is also a quasi-isometric embedding.  Then
$\widetilde{U(S^h,K)}$ is a hyperbolic metric space.  \end{prop}

From Proposition
 \ref{egrmk} we obtain directly the following consequence (see \cite{farb-mosher} for definitions).

Consider a surface $S^h$ with finitely many
punctures. Let $K=\pi_1(S^h)$ and let $\KK = \{ K_1, \cdots , K_n\}$ be the
collection of peripheral subgroups. The pure mapping class group is the subgroup
of the mapping class group that preserves individual punctures. 

\begin{prop} \label{coco}
Let $K=\pi_1(S^h)$  be the fundamental group of a surface with finitely many punctures
and let $K_1, \cdots, K_n$ be its peripheral subgroups.  Let $Q$ be a convex cocompact subgroup of the  
{\em pure} mapping class group of $S^h$.
Let
\[
1\rightarrow K\rightarrow G\stackrel{p}{\rightarrow}Q\rightarrow 1
\]
and
\[
1\rightarrow K_i\rightarrow N_G(K_i) \stackrel{p}{\rightarrow}Q_i\rightarrow 1
\]
be the induced short exact sequences of   groups.
Then
$G$ is  strongly hyperbolic  relative to the collection $\{ N_G(K_i)\}, i=1, \cdots, n$. 

Conversely, if $G$ is  (strongly) hyperbolic  relative to the collection $\{ N_G(K_i)\}, i=1, \cdots, n$,
 then $Q$ is convex-cocompact.  \end{prop}

\begin{proof} Suppose that $Q$ is convex cocompact. 
Then $Q$ is hyperbolic by \cite{farb-mosher}, \cite{kl}. Also, by Theorem \ref{qi-pal},
 $Q=Q_i$ for all $i$. This is because $K$ is strongly hyperbolic relative to $K_i$
for {\it each} $i=1, \cdots, n$.
Let $\EE(G, K_1,\cdots, K_n)$ denote the electric space obtained from (the Cayley graph of)
$G$ after coning off translates of (the Cayley graphs of) $K_i$ for all $i=1, \cdots, n$. 
Then $\EE(G, K_1,\cdots, K_n)$
  is a metric graph bundle quasi-isometric to a $\widehat{K}(=\EE(K, K_1,\cdots, K_n))$-bundle over $Q$ where $\widehat{K}$
denotes $K$ with copies of $K_i$ coned off for all $i=1, \cdots, n$. .
The flaring condition for this bundle  and hence weak relative hyperbolicity of the pair $(G, 
\{K_1,\cdots, K_n\})$ follow from Proposition \ref{egrmk}. 

Let $\mathcal Q_i$ denote the collection of translates of (Cayley graphs of)
$Q(=N_G(K_i)/K_i)$ in
 $\EE(G, K_1,\cdots, K_n)$, where each copy of $Q$ in $\EE(G,K_1,\cdots, K_n)$
 is a copy of the electric space $\EE(N_G(K_i), K_i)$
obtained by coning off $K_i$ in  translates of (Cayley graphs of)
$N_G(K_i)$.
Let $\mathcal Q = \cup_i \mathcal Q_i$. 

To prove that $G$ is  strongly hyperbolic  relative to the $N_G(K_i)$'s,  it suffices to prove
that $\EE(G,K_1, \cdots, K_n)$
 is  strongly hyperbolic  relative to $\mathcal Q$, as $K$ is already strongly hyperbolic relative to $K_1, 
\cdots, K_n$
by \cite{farb-relhyp}.  That $\EE(G,K_1, \cdots, K_n)$
 is  strongly hyperbolic  relative to $\mathcal Q$ would in turn follow 
\cite{brahma-ibdd} from (uniform) mutual coboundedness of pairs of elements in $\mathcal Q$.
Note also that each $Q_j \in \mathcal{Q}$ is  quasi-isometrically embedded and hence a quasiconvex subset of $\EE
(G,K_1, \cdots, K_n)$.
Any two such $Q_1, Q_2$'s define a ladder $C(Q_1, Q_2)$ by regarding $Q_1, Q_2$ as qi sections of the 
metric graph bundle $\EE(G,K_1, \cdots, K_n)$ over $Q$.
Each
$C(Q_1, Q_2)$ is hyperbolic by Proposition \ref{main-prop}. 
Hence,  the ladder  $C(Q_1, Q_2)$ also satisfies flaring by Proposition \ref{necflaring}.

To establish mutual coboundedness, we argue by contradiction.  Let $d_h$ denote the horizontal
distance  in $\EE(G,K_1, \cdots, K_n)$.
Suppose that elements of the collection $\mathcal Q$ 
do not satisfy (uniform) mutual coboundedness.  
Then there exists 
$ D_0 > 0$  such that for any $l>0$, there exists a pair $Q_1, Q_2 \in \mathcal{Q}$ 
and a geodesic segment $r: [0,l]\rightarrow Q$ such that  $d_h(s_1\circ r (t), s_2\circ r (t))
\leq D_0$ for all $t \in [0,l]$,
  where $s_i:Q 
\rightarrow \EE(G,K_1, \cdots, K_n)$ are quasi-isometric embeddings defining the sections $Q_1, Q_2$. Since the number of elements in $K$ of length
at most $D_0$ is bounded it follows that 
for sufficiently large $l$, there exists $q \in Q, q \neq 1$ and a 
non-peripheral element $h \in K $ such that $s(q), h$ commute.
This is impossible for $Q$ convex cocompact, proving the forward direction of the Proposition.

We now prove the converse direction. Hyperbolicity of $Q$ follows from
Proposition \ref{mosher-genlzn}. To prove convex cocompactness,
it is enough to show by \cite{farb-mosher} that {\it some} orbit of the action of $Q$ on 
$(Teich(S), d_T)$ is quasiconvex.

Since $G$ is strongly hyperbolic relative to the collection
$\{ N_G(K_i)\}$, it follows from Lemma \ref{pel} that $\EE(G,K_1, \cdots, K_n)$ is strongly hyperbolic relative 
to the collection $\mathcal Q$ of translates of (Cayley graphs of)
$Q(=N_G(K_i)/K_i)$ in $\EE(G,K_1, \cdots, K_n)$ as defined above. Since $Q$ is hyperbolic, it follows 
(cf. \cite{bowditch-relhyp} Section 7) that $\EE(G,K_1, \cdots, K_n)$ is hyperbolic. Thus,  $\EE(G,K_1, \cdots, K_n)$
is a hyperbolic metric graph bundle over $Q$. Hence,  from  Proposition \ref{necflaring}, the bundle $\EE(G,K_1, \cdots, K_n)$ over $Q$ satisfies flaring.
The logarithm of the stretch factor
guaranteed by flaring gives a lower bound on the Teichmuller distance.

The remainder of the argument
 is an exact replica of the proof of Theorem 1.2 of \cite{farb-mosher} (Section 5.2 of \cite{farb-mosher} in particular),
which proves the analogous statement for surfaces without punctures.
We do not reproduce  the argument here but point out that the only place in the proof where
explicit use is made of closedness of $S$ is Theorem 5.5 of \cite{farb-mosher}, which, in turn is taken from
 \cite{mosher-gt03}. A straightforward generalization of this fact to  the punctured surface case 
is given in \cite{ap}.
\end{proof}

\begin{rmk} It is worth noting that a group $G$ as in Proposition \ref{coco} cannot act freely, properly discontinuously by isometries
on a Hadamard manifold of pinched negative curvature unless $Q$ is virtually cyclic, as the normalizer $N_G(K_i)$ is not nilpotent.
\end{rmk}

As an application of Proposition \ref{eg} we give the first examples of surface bundles over hyperbolic disks, with
 Gromov-hyperbolic universal cover.  It has been an open question (cf. \cite{kl} \cite{farb-mosher}) to find purely pseudo Anosov surface groups in $MCG(S)$.
The example below is a step towards this.

In \cite{ls-disk} Leininger and Schleimer construct examples of disks $(Q, d_Q)$ quasi-isometric to ${\mathbb{H}}^2$
and quasi-isometric embeddings $j: Q \rightarrow (Teich(S), d_T)$ 
such that $E\circ j: Q \rightarrow (Teich(S), d_e)$ is also a quasi-isometric embedding. 
By Proposition \ref{eg}, the hyperbolic plane bundle
$\widetilde{U(S,Q)}$ is a hyperbolic metric space.

A brief sketch of Leininger-Schleimer's construction \cite{ls-disk} follows:\\
The curve complex $CC(S,x)$ of a surface with one puncture (or
equivalently, a marked point $x$) admits a surjective map to $CC(S)$
such that the fiber over $\eta \in CC(S)$ is the Bass-Serre tree of the
splitting of $\pi_1(S)$ over the cyclic groups represented by the simple
closed curves in $\eta$.  Suppose $\gamma$ is a bi-infinite geodesic in
$CC(S)$ coming from a geodesic in $Teich(S)$ lying in the thick part.
Inside $CC(S,x)$ one has the space of trees over $\gamma$, and the authors of 
\cite{ls-disk} construct lines in each tree over $\gamma$ whose union $Q_1$
is quasi-isometric to the hyperbolic plane.  Using a branched cover-trick,
they construct from $Q_1$ a new disk $Q \subset
CC(S^{\prime})$ (for a closed surface $S^{\prime}$, which is
a branched cover of $S$ branched over the marked point of $S$) such that $Q$
satisfies the hypotheses of Proposition \ref{eg}.


\bibliography{mbdl_feb12}
\bibliographystyle{amsalpha}

\end{document}